\documentclass[a4paper,12pt]{article}
\usepackage[english]{babel}
\usepackage[top=3cm, bottom=3cm, left=2.3cm, right=2.3cm]{geometry}

\usepackage{microtype}

\usepackage{amssymb,amsthm,enumerate,xcolor,latexsym,dsfont}
\usepackage{bbm,mathrsfs}
\usepackage{amsmath}

\usepackage[pagebackref, draft=false]{hyperref}

\usepackage{color}

\newcommand{\R}{\mathbb{R}}
\newcommand{\Z}{\mathbb{Z}}
\newcommand{\T}{\mathbb{T}}
\newcommand{\N}{\mathbb{N}}
\newcommand{\E}{\mathbb{E}}
\renewcommand{\P}{\mathbb{P}}
\newcommand{\1}{\mathds{1}}
\newcommand{\dd}{\mathrm{d}}
\newcommand{\LL}{\mathcal{L}}
\newcommand{\GG}{\mathcal{G}}
\newcommand{\NN}{\mathcal{N}}
\newcommand{\KK}{\mathcal{K}}
\newcommand{\CL}{\mathcal{L}}
\newcommand{\CK}{\mathcal{U}}
\newcommand{\CC}{\mathcal{C}}
\newcommand{\CF}{\mathcal{F}}
\newcommand{\CD}{\mathcal{D}}
\newcommand{\CS}{\mathcal{S}}
\newcommand{\CE}{\mathcal{E}}

\newcommand{\blue}[1]{{\color{blue} #1}}

\newtheorem{theorem}{Theorem}[section]
\newtheorem{corollary}[theorem]{Corollary}
\newtheorem{definition}[theorem]{Definition}
\newtheorem{lemma}[theorem]{Lemma}
\newtheorem{remark}[theorem]{Remark}
\newtheorem{proposition}[theorem]{Proposition}
\newtheorem*{convention}{Convention}

\newcommand{\assign}{:=}
\newcommand{\backassign}{=:}
\newcommand{\tmop}[1]{\ensuremath{\operatorname{#1}}}

\begin{document}

\title{The infinitesimal generator of\\ the stochastic Burgers equation}
\author{
  Massimiliano Gubinelli\thanks{Financial support by DFG via the CRC 1060 and
partially by EPSRC Grant Number EP/R014604/1 is gratefully acknowledged.}\\
  Hausdorff Center for Mathematics\\
   \& Institute for Applied Mathematics, Universit{\"a}t Bonn \\
  \texttt{gubinelli@iam.uni-bonn.de}
  \and
  Nicolas Perkowski\thanks{Financial support by DFG via the Heisenberg program and via Research Unit FOR 2402 is gratefully acknowledged.} \\
  Max-Planck-Institute for Mathematics in the Sciences, Leipzig \\
  \& Humboldt--Universit\"at zu Berlin \\
  \texttt{nicolas.perkowski@mis.mpg.de}
}

\maketitle

\begin{abstract}
	We develop a martingale approach for a class of singular stochastic PDEs of Burgers type (including fractional and multi-component Burgers equations) by constructing a domain for their infinitesimal generators. It was known that the domain must have trivial intersection with the usual cylinder test functions, and to overcome this difficulty we import some ideas from paracontrolled distributions to an infinite dimensional setting in order to construct a domain of \emph{controlled functions}. Using the new domain, we are able to prove existence and uniqueness for the Kolmogorov backward equation and the martingale problem. We also extend the uniqueness result for ``energy solutions'' of the stochastic Burgers equation of~\cite{Gubinelli2018Energy} to a wider class of equations.
\end{abstract}

\section{Introduction}
The (conservative) stochastic Burgers equation $u\colon \R_+ \times \T \to \R$ (or $u \colon \R_+ \times \R \to \R$)
\begin{equation}\label{eq:burgers}
	\partial_t u = \Delta u + \partial_x u^2 + \sqrt{2}\partial_x \xi,
\end{equation}
where $\xi$ is a space-time white noise, is one of the most prominent \emph{singular stochastic PDEs}, a class of equations that are ill posed due to the interplay of very irregular noise and nonlinearities. The difficulty is that $u$ only has only distributional regularity (under the stationary measure it is a white noise in space for all times), and therefore the meaning of the nonlinearity $\partial_x u^2$ is dubious.

In recent years, new solution theories like regularity structures~\cite{Hairer2014, Friz2014} or para\-controlled distributions~\cite{Gubinelli2015Paracontrolled, Gubinelli2017KPZ} were developed for singular SPDEs, see \cite{Gubinelli2018Panorama} for an up-to-date and fairly exhaustive review. These theories are based on analytic (as opposed to probabilistic) tools. In the example of the stochastic Burgers equation we roughly speaking use that $u$ is not a generic distribution, but it is a local perturbation of a Gaussian (obtained from $\xi$). We construct the nonlinearity and some higher order terms of the Gaussian by explicit computation, and then we freeze the realization of $\xi$ and of the nonlinear terms we just constructed and use pathwise and analytic tools to control the nonlinearity for the (better behaved) remainder. This requires the introduction of new function spaces of \emph{modelled} (resp. \emph{paracontrolled}) distributions, which are exactly those distributions that are given as local perturbations as described before, and for which the nonlinearity can be constructed.

This point of view was first developed for rough paths, which provide a pathwise solution theory for SDEs by writing the solutions as local perturbations of the Brownian motion~\cite{Lyons1998, Gubinelli2004}. Rough paths provide a new topology in which the solution depends continuously on the driving noise, and this is useful in a range of applications. But of course there are also probabilistic solution theories for SDEs, based for example on It\^o or Stratonovich integration (strong solutions) or on the martingale problem (weak solutions), and depending on the aim it may be easier to work with the pathwise approach or with the probabilistic one.

For singular SPDEs the situation is somewhat unsatisfactory because while the pathwise approach applies to a wide range of equations, it seems completely unclear how to set up a general probabilistic solution theory. There are some exceptions, for example martingale techniques tend to work in the ``not-so-singular'' case when the equation is singular but can be handled via a simple change of variables and does not require regularity structures (sometimes this is called the \emph{Da Prato-Debussche regime}~\cite{DaPrato2003, DaPrato2002}); see~\cite{Stannat2007, Rockner2017} and also~\cite{Flandoli2018Kolmogorov, Flandoli2018Convergence} for a an example where the change of variable trick does not work but still the equation is not too singular. For truly singular equations there exist only very few probabilistic results. R. and X. Zhu constructed a Dirichlet form for the $\Phi^4_3$ equation and used the pathwise results to show that the form is closable~\cite{Zhu2017}, but it is unclear if the process corresponding to this form is the same as the one that is constructed via regularity structures or even if it is unique.

Maybe the strongest probabilistic results exist for the stochastic Burgers equation~\eqref{eq:burgers}: First results, on which we comment more below, are due to Assing~\cite{Assing2002}. In~\cite{Goncalves2014} Gon\c{c}alves and Jara construct so called \emph{energy solutions} to Burgers equation, roughly speaking by requiring that $u$ solves the martingale problem associated to
\[
	\partial_t u = \Delta u + \lim_{\varepsilon \to 0} \partial_x (u\ast \rho^\varepsilon)^2 + \sqrt{2} \partial_x \xi,
\]
where $\rho^\varepsilon$ is an approximation of the identity. This notion of solution is refined in~\cite{Gubinelli2013} where the authors additionally impose a structural condition for the time-reversed process $(u_{T-t})_{t \in [0,T]}$, and they assume that $u$ is stationary. These two assumptions allow them to derive strong estimates for additive functionals $\int_0^\cdot F(u_s) \dd s$ of $u$ via the \emph{It\^o trick}. They obtain the existence of solutions in this stronger sense by Galerkin approximation. The uniqueness of the refined solutions is shown in~\cite{Gubinelli2018Energy}, leading to the first probabilistic well-posedness result for a truly singular SPDE. Extensions to non-stationary initial conditions that are absolutely continuous with respect to the invariant measure are given in~\cite{Goncalves2015, Gubinelli2018Probabilistic}, and in~\cite{Yang2018} some singular initial conditions are considered; see also~\cite{Goncalves2017} for Burgers equation with Dirichlet boundary condition.

The reason why the uniqueness proofs work is that we can linearize the equation via the \emph{Cole-Hopf transform}: By formally applying It\^o's formula, we get $u = \partial_x \log w$, where $w$ solves the stochastic heat equation $\partial_t w = \Delta w + \sqrt{2} w \xi$, a well posed equation which can be handled with classical SPDE approaches as in~\cite{Walsh1986, DaPrato2014, Liu2015}. The proof of uniqueness in~\cite{Gubinelli2018Energy} shows that the formal application of It\^o's formula is allowed for the refined energy solutions of~\cite{Gubinelli2013}, and it heavily uses the good control of additive functionals from the It\^o trick. Since the Cole-Hopf transform breaks down for essentially all other singular SPDEs, there is no hope of extending this approach to other equations.

The aim of the present paper is to provide a new and intrinsic (without transformation) martingale approach to some singular SPDEs. For simplicity we lead the main argumentation on the example of the Burgers equation, but later we also treat multi-component and fractional generalizations. The starting point is the observation that $u$ is a Markov process, and therefore it must have an infinitesimal generator. The problem is that typical test functions on the state space of $u$ (the space of Schwartz distributions) are not in the domain of the generator; this includes the test functions that are used in the energy solution approach, where the term
\[
	\lim_{\varepsilon \to 0} \int_0^t [\partial_x (u_s\ast \rho^\varepsilon)^2] (f) \dd s
\]
for a test function $f$ is not of finite variation, which means that for $\varphi(u) = u(f)$ the process $(\varphi(u_t))_t$ is not a semimartingale, and therefore $\varphi$ cannot be in the domain of the generator. This was already noted by Assing~\cite{Assing2002}, who defined the formal generator on \emph{cylinder test functions} but with image in the space of \emph{Hida distributions}. Our aim is to find a (more complicated) domain of functions that are mapped to functions and not distributions under a formal extension of Assing's operator.

For this purpose we take inspiration from recent developments in \emph{singular diffusions}, i.e. diffusions with distributional drift. Indeed, Assing's results show that we can interpret the Burgers drift as a distribution in an infinite-dimensional space, see also the discussion in~\cite{Gubinelli2018Probabilistic}. In finite-dimensions the papers~\cite{Flandoli2003, Flandoli2004, Delarue2016, Cannizzaro2018} all follow a similar strategy for solving $\dd X_t = b(X_t) \dd t + \dd W_t$ for distributional $b$: They identify a domain for the formal infinitesimal generator $\CL = \tfrac12 \Delta + b \cdot \nabla$ and then show existence and uniqueness of solutions for the corresponding martingale problem. So far this is very classical, but the key observation is that for distributional $b$ the domain does not contain any smooth functions and instead one has to identify a class of non-smooth test functions with a special structure, adapted to $b$. Roughly speaking they must be local perturbations of a linear functional constructed from $b$. This is very reminiscent of the rough path/regularity structure philosophy, and in fact~\cite{Delarue2016, Cannizzaro2018} even use tools from rough paths resp. paracontrolled distributions.

We would like to use the same strategy for the stochastic Burgers equation. But rough paths and controlled distributions are finite-dimensional theories, and here we are in an infinite-dimensional setting. To set up a theory of function spaces and distributions we need a reference measure (in finite dimensions this is typically Lebesgue measure), and we will work with the stationary measure of $u$, the law $\mu$ of the white noise. This is a Gaussian measure, and by the chaos decomposition we can identify $L^2(\mu)$ with the Fock space $\bigoplus_{n=0}^\infty L^2(\T^n)$, which has enough structure so that we can do analysis on it. In that way we construct a domain of \emph{controlled functions} which are mapped to $L^2(\mu)$ by the generator of $u$, and this allows us to define a martingale problem for $u$. By Galerkin approximation we easily obtain the existence of solutions to the martingale problem. To see uniqueness, we use the duality with the Kolmogorov backward equation: Existence for the backward equation yields uniqueness for the martingale problem, and existence for the martingale problem yields uniqueness for the backward equation. We construct solutions to the backward equation by a compactness argument, relying on energy estimates in spaces of controlled functions. In that way we obtain a self-contained probabilistic solution theory for Burgers equation and fractional and multi-component generalizations. As a simple application we obtain the exponential $L^2$-ergodicity of $u$. This program is somewhat related to the recent advances in regularization by noise for SPDEs~\cite{DaPrato2013, DaPrato2016}, where unique strong solutions for SPDEs with bounded measurable drift are constructed by solving infinite-dimensional resolvent type equations. Of course our drift is strongly unbounded (and not even a function).

Finally we study the connection of our new approach with the Gon\c{c}alves-Jara energy solutions. One of the main motivations for studying the martingale problem for singular SPDEs is that it is a convenient tool for deriving the equations as scaling limits: The \emph{weak KPZ universality conjecture}~\cite{Quastel2011, Corwin2012, Quastel2015} says that a wide range of interface growth models converge in the weakly asymmetric or the weak noise regime to the Kardar-Parisi-Zhang (KPZ) equation $h$, for which $u = \partial_x h$. Energy solutions are a powerful tool for proving this convergence, see e.g.~\cite{Goncalves2014, Goncalves2015, Franco2016, Diehl2017, Gubinelli2016Hairer}. For that purpose it is crucial to work with nice test functions, and since there seems to be no easy way of identifying the complicated functions in the domain of the generator of $u$ with test functions on the state space of a given particle system, our new martingale problem is probably not so useful for deriving convergence theorems. This motivates us to show that the notion of energy solution is in fact stronger than our martingale problem: Every energy solution solves the martingale problem for our generator, and thus it is unique in law.

All this also works for the fractional and multi-component Burgers equations. For the fractional Burgers equation we treat the entire \emph{locally subcritical} regime (in the language of Hairer~\cite{Hairer2014}), which in regularity structures would lead to very complicated expansions, while for us a first order expansion is sufficient. Although by now there are very sophisticated and powerful black box type tools available in regularity structures that should handle the complicated expansion automatically~\cite{Bruned2016, Chandra2016, Bruned2017Renormalising}.

The lynchpin of our approach is the Gaussian invariant measure $\mu$, and in principle our methods should extend to other equations with Gaussian invariant measures, like the singular stochastic Navier Stokes equations studied in~\cite{Gubinelli2013}. It would even suffice to have a Gaussian quasi-invariant measure, i.e. a process which stays absolutely continuous (or rather \emph{incompressible} in the sense of Definition~\ref{def:incompressible}) with respect to a Gaussian reference measure. But for general singular SPDEs we would have to work with more complicated measures like the $\Phi^4_3$ measure for which we cannot reduce the analysis to the Fock space. Currently it is not clear how to extend our methods to such problems, so while we provide a probabilistic theory of some singular SPDEs that actually tackles the problem at hand and does not shift the singularity away via the Cole-Hopf transform, it is still much less general than regularity structures and it remains an important and challenging open problem to find more general probabilistic methods for singular SPDEs.

\paragraph{Structure of the paper} Below we introduce some commonly used notation. In Section~\ref{sec:domain} we derive the explicit representation of the Burgers generator on Fock space and we introduce a space of controlled functions which are in the domain of the generator. In Section~\ref{sec:bw-eq} we study the Kolmogorov backward equation and show the existence of solutions with the help of energy estimates for the Galerkin approximation and a compactness principle in controlled spaces, while uniqueness is easy. Section~\ref{sec:mp} is devoted to the martingale problem: We show existence via tightness of the Galerkin approximations and uniqueness via duality with the backward equation. As an application of our results we give a short proof of exponential $L^2$-ergodicity. Finally we formulate a cylinder function martingale problem in the spirit of energy solutions, and we show that it is stronger than the martingale problem and therefore also has unique solutions. In Section~\ref{sec:extensions} we briefly discuss extensions to multi-component and fractional Burgers equations. We do all the analysis on the torus, but with minor changes it carries over to the real line, as we explain in Section~\ref{sec:full-space}. The appendix collects some auxiliary estimates.

\paragraph{Acknowledgments}
The authors would like to thank the Isaac Newton Institute for Mathematical Sciences
for support and hospitality during the programme SRQ: Scaling limits, Rough
paths, Quantum field theory when part of the work on this paper was
undertaken. 

\paragraph{Notation}

We work on the torus $\T = \R / \Z$ and the Fourier transform of $\varphi \in L^2(\T^n)$ is
\[
	\CF \varphi(k_1,\dots, k_n) = \hat \varphi (k_1,\dots, k_n) = \int_{\T^n} e^{-2\pi \iota k\cdot x} \varphi(x) \dd x,\qquad k \in \Z^n.
\]
To shorten the formulas we usually write
\[
	k_{1:n} \assign (k_1, \dots, k_n),\qquad x_{1:n} \assign (x_1, \dots, x_n)
\]
and
\[
	\int_x (\cdots) \assign \int (\cdots) \dd x
\]
Moreover, we set $\Z_0 \assign \Z \setminus\{0\}$ and we mostly restrict our attention to the subspace
\[
	L^2_0(\T^n) \assign \{\varphi \in L^2(\T^n): \hat{\varphi}(k_{1:n}) = 0 \ \forall k \in \Z^n \setminus \Z_0^n\}.
\]
The space $C^k_p (\R^n)$ consists of all $C^k$ functions whose partial derivatives of order up to $k$ have polynomial growth.

We write $a \lesssim b$ or $b \gtrsim a$ if there exists a constant $c > 0$, independent
of the variables under consideration, such that $a \leqslant c \cdot b$, and we write $a \simeq b$ if $a \lesssim b$ and $b \lesssim a$.

\section{A domain for the Burgers generator}\label{sec:domain}

\subsection{The generator of the Galerkin approximation}

Consider the solution $u^m \colon \R_+ \times \T \to \R$ to the Galerkin approximation of the conservative
stochastic Burgers equation
\begin{equation}
  \label{eq:burgers-galerkin} \partial_t u^m = \Delta u^m + B_m (u^m) +
  \sqrt{2} \partial_x \xi \assign \Delta u^m + \partial_x \Pi_m (\Pi_m u^m)^2
  + \sqrt{2} \partial_x \xi,
\end{equation}
where $\xi$ is a space-time white noise and
\[ \Pi_m u (x) = \sum_{| k | \leqslant m} e^{2 \pi \iota k x} \hat{u} (k) \]
is the projection onto the first $2 m + 1$ Fourier modes. Throughout the paper
we write $\mu$ for the law of the average zero white noise on $\T$,
i.e. the centered Gaussian measure on $H^{-1/2-}(\T) \assign \bigcup_{\varepsilon > 0} H^{- 1 / 2 - \varepsilon} (\T)$ with
covariance
\[ \int u (f) u (g) \mu (\dd u) = \langle f - \hat{f} (0), g - \hat{g} (0)
   \rangle_{L^2 (\T)} \]
for all $f, g \in \bigcup_{\varepsilon > 0} H^{1 / 2 + \varepsilon}
(\T)$.

\begin{lemma}
  Equation~(\ref{eq:burgers-galerkin}) has a unique strong solution $u \in C
  (\R_+, H^{- 1 / 2 -} (\T))$ for every deterministic
  initial condition in $H^{- 1 / 2 -} (\T)$. The solution is a strong
  Markov process and it is invariant under $\mu$. Moreover, for all $\alpha >
  1 / 2$ there exists $C = C (m, t, p, \alpha) > 0$ such that
  \[ \E [\sup_{s \in [0, t]} \| u^m_s \|_{H^{- \alpha}}^p] \leqslant
     C (1 + \| u_0^m \|^p_{H^{- \alpha}}) . \]
\end{lemma}

\begin{proof}
  Local existence and uniqueness and the strong Markov property follow from
  standard theory because written in Fourier coordinates we can decouple $u^m = v^m + Z^m \assign \Pi_m u^m + (1 - \Pi_m) u^m$, 
  where $v^m$ solves a finite-dimensional SDE with locally Lipschitz
  continuous coefficients and $Z^m$ solves an infinite-dimensional but linear
  SDE. Global existence and invariance of $\mu$ are shown in Section~4 of
  \cite{Gubinelli2013}. It is well known and easy to check that $Z^m$ has trajectories in $C(\R_+,H^{-1/2-}(\T))$, see
  e.g. \cite[Chapter~2.3]{Gubinelli2015EBP}, and $v^m$ has compact spectral support and
  therefore even $v^m \in C (\R_+, C^{\infty} (\T))$. Thus $u^m$ has trajectories in $C (\R_+, H^{- 1 / 2 -} (\T))$. The moment bound can be derived using similar
  arguments as in~\cite{Gubinelli2013}. The reason why $v^m$ behaves nicely is that
  $B_m$ leaves the $L^2 (\T)$ norm invariant since
  \[ \langle u, B_m (u) \rangle_{L^2 (\T)} = - \langle \partial_x
     \Pi_m u, (\Pi_m u)^2 \rangle_{L^2 (\T)} = - \frac{1}{3} \langle
     \partial_x (\Pi_m u)^3, 1 \rangle_{L^2 (\T)} = 0 \]
  by the periodic boundary conditions. To see the invariance of $\mu$ we also
  need that $B_m$ is divergence free when written in Fourier coordinates. See
  Section~4 of \cite{Gubinelli2013} or Lemma 5 of \cite{Gubinelli2016Hairer} for details.
\end{proof}

We define the semigroup of $u^m$ for all bounded and measurable $\varphi\colon
H^{- 1 / 2 -} \rightarrow \R$ as $T_t^m \varphi (u) \assign \E_u [\varphi (u^m_t)]$, where under $\P_u$ the process $u^m$ solves
(\ref{eq:burgers-galerkin}) with initial condition $u$.

\begin{lemma}
  \label{lem:semigroup}For all $p \in [1, \infty]$ the family of operators
  $(T^m_t)_{t \geqslant 0}$ can be uniquely extended to a contraction
  semigroup on $L^p (\mu)$, which is continuous for $p \in [1, \infty)$.
\end{lemma}

\begin{proof}
  This uses the invariance of $\mu$ and follows by approximating $L^p$ functions with bounded measurable functions. To see the continuity for $p\in [1,\infty)$ we use that in this case continuous bounded functions are dense in $L^p(\mu)$.
\end{proof}

Our next aim is to derive the generator of the semigroup $T^m$ on $L^2 (\mu)$.
For that purpose let $f_1, \ldots, f_n \in C^{\infty} (\T)$, let
$\Phi \in C^2_p (\R^n, \R)$, the $C^2$ functions with polynomially growing partial derivatives of order up to $2$, and let $\varphi \in
\CC$ be a \emph{cylinder function} of the form $\varphi (u) = \Phi
(u (f_1), \ldots, u (f_n))$. Let us introduce the notation
\begin{gather*}
	\LL_0 \varphi (u) \assign \sum_{i = 1}^n \partial_i \Phi (u (f_1),
   \ldots, u (f_n)) u (\Delta f_i) + \sum_{i, j = 1}^n \partial_{i j}^2 \Phi
   (u (f_1), \ldots, u (f_n)) \langle \partial_x f_i, \partial_x f_j
   \rangle_{L^2 (\T)}, \\
	\GG^m \varphi (u) \assign \sum_{i = 1}^n \partial_i \Phi (u (f_1),
   \ldots, u (f_n)) \langle B_m (u), f_i \rangle_{L^2 (\T)} =
   \int_{\T} B_m (u) (x) D_x \Phi (u) \dd x,
\end{gather*}
where $D_x$ is the Malliavian derivative, and
\[ \LL^m \assign \LL_0 +\GG^m . \]
Then It{\^o}'s formula gives
\[ \dd \varphi (u^m_t) =\LL^m \varphi (u^m_t) \dd t + \sum_{i =
   1}^n \partial_i \Phi (u^m_t (f_1), \ldots, u^m_t (f_n)) \dd M_t (f_i),
\]
where $M (f_i)$ is a continuous martingale under $\P_u$, with
quadratic variation $\langle M (f_i) \rangle_t = 2 \| \partial_x f_i \|_{L^2
(\T)}^2 t$ and therefore $\int_0^{\cdot} \sum_{i = 1}^n \partial_i
\Phi (u^m_t (f_1), \ldots, u^m_t (f_n)) \dd M_t (f_i)$ is a martingale
under $\P_u$. Consequently, we have
\[ T^m_t \varphi (u) - \varphi (u) = \int_0^t T^m_s (\LL^m \varphi)
   (u) \dd s \]
for all $u \in H^{- 1 / 2 -}$.

To extend this to more general functions $\varphi$ and to obtain suitable
bounds for $\LL_0$ and $\GG^m$ we work with the chaos
expansion: Every function $\varphi \in L^2 (\mu)$ can be written uniquely as $\varphi = \sum_{n \geqslant 0} W_n (\varphi_n)$,
where $\varphi_n \in L^2_0 (\T^n)$ is symmetric in its $n$ arguments
and $W_n$ is an $n$-th order Wiener-It{\^o} integral; here $L^2_0
(\T^n) = \{ \varphi \in L^2 (\T^n) : \hat{\varphi} (k) = 0
\forall k \in \Z^n \setminus \Z^n_0 \}$. Moreover, we have
\[ \| \varphi \|_{L^2 (\mu)}^2 = \sum_{n \geqslant 0} n! \| \varphi_n \|_{L^2
   (\T^n)}^2, \]
see \cite{Nualart2006, Janson1997} for details. If $\varphi_n \in L^2_0 (\T^n)$ is
not necessarily symmetric, then we define $W_n (\varphi_n) \assign W_n
(\widetilde{\varphi}_n)$, where $\widetilde{\varphi}_n (x_1, \ldots, x_n) = \tfrac{1}{n!} \sum_{\sigma \in
   \Sigma_n} \varphi_n ( x_{\sigma(1)}, \ldots,  x_{\sigma(n)})$ for the symmetric group $\Sigma_n$ is the symmetrization of $\varphi_n$. By the triangle inequality we have $\| \widetilde{\varphi}_n \|_{L^2
(\T^n)} \leqslant \| \varphi_n \|_{L^2 (\T^n)}$.

\begin{convention} In the following a norm $\| \cdot \|$ without subscript always denotes the
  $L^2 (\mu)$ norm, and an inner product $\langle \cdot, \cdot \rangle$
  without subscript denotes the $L^2 (\mu)$ inner product.
\end{convention}

\begin{lemma}
  Let $\varphi \in \CC$ with chaos expansion $\varphi = \sum_{n
  \geqslant 0} W_n (\varphi_n)$. Then
  \[ \LL_0 \varphi = \sum_{n \geqslant 0} W_n (\Delta \varphi_n)
     \assign \sum_{n \geqslant 0} W_n ((\partial^2_{11} + \cdots +
     \partial^2_{n n}) \varphi_n) . \]
\end{lemma}

\begin{proof}
  The proof is the same as for \cite[Lemma 3.7]{Gubinelli2018Energy}.
\end{proof}

\begin{lemma}
  \label{lem:Gm}Let $\varphi \in \CC$ have the chaos expansion
  $\varphi = \sum_{n \geqslant 0} W_n (\varphi_n)$. Then $\GG^m
  =\GG^m_+ +\GG^m_-$, and writing $\rho^m$ for the inverse
  Fourier transform of $\1_{| \cdot | \leqslant m}$ and $f_x \assign f(x - \cdot)$ we have
  \begin{gather}
    \GG^m_+ W_n (\varphi_n) = n W_{n + 1} \left( \int_{x, s}
    \partial_x \rho^m_x (s) (\rho^m_s \otimes \rho^m_s) (\cdot) \varphi_n
    (x, \cdot) \right),\\
    \GG^m_- W_n (\varphi_n) = 2 n (n - 1) W_{n - 1} \left( \int_{x, y,
    s} \partial_x \rho^m_x (s) \rho^m_s (y) \rho^m_s (\cdot) \varphi_n (x,
    y, \cdot) \right),
  \end{gather}
  and moreover we have for all $\varphi_{n + 1} \in L^2_0 (\T^{n +
  1})$ and $\varphi_n \in L^2_0 (\T^n)$
  \[ \langle W_{n + 1} (\varphi_{n + 1}), \GG^m_+ W_n (\varphi_n)
     \rangle = - \langle \GG^m_- W_{n + 1} (\varphi_{n + 1}), W_n
     (\varphi_n) \rangle. \]
\end{lemma}

\begin{proof}
  Since $\| \rho^m_s \|_{L^2 (\T)}^2 = \| \rho^m \|_{L^2
  (\T)}^2$ does not depend on $s$ and thus vanishes under
  differentiation, we have
  \begin{align*}
    B_m (u) (x) & = W_2 \left( \int \partial_x \rho^m_x (s) \rho^m_s \otimes
    \rho^m_s \dd s \right) + \int \partial_x \rho^m_x (s) \| \rho^m_s
    \|_{L^2 (\T)}^2 \dd s\\
    & = W_2 \left( \int \partial_x \rho^m_x (s) \rho^m_s \otimes \rho^m_s
    \dd s \right)
  \end{align*}
  and then, since $D_x W_n(\varphi_n) = n W_{n-1}(\varphi_n(x,\cdot))$ \cite[Proposition~1.2.7]{Nualart2006} and by the contraction rules for Wiener-It\^o integrals~\cite[Proposition~1.1.3]{Nualart2006},
  \begin{align*}
    \int_x B_m (u) (x) D_x W_n (\varphi_n) & = n \int_x W_2 \left( \int_s
    \partial_x \rho^m_x (s) (\rho^m_s)^{\otimes 2} (\cdot) \right) W_{n - 1}
    (\varphi_n (x, \cdot))\\
    & = n W_{n + 1} \left( \int_{x, s} \partial_x \rho^m_x (s)
    (\rho^m_s)^{\otimes 2} (\cdot) \varphi_n (x, \cdot) \right)\\
    & \quad + 2 n (n - 1) W_{n - 1} \left( \int_{x, y, s} \partial_x \rho^m_x
    (s) \rho^m_s (y) \rho^m_s (\cdot) \varphi_n (x, y, \cdot) \right)\\
    & \quad + n (n - 1) (n - 2) W_{n - 3} \left( \int_{x, y, z, s} \partial_x
    \rho^m_x (s) \rho^m_s (y) \rho^m_s (z) \varphi_n (x, y, z, \cdot)
    \right) .
  \end{align*}
  Let us look more carefully at the last term on the right hand side. Note
  that $\partial_x \rho^m_x (s) = -\partial_s \rho^m_s (x)$ and $\varphi_n$ is
  symmetric under exchange of its arguments. Therefore, by symmetrisation,
  \begin{align*}
    & \int_{x, y, z, s} \partial_x \rho^m_x (s) \rho^m_s (y) \rho^m_s (z)
    \varphi_n (x, y, z, \cdot)\\
    & \hspace{40pt} =  \int_{x, y, z, s}  (-\partial_s \rho^m_s (x)) \rho^m_s (y) \rho^m_s (z) \varphi_n (x, y, z,
    \cdot)\\
    & \hspace{40pt} = -\frac13 \int_{x, y, z, s} \partial_s (\rho^m_s (x) \rho^m_s (y) \rho^m_s (z)) \varphi_n (x, y, z, \cdot) = 0
  \end{align*}
  since now $\partial_s$ can be integrated by parts. We deduce that the last
  term in the decomposition of $\int_x B_m (u) (x) D_x W_n (\varphi_n)$
  vanishes.
  
  It remains to show that $-\GG_m^+$ is the adjoint of
  $\GG_m^-$: Since $\varphi_{n + 1}$ is symmetric in its $(n + 1)$
  arguments, we have $\langle \varphi_{n + 1}, \psi \rangle_{L^2 (\T^{n + 1})} =
     \langle \varphi_{n + 1}, \tilde{\psi} \rangle_{L^2 (\T^{n + 1})}$ for all $\psi$, where $\tilde{\psi}$ is the symmetrization of $\psi$, and therefore we do not need to symmetrize the kernel of $\GG^m_+ W_n (\varphi_n)$ in the following computations:
  \begin{align*}
    & \langle W_{n + 1} (\varphi_{n + 1}), \GG^m_+ W_n (\varphi_n)
    \rangle\\
    & = (n + 1) ! \int_{r_{1 : n + 1}} \varphi_{n + 1} (r_{1 : n + 1}) n
    \int_{x, s} \partial_x \rho^m_x (s) \rho^m_s (r_1) \rho^m_s (r_2)
    \varphi_n (x, r_{3 : n + 1})\\
    & = (n + 1) ! \int_{r_{1 : n + 1}} \varphi_{n + 1} (r_{1 : n + 1}) n
    \int_{x, s} \rho^m_x (s) 2 \partial_s \rho^m_s (r_1) \rho^m_s (r_2)
    \varphi_n (x, r_{3 : n + 1})\\
    & = n! 2 (n + 1) n \int_{r_{1 : n + 1, x, s}} \varphi_{n + 1} (r_{1 : n +
    1}) \rho^m_x (s) \partial_s \rho^m_s (r_1) \rho^m_s (r_2) \varphi_n (x,
    r_{3 : n + 1})\\
    & = n! 2 (n + 1) n \int_{r_{1 : n, x, y, s}} \varphi_{n + 1} (x, y, r_{2
    : n}) \rho^m_{r_1} (s) \partial_s \rho^m_s (x) \rho^m_s (y) \varphi_n
    (r_{1 : n}),
  \end{align*}
  where in the last step we renamed the variables as follows: $r_1
  \leftrightarrow x$, $r_2 \rightarrow y$, $r_i \rightarrow r_{i - 1}$ for $i
  \geqslant 3$. The claim now follows by noting that $\rho^m_{r_1} (s) =
  \rho^m_s (r_1)$ and $\partial_s \rho^m_s (x) = - \partial_x \rho^m_x (s)$,
  and thus
  \begin{align*}
    \langle W_{n + 1} (\varphi_{n + 1}), \GG^m_+ W_n (\varphi_n)
    \rangle & = - n! 2 (n + 1) n \int_{r_{1 : n}} \int_{x, y, s} \partial_x
    \rho^m_x (s) \rho^m_s (y) \rho^m_s (r_1) \varphi_{n + 1} (x, y, r_{2 : n})
    \varphi_n (r_{1 : n})\\
    & = - \langle \GG^m_- W_{n + 1} (\varphi_{n + 1}), W_n
    (\varphi_n) \rangle .
  \end{align*}
\end{proof}

\begin{remark}
Note that the proof did not use the specific form of $\rho^m$ and the same arguments
  work as long as $\rho^m$ is an even function.
\end{remark}

For $m \rightarrow \infty$ the kernel for
$\GG_-^m W_n (\varphi_n)$ formally converges to
\[ \int_{x, y} \partial_x (\delta_x (y) \delta_x (r_1)) \varphi_n (x, y, r_{2
   : n - 1}) = - \int_{x, y} \delta_x (y) \delta_x (r_1) \partial_1 \varphi_n
   (x, y, r_{2 : n - 1}) = - \partial_1 \varphi_n (r_1, r_1, r_{2 : n - 1}),
\]
where $\delta$ denotes the Dirac delta. For sufficiently nice $\varphi_n$ this
kernel is in $L^2_0 (\T^{n - 1})$. On the other hand we get for the
formal limit $\GG_+ W_n (\varphi_n)$ the kernel
\[ \int_x \partial_x (\delta_x (r_1) \delta_x (r_2)) \varphi_n (x, r_{3 : n +
   1}) = - \int_x \delta_x (r_1) \delta_x (r_2) \partial_x \varphi_n (x, r_{3
   : n + 1}) = - \delta_{r_1} (r_2) \partial_1 \varphi_n (r_{2 : n + 1}), \]
which will never be in $L^2_0 (\T^{n + 1})$, no matter how nice
$\varphi_n$ is. The idea is therefore to construct (non-cylinder) functions
for which suitable cancellations happen between $\LL_0$ and
$\GG$ and whose image under the Burgers generator $\LL$
belongs to $L^2 (\mu)$.

It will be easier for us to work on the Fock space $\Gamma L^2 = \Gamma L^2
(\T) = \bigoplus_{n = 0}^{\infty} L^2_0 (\T^n)$ with norm
\[ \| \varphi \|_{\Gamma L^2}^2 = \sum_n n! \| \varphi_n \|_{L^2
   (\T^n)}^2 = \sum_n n! \sum_{k \in \Z^n}  |
   \hat{\varphi}_n (k) |^2, \]
where the functions $\varphi_n \in L^2_0 (\T^n)$ are symmetric, and
where we applied Parseval's identity. We also identify non-symmetric
$\varphi_n \in L^2 (\T^n)$ with their symmetrizations. As discussed
above, the space $\Gamma L^2$ is isomorphic to $L^2 (\mu)$, and in the
following we will often identify $\varphi \in \Gamma L^2$ with an element of
$L^2 (\mu)$ and vice versa, without explicitly mentioning it.

\begin{definition}
  The \emph{number operator (or Ornstein-Uhlenbeck operator) $\NN$} acts on Fock space as
  $(\NN \varphi)_n \assign n \varphi_n$. With a small abuse of notation, we denote with the same symbols $\LL,  \LL_0, \GG_+^m, \GG_-^m$ the Fock version of the operators introduced above in such a way that on smooth  cylinder
  functions we have:
  \begin{equation}
    \LL_0 \sum_{n \geqslant 0} W_n (\varphi_n) = \sum_{n \geqslant 0}
    W_n ((\LL_0 \varphi)_n), \quad \GG_{\pm}^m \sum_{n
    \geqslant 0} W_n (\varphi_n) = \sum_{n \geqslant 0} W_n
    ((\GG_{\pm}^m \varphi)_n) .
  \end{equation}
\end{definition}

\begin{lemma}
  In Fourier variables the operators $\LL_0, \GG^m_+,
  \GG^m_-$ are given by
  \begin{equation}
  \begin{aligned}
    \CF (\LL_0 \varphi)_n (k_{1 : n}) & = - (| 2 \pi k_1 |^2 +
    \cdots + | 2 \pi k_n |^2) \hat{\varphi}_n (k_{1 : n}), \\
    \CF (\GG^m_+ \varphi)_n (k_{1 : n}) & = - (n - 1)
    \1_{| k_1 |, | k_2 |, | k_1 + k_2 | \leqslant m} 2 \pi \iota (k_1
    + k_2) \hat{\varphi}_{n - 1} (k_1 + k_2, k_{3 : n}),\\
    \CF (\GG^m_- \varphi)_n (k_{1 : n}) & = - 2 \pi \iota k_1 n
    (n + 1) \sum_{p + q = k_1} \1_{| k_1 |, | p |, | q | \leqslant m}
    \hat{\varphi}_{n + 1} (p, q, k_{2 : n}),
  \end{aligned}
  \end{equation}
  respectively, where the functions on the right hand side may not be
  symmetric, so strictly speaking we still have to symmetrize them.
\end{lemma}

\begin{proof}
  The Fourier representation for $\LL_0$ is obvious. In the following
  we often use without comment that $\rho^m$ is an even function, i.e.
  $\rho^m_s (x) = \rho_x^m (s)$. The kernel for $(\GG_+^m \varphi)_{n + 1}$ has the Fourier
  transform
   \begin{align*}
    & n \int_{r_{1 : n + 1}} e^{- 2 \pi \iota k \cdot r} \int_{x, s}
    \partial_x \rho^m_x (s) \rho^m_s (r_1) \rho^m_s (r_2) \varphi_n (x, r_{3 :
    n + 1})\\
    & = n\1_{| k_1 |, | k_2 | \leqslant m} \int_{r_{3 : n + 1}}
    \int_{x, s} \partial_x \rho^m_x (s) e^{- 2 \pi \iota (k_1 + k_2) s - 2 \pi
    \iota k_{3 : n + 1} \cdot r_{3 : n + 1}} \sum_{\ell} e^{2 \pi \iota
    (\ell_1 x + \ell_{2 : n} \cdot r_{3 : n + 1})} \hat{\varphi}_n (\ell)\\
    & = - n\1_{| k_1 |, | k_2 |, | k_1 + k_2 | \leqslant m} 2 \pi
    \iota (k_1 + k_2) \sum_{\ell_1} \int_x e^{- 2 \pi \iota (k_1 + k_2) x}
    e^{2 \pi \iota \ell_1 x} \hat{\varphi}_n (\ell_1, k_{3 : n + 1})\\
    & = - n\1_{| k_1 |, | k_2 |, | k_1 + k_2 | \leqslant m} 2 \pi
    \iota (k_1 + k_2) \hat{\varphi}_n (k_1 + k_2, k_{3 : n + 1}) .
  \end{align*}
  To derive $\CF (\GG_-^m \varphi)_{n - 1}$, note that
  \begin{align*}
    & \int_{r_{1 : n - 1}} e^{- 2 \pi \iota k_{1 : n - 1} \cdot r_{1 : n -
    1}} \int_{x, y, s} \partial_x \rho^m_x (s) \rho^m_s (y) \rho^m_s (r_1)
    \varphi_n (x, y, r_{2 : n - 1})\\
    & =\1_{| k_1 | \leqslant m} \int_{r_{2 : n - 1}} e^{- 2 \pi
    \iota k_{2 : n - 1} \cdot r_{2 : n - 1}} \int_{x, y, s} e^{- 2 \pi \iota
    k_1 s} \partial_x \rho^m_x (s) \rho^m_s (y) \varphi_n (x, y, r_{2 : n -
    1})\\
    & =\1_{| k_1 | \leqslant m} \int_{r_{2 : n - 1}} e^{- 2 \pi
    \iota k_{2 : n - 1} \cdot r_{2 : n - 1}} \int_{x, y} \sum_{p + q = k_1}
    \1_{| p |, | q | \leqslant m} (- 2 \pi \iota p) e^{- 2 \pi
    \iota (p x + q y)} \varphi_n (x, y, r_{2 : n - 1})\\
    & = - \sum_{p + q = k_1} \1_{| k_1 |, | p |, | q | \leqslant m}
    2 \pi \iota p \hat{\varphi}_n (p, q, k_{2 : n - 1})\\
    & = - \sum_{p + q = k_1} \1_{| k_1 |, | p |, | q | \leqslant m}
    \pi \iota (p + q) \hat{\varphi}_n (p, q, k_{2 : n - 1})\\
    & = - \sum_{p + q = k_1} \1_{| k_1 |, | p |, | q | \leqslant m}
    \pi \iota k_1 \hat{\varphi}_n (p, q, k_{2 : n - 1}),
  \end{align*}
  from where our representation for $\GG^m_-$ follows.
\end{proof}

\subsection{A priori estimates for the Burgers drift}

Here we derive some a priori estimates for the Burgers drift. We work with
weighted norms on the Fock space.

\begin{lemma}
  \label{lem:G-apriori}Let $w\colon \N_0 \rightarrow \R_+$. Then we have uniformly in $m$
  \begin{equation}\label{eq:Gminus-distributional}
    \| w (\NN) (-\LL_0)^{-\gamma} \GG_-^m \varphi \| \lesssim \| w (\NN-
    1) \NN (-\LL_0)^{3 / 4-\gamma} \varphi \|
  \end{equation}
  for all $\gamma \leqslant 1/4$, and
  \begin{equation}
    \label{eq:Gplus-distributional} \| w (\NN) (-\LL_0)^{-
    \gamma} \GG_+^m \varphi \| \lesssim \| w (\NN+ 1)
    (1 + \NN) (-\LL_0)^{3 / 4 - \gamma} \varphi \|
  \end{equation}
  for all $\gamma > 1/4$. Moreover, we have the following $m$-dependent bound:
  \begin{equation}\label{eq:G-m-dependent}
    \| w (\NN) \GG^m \varphi \| \lesssim m^{1/2} \| (w
    (\NN+ 1) + w(\NN - 1)) (1 + \NN) (-\LL_0)^{1 / 2} \varphi \|.
  \end{equation}
\end{lemma}

\begin{proof}
  1. We start by estimating $\GG_-^m$ uniformly in $m$. Observe that, by the
  Cauchy--Schwartz inequality together with Lemma~\ref{lem:sum-estimate} (here we need $\gamma < 1/2$, which holds because $\gamma \leqslant 1/4$),
  \begin{align*}
    & \left| \sum_{p + q = k_1} \1_{| k_1 |, | p |, | q | \leqslant
    m} \hat{\varphi}_{n + 1} (p, q, k_{2 : n}) \right|^2\\
    & \leqslant \sum_{p + q = k_1} \frac{1}{(p^2 + q^2)^{3 / 2 - 2\gamma}}
    \sum_{p + q = k_1} (p^2 + q^2)^{3 / 2 - 2\gamma} | \hat{\varphi}_{n + 1} (p,
    q, k_{2 : n}) |^2\\
    & \lesssim \frac{|k_1^2|^{2\gamma}}{k_1^{2}} \sum_{p + q = k_1} (p^2 + q^2)^{3
    / 2 - 2\gamma} | \hat{\varphi}_{n + 1} (p, q, k_{2 : n}) |^2,
  \end{align*}
  and thus
  \begin{align*}
    \sum_{k_{1 : n}} (k_1^2+ \dots+ k_n^2)^{-2\gamma} | \CF (\GG_-^m \varphi)_n (k_{1 : n})
    |^2 & \lesssim \sum_{k_{1 : n}} \frac{k_1^2}{(k_1^2+\dots+k_n^2)^{2\gamma}} n^4 \left| \sum_{p + q = k_1}
    \1_{| k_1 |, | p |, | q | \leqslant m} \hat{\varphi}_{n + 1} (p,
    q, k_{2 : n}) \right|^2\\
    & \lesssim n^4 \sum_{k_{1 : n}} \sum_{p + q = k_1} (p^2 + q^2)^{3
    / 2 - 2\gamma} | \hat{\varphi}_{n + 1} (p, q, k_{2 : n}) |^2\\
    & = n^4 \sum_{k_{1 : n + 1}} (k_1^2 + k_2^2)^{3 / 2-2\gamma} |
    \hat{\varphi}_{n + 1} (k_{1 : n + 1}) |^2\\
    & \lesssim n^3 \sum_{k_{1 : n + 1}} (k_1^2 + \cdots + k_{n + 1}^2)^{3 / 2 - 2\gamma} | \hat{\varphi}_{n + 1} (k_{1 : n + 1}) |^2,
  \end{align*}
  where in the last step we used the symmetry of $\hat{\varphi}_{n + 1}$ in the variables $k_{1
  : n + 1}$ and that $3/2-2\gamma \geqslant 1$ (which is equivalent to $\gamma \leqslant 1/4$). Therefore, we have uniformly in $m$
  \begin{align*}
    \| w (\NN) (-\LL_0)^{-\gamma} \GG_-^m \varphi \| & \simeq \sum_{n \geqslant 0}
    n!w (n)^2 \sum_{k_{1 : n}} (k_1^2+ \dots+ k_n^2)^{-2\gamma} | \CF (\GG_-^m \varphi)_n
    (k_{1 : n}) |^2\\
    & \lesssim \sum_{n \geqslant 0} n!w (n)^2 n^3 \sum_{k_{1 : n + 1}} (k_1^2+ \dots+ k_{n+1}^2)^{3/2-2\gamma} | \hat{\varphi}_{n + 1} (k_{1 : n
    + 1}) |^2\\
    & \lesssim \sum_{n \geqslant 1} n!w (n - 1)^2 n^2 \sum_{k_{1 : n}} (k_1^2+ \dots+ k_n^2)^{3/2-2\gamma} | \hat{\varphi}_n (k_{1 : n}) |^2\\
    & = \| w (\NN- 1) \NN (-\LL_0)^{3 / 4 - \gamma}
    \varphi \| .
  \end{align*}
  
  \noindent 2. To derive the uniform-in-$m$ bound for $\GG_+^m$, we apply Lemma~\ref{lem:sum-estimate} in the fourth line below (using that $2\gamma > 1/2$):
    \begin{align*}
    & \sum_{k_{1 : n}} (k_1^2 + \cdots + k_n^2)^{- 2 \gamma} | \CF (\GG_+^m \varphi)_n (k_{1 : n})
    |^2 \\
    & \lesssim \sum_{k_{1 : n}} \1_{| k_1 |, | k_2 |, | k_1 + k_2 | \leqslant
    m} n^2 (k_1^2 + \cdots + k_n^2)^{- 2 \gamma} | k_1 + k_2 |^2 |
    \hat{\varphi}_{n - 1} (k_1 + k_2, k_{3 : n}) |^2\\
    & \lesssim n^2 \sum_{\ell, k_{3:n}} \sum_{k_1 + k_2 = \ell} (k_1^2 + \cdots + k_n^2)^{- 2 \gamma} \ell^2
    | \hat{\varphi}_{n - 1} (\ell, k_{3 : n}) |^2\\
    & \lesssim n^2 \sum_{\ell, k_{3:n}} (\ell^2 + k_3^2 + \cdots
    + k_n^2)^{- 2 \gamma + 1 / 2} \ell^2 | \hat{\varphi}_{n - 1} (\ell, k_{3 : n}) |^2\\
    & \lesssim n \sum_{k_{1 : n - 1}} (k_1^2 + \cdots + k_{n - 1}^2)^{3 / 2 - 2 \gamma} | \hat{\varphi}_{n - 1} (k_{1 : n - 1}) |^2
  \end{align*}
  from where we deduce that uniformly in $m$
  \[ \| w (\NN) (-\LL_0)^{-
    \gamma} \GG_+^m \varphi \| \lesssim \| w (\NN+ 1)
    (1 + \NN) (-\LL_0)^{3 / 4 - \gamma} \varphi \|.
  \]
    
  \noindent 3. If we do not estimate $\GG^m_+$ in a distributional space, we still have
  \begin{align*}
    \sum_{k_{1 : n}} | \CF (\GG_+^m \varphi)_n (k_{1 : n})
    |^2 & \lesssim n^2 \sum_{k_{1 : n}} \1_{| k_1 |, | k_2 |, | k_1 +
    k_2 | \leqslant m} | k_1 + k_2 |^2 | \hat{\varphi}_{n - 1} (k_1 + k_2,
    k_{3 : n}) |^2\\
    & \lesssim n^2 m \sum_{k_{1 : n - 1}}  | k_1 |^2 | \hat{\varphi}_{n - 1}
    (k_{1 : n - 1}) |^2\\
    & \lesssim n m \sum_{k_{1 : n - 1}}  (| k_1 |^2 + \cdots + | k_{n - 1}
    |^2) | \hat{\varphi}_{n - 1} (k_{1 : n - 1}) |^2,
  \end{align*}
  and thus as before $\| w (\NN) \GG_+^m \varphi \| \lesssim m^{1/2} \| w
     (\NN+ 1) (1 + \NN) (-\LL_0)^{1 / 2} \varphi \|$. By making similar use of the cutoff $\1_{|p|, |q| \leqslant m}$ we obtain also the bound $\| w (\NN) \GG_-^m \varphi \| \lesssim m^{1/2} \| w (\NN+ 1) (1 + \NN) (-\LL_0)^{1 / 2} \varphi \|$.
\end{proof}

\begin{remark}
	For later reference let us recall the following bound from the proof: For all $\beta < 1/2$ we have
	\begin{equation}\label{eq:Gminus-CS}
		\left| \sum_{p + q = k_1} \hat{\varphi}_{n + 1} (p, q, k_{2 : n}) \right|^2 \lesssim (k_1^2)^{2\beta-1} \sum_{p + q = k_1} (p^2 + q^2)^{3
    / 2 - 2\beta} | \hat{\varphi}_{n + 1} (p, q, k_{2 : n}) |^2.
	\end{equation}	
\end{remark}

\begin{remark}
	In the study of fluctuations of additive functionals of Markov processes the \emph{graded sector condition} is sometimes useful. This condition assumes that there exists a grading of orthogonal subspaces of $L^2(\mu)$, such that on each subspace the quadratic form of the full generator can be controlled in terms of the one of the symmetric part of the generator, see~\cite[Chapter~2.7.4]{Komorowski2012}. However, while at first glance this may seem tailor made to describe our situation, there is an important restriction: For the graded sector condition we would need
	\[
		|\langle \varphi_n, \GG_- \varphi_{n+1}\rangle| \lesssim (1+n)^\beta \|(-\LL_0)^{1/2} \varphi_n \| \|(-\LL_0)^{1/2} \varphi_{n+1}\|
	\]
	for some $\beta < 1$, see~\cite[eq.~(2.45)]{Komorowski2012} while by Lemma~\ref{lem:G-apriori} we can only take $\beta=1$ and therefore the graded sector condition just barely fails. On the other hand we can take $\|(-\LL_0)^{1/4} \varphi_n \|$ on the right hand side, and we will leverage this gain in regularity. And also for us it will be important that $\beta=1$, for $\beta > 1$ the computations in Section~\ref{sec:bw-apriori} would not work.
\end{remark}

\begin{corollary}
  \label{cor:bw-eq}Let $\varphi \in \Gamma L^2 = L^2 (\mu)$ be such that $\|
  (1 + \NN) (-\LL_0)^{1/2} \varphi \| + \| (-\LL_0)
  \varphi \| < \infty$. Then
  \[ T^m_t \varphi - \varphi = \int_0^t T_s^m (\LL^m \varphi) \dd
     s = \int_0^t \LL^m (T_s^m \varphi) \dd s, \]
  and therefore $t \mapsto T^m_t \varphi$ solves the Kolmogorov backward equation $\partial_t T^m_t \varphi =\LL^m T^m_t \varphi$ with initial condition $T^m_0 \varphi = \varphi$.
\end{corollary}

\begin{proof}
  Let $u_m$ be the solution of the martingale problem for the generator
  $\LL^m$, with initial condition $u$. If $\varphi \in \CC$ is a
  cylinder function, then
  \[ T^m_t \varphi (u) - \varphi (u) =\E_u \left[ \int_0^t \LL^m \varphi (u_m
     (s)) \dd s \right] = \int_0^t T^m_s (\LL^m \varphi) (u) \dd s,
  \]
  so we get the identity $T^m_t \varphi - \varphi = \int_0^t T^m_s
  (\LL^m \varphi) \dd s$ by approximation (with a Bochner integral
  in $L^2 (\mu)$ on the right hand side), where we used our a priori estimates
  for $\GG^m_{\pm}$ and the trivial identity $\| \LL_0 \psi
  \| = \| (-\LL_0) \psi \|$. By Lemma~\ref{lem:semigroup} the map $s
  \mapsto T^m_s \LL^m \varphi \in L^2 (\mu)$ is continuous, and thus $t^{-1}(T^m_t \varphi - \varphi) \to \LL^m
     \varphi$ as $t \to 0$, where the convergence is in $L^2 (\mu)$. From this it follows that $\varphi
  \in \tmop{dom} (\LL^m)$, where now we take $\LL^m$ as the
  infinitesimal generator of $(T^m_t)_{t \geqslant 0}$ (which is only a small
  abuse of notation, because both our definitions of $\LL^m$ agree on
  cylinder functions). Our claim now follows by standard results for
  semigroups in Banach spaces, see e.g. Proposition 1.1.5 in \cite{Ethier1986}.
\end{proof}

\subsection{Controlled functions}\label{sec:controlled}

Lemma~\ref{lem:G-apriori} gives bounds for $\GG^m \varphi$ that are either in distributional spaces, or they diverge
with $m$. To construct a domain that is mapped to $\Gamma L^2$ by the limiting generator $\LL$ we
need to consider functions $\varphi$ for which $\GG \varphi$ and $\LL_0 \varphi$ have some cancellations, so in particular also $\LL_0 \varphi$ should also be a distribution and $\varphi$ should be non-smooth. For finite-dimensional diffusions with distributional drift $b$ such functions can be constructed by solving the resolvent equation $(\lambda - \tfrac12 \Delta) u = b \cdot \nabla u + v$ for nice $v$.

\begin{remark}
	This remark addresses experts in pathwise approaches to singular SPDEs and can be skipped: If $b$ is in the Besov space $\CC^{-\alpha} \assign B_{\infty, \infty}^{-\alpha}$ for $\alpha > 0$, then $u \mapsto b \cdot \nabla u$ is well defined whenever $u \in \CC^{1+\alpha+\varepsilon}$ for some $\varepsilon > 0$, and in that case $b \cdot \nabla u \in \CC^{-\alpha}$. Since the Laplacian gains back $2$ degrees of regularity we are mapped back to $\CC^{2-\alpha}$, so we can close the estimates if $2-\alpha > 1+\alpha$, i.e. if $\alpha < 1/2$. This is the ``Young regime'', but the equation is subcritical for all $\alpha < 1$ and for $\alpha \in [1/2, 1)$ we need to assume that $u$ is not a generic element of the function space $\CC^{2-\alpha}$ but instead it has a special structure, adapted to the equation (it is modelled, or paracontrolled if $\alpha < 2/3$).
\end{remark}

In our case we could start with a nice
function $\psi \in \Gamma L^2$ and try to solve
\[ (\lambda -\LL_0) \varphi =\GG \varphi +
   \psi \qquad \Leftrightarrow \qquad \varphi = (\lambda - \LL_0)^{-1} \GG \varphi + (\lambda - \LL_0)^{-1} \psi \]
so that $\LL \varphi = \lambda \varphi - \psi$, and the right hand side is in $\Gamma L^2$ if
$\varphi, \psi \in \Gamma L^2$. Regarding regularity with respect to $\LL_0$, this is actually in the ``Young regime'': $\GG \varphi$ is well defined whenever $\varphi \in (-\LL_0)^{-1/4-\varepsilon} \Gamma L^2$, and then $\GG$ loses $(-\LL_0)^{3/4}$ ``derivatives'', while $(\lambda-\LL_0)^{-1}$ gains  enough regularity to map back to $(-\LL_0)^{-1/4-\varepsilon} \Gamma L^2$. But in this formal discussion we ignored the behavior with respect to $\NN$, and we are unable to solve the resolvent equation with such simple arguments because $\GG$ introduces some growth in $\NN$ which cannot
be cured by applying $(\lambda -\LL_0)^{- 1}$. So instead we introduce an approximation $\GG^{\succ}$ of $\GG$ which captures the singular part of the small scale behaviour of $\GG$ by letting
\[ \CF (\GG^{\succ} \varphi)_n (k_{1 : n}) \assign
   \1_{| k_{1 : n} |_{\infty} \geqslant N_n} \CF (\GG \varphi)_n (k_{1 : n}) \]
for a suitable ($\NN$-dependent) cutoff $N_n$ to be determined in order for this operator to be small enough in certain norms. Using $\GG^{\succ}$ we introduce a controlled Ansatz of the form
\begin{equation}\label{eq:controlled}
	 \varphi = (-\LL_0)^{- 1} \GG^{\succ} \varphi +
   \varphi^{\sharp}, 
\end{equation}
where $\varphi^{\sharp}$ will be chosen with sufficient regularity in $\Gamma
L^2$. 
 Note that this is essentially the resolvent equation for $\lambda = 0$ and $\psi = (-\LL_0) \varphi^\sharp$, except that we replaced $\GG$ with $\GG^\succ$. The motivation for this is that now we can trade in regularity in $(-\LL_0)$ for regularity in $\NN$, as will become clear from the the computations below. 
A useful intuition about the Ansatz~\eqref{eq:controlled} is that, starting from a given test function $\varphi^\sharp$, it "prepares" functions $\varphi$ which have the right small scale behaviour compatible with the operator $\LL$. 
 
 We start by showing that for an appropriate cutoff $N_n$ we can
solve equation~\eqref{eq:controlled} and express $\varphi$ as a function of $\varphi^{\sharp}$.

\begin{definition}
	A \emph{weight} is a map $w \colon \N_0 \rightarrow (0, \infty)$ such that there exists $C > 0$ with $w (n) \leqslant C w (n + i)$ for $i \in \{ - 1, 1 \}$. In that case we write $|w|$ for the smallest such constant $C$.
\end{definition}

\begin{lemma}\label{lem:controlled-fct}
	Let $w$ be a weight, let $\gamma \in (1/4, 1 / 2]$, and let $L \geqslant 1$. For $N_n = L (1 + n)^3$ we have
  \begin{equation}
    \label{eq:controlled-fct-1} \| w (\NN) (-\LL_0)^{\gamma}
    (-\LL_0)^{- 1} \GG^{\succ} \varphi \| \lesssim |w| L^{-
    1 / 2} \| w (\NN) (-\LL_0)^{\gamma} \varphi \|.
  \end{equation}
  Thus there exists $L_0 = L_0 (|w|)$ such that for all $L \geqslant
  L_0$ and all $\varphi^{\sharp}$ with $\| w (\NN)
  (-\LL_0)^{\gamma} \varphi^{\sharp} \| < \infty$ there is a unique
  controlled
  \[
  	\KK \varphi^{\sharp} \assign \varphi = (-\LL_0)^{- 1} \GG^{\succ} \varphi + \varphi^{\sharp}
  \]
  in $w (\NN)^{- 1} (-\LL_0)^{- \gamma} \Gamma L^2$, and
  $\KK \varphi^{\sharp}$ satisfies
  \begin{equation}
    \label{eq:controlled-fct-estimate} \| w (\NN)
    (-\LL_0)^{\gamma} \KK \varphi^{\sharp} \| + |w|^{- 1} L^{1
    / 2} \| w (\NN) (-\LL_0)^{\gamma} (\KK
    \varphi^{\sharp} - \varphi^{\sharp}) \| \lesssim \| w (\NN)
    (-\LL_0)^{\gamma} \varphi^{\sharp} \| .
  \end{equation}
  We also write $\varphi^{\succ} \assign \KK \varphi^{\sharp} -
  \varphi^{\sharp} = (-\LL_0)^{- 1} \GG^{\succ}
  \KK \varphi^{\sharp}$.
\end{lemma}

\begin{proof}
  1. We start by estimating $\GG_+^{\succ}$ (which is defined like $\GG^\succ$, only with $\GG_+$ in place of $\GG$):
  \begin{align*}
    & \sum_{k_{1 : n}}  | \CF ((-\LL_0)^{\gamma - 1}
    \GG_+^{\succ} \varphi)_n (k_{1 : n}) |^2 \lesssim n^2 \sum_{k_{1 : n}}
    \1_{| k_{1 : n} |_{\infty} \geqslant N_n} \frac{(k_1 +
    k_2)^2}{(k_1^2 + \cdots + k_n^2)^{2 - 2 \gamma}}  | \hat{\varphi}_{n - 1}
    (k_1 + k_2, k_{3 : n}) |^2\\
    &\hspace{40pt} \leqslant n^2 \sum_{\ell_{1 : n - 1}, p}
    \1_{| \ell_{1 : n - 1} |_{\infty} \vee | p | \geqslant N_n/2}
    \frac{\ell_1^2}{((\ell_1 - p)^2 + p^2 + \ell_2^2 + \cdots + \ell_{n -
    1}^2)^{2 - 2 \gamma}} | \hat{\varphi}_{n - 1} (\ell_{1 : n - 1}) |^2,
  \end{align*}
  where we used the change of variables $\ell_1 = k_1 + k_2$, $p = k_2$, and
  $\ell_{2 + i} = k_{1 + i}$ for $i \geqslant 0$, and we used that $|p| \vee |\ell_1 - p| \geqslant N_n$ implies $|p| \vee |\ell_1| \geqslant N_n/2$. Since $(\ell_1 - p)^2 + p^2 \simeq \ell_1^2 + p^2$ we can replace $((\ell_1 - p)^2 + p^2 + \ell_2^2 + \cdots + \ell_{n - 1}^2)^{- (2 - 2
    \gamma)}$ by $(p^2 + \ell_1^2 + \cdots + \ell_{n - 1}^2)^{- (2 - 2
    \gamma)}$. 
  And since $1 - 2 \gamma \geqslant 0$ we have
  \begin{equation}\label{eq:controlled-fct-pr1}
  	\ell_1^2 + \cdots + \ell_{n - 1}^2 \leqslant (\ell_1^2 + \cdots + \ell_{n
     - 1}^2)^{2 \gamma} (p^2 + \ell_1^2 + \cdots + \ell_{n - 1}^2)^{1 - 2
     \gamma}.
  \end{equation}
     We now use the symmetry of $\hat{\varphi}_{n - 1} (\ell_{1 : n - 1})$ in
  $\ell_{1 : n - 1}$ and then we apply~\eqref{eq:controlled-fct-pr1} and Lemma~\ref{lem:sum-estimate}, to derive the estimate
  \begingroup
	\allowdisplaybreaks
  \begin{align*}
    &  n^2 \sum_{\ell_{1 : n - 1}, p}
    \1_{| \ell_{1 : n - 1} |_{\infty} \vee | p | \geqslant N_n/2}
    \frac{\ell_1^2}{(p^2 + \ell_1^2 + \cdots + \ell_{n - 1}^2)^{2 - 2
    \gamma}} | \hat{\varphi}_{n - 1} (\ell_{1 : n - 1}) |^2\\
    & \hspace{30pt} \lesssim n \sum_{\ell_{1 :
    n - 1}, p} \1_{| \ell_{1 : n - 1} |_{\infty} \vee | p | \geqslant
    N_n/2} \frac{\ell_1^2 + \cdots + \ell_{n - 1}^2}{(p^2 + \ell_1^2 + \cdots +
    \ell_{n - 1}^2)^{2 - 2 \gamma}} | \hat{\varphi}_{n - 1} (\ell_{1 : n -
    1}) |^2\\
    & \hspace{30pt} \leqslant n \sum_{\ell_{1
    : n - 1}, p} (\1_{| p | \geqslant N_n/2} +\1_{| \ell_{1 :
    n - 1} |_{\infty} \geqslant N_n/2}) \frac{(\ell_1^2 + \cdots + \ell_{n -
    1}^2)^{2 \gamma}}{p^2 + \ell_1^2 + \cdots + \ell_{n - 1}^2} |
    \hat{\varphi}_{n - 1} (\ell_{1 : n - 1}) |^2\\
    & \hspace{30pt} \lesssim  n \sum_{\ell_{1 :
    n - 1}} \left( \sum_{| p | \geqslant N_n/2} \frac{1}{p^2} +
    \frac{\1_{| \ell_{1 : n - 1} |_{\infty} \geqslant N_n/2}}{(\ell_1^2
    + \cdots + \ell_{n - 1}^2)^{1 / 2}} \right) (\ell_1^2 + \cdots + \ell_{n -
    1}^2)^{2 \gamma} | \hat{\varphi}_{n - 1} (\ell_{1 : n - 1}) |^2\\
    & \hspace{30pt} \lesssim  n \sum_{\ell_{1 :
    n - 1}} N_n^{- 1} (\ell_1^2 + \cdots + \ell_{n - 1}^2)^{2 \gamma} |
    \hat{\varphi}_{n - 1} (\ell_{1 : n - 1}) |^2
  \end{align*}
  \endgroup
  and thus with our choice of $N_n = L (1+n)^3$
  \begin{equation}\label{eq:controlled-fct-pr2}
  	 \| w(\NN) (-\LL_0)^{\gamma} (-\LL_0)^{-1} \GG_+^\succ \varphi \| \lesssim |w| L^{-1/2} \| w(\NN) (-\LL_0)^\gamma \varphi \|.
  \end{equation}
  
  \noindent 2. Next, let us bound $\GG_-^\succ$: 
  We apply \eqref{eq:Gminus-CS} with $\beta = 3/4 - \gamma < 1/2$ (here we need $\gamma > 1/4$) to estimate
  \begin{align*}
  	\sum_{k_{1 : n}}  | \CF ((-\LL_0)^{\gamma - 1} \GG_-^{\succ} \varphi)_n (k_{1 : n}) |^2 & \lesssim \sum_{k_{1:n}} \frac{ \1_{|k_{1:n}|_\infty \geqslant N_n} n^4 k_1^2}{(k_1^2 + \cdots + k_n^2)^{2 - 2\gamma}} \left| \sum_{p+q = k_1} \hat{\varphi}_{n + 1} (p, q, k_{2 : n}) \right|^2 \\
  	& \lesssim \sum_{k_{1:n}} \frac{ \1_{|k_{1:n}|_\infty \geqslant N_n} n^4 k_1^2 (k_1^2)^{3/2 - 2\gamma-1}}{(k_1^2 + \cdots + k_n^2)^{2 - 2\gamma}}  \sum_{p + q = k_1} (p^2 + q^2)^{2\gamma } | \hat{\varphi}_{n + 1} (p, q, k_{2 : n}) |^2 \\
  	& \lesssim \sum_{k_{1:n}} \frac{ \1_{|k_{1:n}|_\infty \geqslant N_n} n^4 (k_1^2)^{3/2 }}{(k_1^2 + \cdots + k_n^2)^{2}}  \sum_{p + q = k_1} (p^2 + q^2)^{2\gamma } | \hat{\varphi}_{n + 1} (p, q, k_{2 : n}) |^2 \\
  	& \leqslant N_n^{-1} n^4 \sum_{\ell_{1:n+1}} (\ell_1^2 + \cdots + \ell_{n+1}^2)^{2\gamma } | \hat{\varphi}_{n + 1} (\ell_{1:n+1}) |^2,
  \end{align*}
  which together with $N_n = L(1+n)^3$ leads to the bound
  \begin{equation}\label{eq:controlled-fct-pr3}
  	\| w(\NN) (-\LL_0)^\gamma (-\LL_0)^{-1} \GG_-^\succ \varphi \| \lesssim |w| L^{-1/2}	 \| w(\NN) (-\LL_0)^\gamma \varphi \|.
  \end{equation}
  The claimed inequality \eqref{eq:controlled-fct-1} now follows by combining \eqref{eq:controlled-fct-pr2} and \eqref{eq:controlled-fct-pr3}.

  \noindent 3. Consequently, for given $\varphi^{\sharp}
  \in w (\NN)^{- 1} (-\LL_0)^{- \gamma} \Gamma L^2$ the map
  \[ \Psi\colon w (\NN)^{- 1} (-\LL_0)^{- \gamma} \Gamma L^2 \ni
     \psi \mapsto (-\LL_0)^{- 1} \GG^{\succ} \psi +
     \varphi^{\sharp} \in w (\NN)^{- 1} (-\LL_0)^{- \gamma}
     \Gamma L^2 \]
  satisfies for some $K > 0$
  \begin{align*}
    \| w (\NN) (-\LL_0)^{\gamma} \Psi (\psi) \| & \leqslant
    \| w (\NN) (-\LL_0)^{\gamma} (-\LL_0)^{- 1}
    \GG^{\succ} \psi \| + \| w (\NN)
    (-\LL_0)^{\gamma} \varphi^{\sharp} \|\\
    & \leqslant K |w| L^{- 1 / 2} \| w (\NN) (-\LL_0)^{\gamma}
    \psi \| + \| w (\NN) (-\LL_0)^{\gamma} \varphi^{\sharp}
    \| < \infty,
  \end{align*}
  i.e. $\Psi$ is well defined. If $L$ is large enough so that $K |w| L^{- 1 / 2}
  \leqslant 1 / 2$, then $\Psi$ is a contraction and it leaves the ball with
  radius $2 \| w (\NN) (-\LL_0)^{\gamma} \varphi^{\sharp} \|$
  invariant. Therefore, it has a unique fixed point $\KK
  \varphi^{\sharp}$ which satisfies
  \[ \| w (\NN) (-\LL_0)^{\gamma} \KK
     \varphi^{\sharp} \| \leqslant 2 \| w (\NN)
     (-\LL_0)^{\gamma} \varphi^{\sharp} \|, \]
  and then also
  \[ \| w (\NN) (-\LL_0)^{\gamma} (\KK
     \varphi^{\sharp} - \varphi^{\sharp}) \| = \| w (\NN)
     (-\LL_0)^{\gamma} (-\LL_0)^{- 1} \GG^{\succ}
     \KK \varphi^{\sharp} \| \lesssim |w| L^{- 1 / 2} \| w
     (\NN) (-\LL_0)^{\gamma} \varphi^{\sharp} \|.
  \]
\end{proof}

\begin{remark}
  \label{rmk:general-remainder}The lemma shows that for all $\varphi \in
  w (\NN)^{- 1} (-\LL_0)^{-\gamma} \Gamma L^2$ we can define
  $\varphi^{\sharp} \assign \varphi - (-\LL_0)^{- 1}
  \GG^{\succ} \varphi$ and then
  \[ \| w (\NN) (-\LL_0)^{\gamma} \varphi^{\sharp} \|
     \lesssim \| w (\NN) (-\LL_0)^{\gamma} \varphi \| . \]
  However, this only works up to $\gamma = 1 / 2$, so no matter how regular
  $\varphi$ is, the (spatial) regularity of $\varphi^{\sharp}$ is limited in
  general. The key point of Lemma~\ref{lem:controlled-fct} is that it
  identifies a class of $\varphi$ for which $\varphi^{\sharp}$ has arbitrarily
  good regularity.
\end{remark}

\begin{remark}
   The cutoff $N_n$ for which we can construct $\KK \varphi^\sharp$ depends on the weight $w$ via $|w|$; we say that the cutoff is \emph{adapted to the weight $w$} if the construction of Lemma~\ref{lem:controlled-fct} works. If we consider weights $w(n) = (1+n)^\alpha$ with $|\alpha| \leqslant K$ for a fixed $K$, then $|w|$ is uniformly bounded and we can choose one cutoff which is adapted to all those weights. This is the situation that we are mostly interested in.
\end{remark}

\begin{remark}
  \label{rmk:controlled-fct-m}The bound~(\ref{eq:controlled-fct-1}) also holds
  for $\GG^{m, \succ}$, which is defined analogously to $\GG^\succ$. Therefore, we can also construct a map $\KK^m\colon w (\NN)^{-
  1} (-\LL_0)^{- \gamma} \Gamma L^2 \rightarrow w (\NN)^{- 1}
  (-\LL_0)^{- \gamma} \Gamma L^2$ that associates to every
  $\varphi^{\sharp} \in w (\NN)^{- 1} (-\LL_0)^{- \gamma}
  \Gamma L^2$ a unique $\KK^m \varphi^{\sharp} \in w (\NN)^{-
  1} (-\LL_0)^{- \gamma} \Gamma L^2$ with
  \[ \KK^m \varphi^{\sharp} = (-\LL_0)^{- 1}
     \GG^{m, \succ} \KK^m \varphi^{\sharp} +
     \varphi^{\sharp} . \]
\end{remark}

Let us write $\GG^\prec = \GG - \GG^\succ$. The following proposition controls $\LL\KK \varphi^{\sharp}$
in terms of $\varphi^{\sharp}$ and it is formulated in the limit $m
\rightarrow \infty$. But by Remark~\ref{rmk:controlled-fct-m} it is clear that
similar bounds hold for $\LL^m \KK^m \varphi^{\sharp}$,
uniformly in $m$.

\begin{proposition}\label{prop:domain}
   Let $w$ be a weight, let $\gamma \geqslant 0$, and let the cutoff $N_n$ be adapted to $w$ and $(w(n)(1+n)^{9/2+7\gamma})_n$, and let $\delta > 0$. Consider
  \[ \varphi^{\sharp} \in w (\NN)^{- 1} (-\LL_0)^{- 1}  \Gamma
     L^2 \cap w(\NN)^{- 1} (1 + \NN)^{- 9/2 - 7\gamma} (-\LL_0)^{- 1 / 4 - \delta}  \Gamma L^2 . \]
  We set $\varphi \assign \KK \varphi^{\sharp}$. Then $\LL
  \varphi \assign \LL_0 \varphi^{\sharp} +\GG^{\prec}
  \varphi$ is a well defined operator and we have
  \begin{equation} \label{eq:domain} \| w (\NN) (-\LL_0)^\gamma \GG^\prec \varphi \| \lesssim
      \| w (\NN) (1 + \NN)^{9/2+7\gamma}
     (-\LL_0)^{1 / 4 + \delta} \varphi^\sharp \|.
  \end{equation}
  Obviously we also have $\| w(\NN) \LL_0 \varphi^\sharp \| = \| w(\NN) (-\LL_0) \varphi^\sharp \|$.
\end{proposition}

\begin{proof}
  We treat $\GG^\prec_+$ and $\GG^\prec_-$ separately (both with their obvious definition). We also assume that $\delta \in (0,1/4]$, but once we established the bound~\eqref{eq:domain} for such $\delta$ it holds of course also for $\delta > 1/4$.
  
  \noindent 1. To control $\GG^\prec_+ \varphi$, we bound
    \begin{align*}
    & \sum_{k_{1 : n}} (k_1^2 + \cdots + k_n^2)^{2 \gamma} | \CF
    (\GG_+^{\prec} \varphi)_n (k_{1 : n}) |^2\\
    & \lesssim n^2 \sum_{k_{1 : n}} \1_{| k_{1 : n} |_{\infty} <
    N_n} (k_1^2 + \cdots + k_n^2)^{2 \gamma} | k_1 + k_2 |^2 |
    \hat{\varphi}_{n - 1} (k_1 + k_2, k_{3 : n}) |^2\\
    & \lesssim n^2 \sum_{k_{1 : n}} \1_{| k_{1 : n} |_{\infty} < N_n} N_n^{4\gamma} n^{2\gamma} (|k_1+k_2|^2)^{1/2 + 2\delta} N_n^{1-4\delta} | \hat{\varphi}_{n - 1} (k_1+k_2, k_{2 : n - 1}) |^2\\
    & \lesssim n^{2+2\gamma} N_n^{2+4\gamma - 4\delta} \sum_{k_{1 : n - 1}} (k_1^2 + \cdots + k_{n - 1}^2)^{1/2+2\delta} |\hat{\varphi}_{n - 1} (k_{1 : n - 1}) |^2,
  \end{align*}
  and since $N_n \simeq (n + 1)^3$ we get $\| w (\NN) \GG_+^{\prec} \varphi \| \lesssim \| w(\NN) (1 + \NN)^{9/2+7\gamma} (-\LL_0)^{1 / 4 + \delta} \varphi \|$. With Lemma~\ref{lem:controlled-fct} we can estimate the right hand side by $\| w(\NN) (1 + \NN)^{9/2 + 7\gamma} (-\LL_0)^{1 / 4 + \delta} \varphi^\sharp \|$, because we assumed that $\delta \in (0,1/4]$.

\noindent 2. Next, let us estimate $\GG^\prec_- \varphi$. As usual we apply~\eqref{eq:Gminus-CS}, this time with $\beta = 1/2 - \delta$, to bound
\begin{align*}
	& \sum_{k_{1 : n}} (k_1^2 + \cdots + k_n^2)^{2 \gamma} | \CF
    (\GG_-^{\prec} \varphi)_n (k_{1 : n}) |^2 \\
    & \lesssim (n+1)^4 \sum_{k_{1 : n}} \1_{|k_{1:n}|_\infty < N_n} k_1^2 (k_1^2 + \cdots + k_n^2)^{2 \gamma} \bigg| \sum_{p+q=k_1} \hat{\varphi}_{n+1} (p,q, k_{2 : n}) \bigg|^2 \\
    & \lesssim (n+1)^4 \sum_{k_{1 : n}} \1_{|k_{1:n}|_\infty < N_n} k_1^2 (k_1^2 + \cdots + k_n^2)^{2 \gamma} |k_1|^{-4\delta} \sum_{p+q=k_1} (p^2+q^2)^{1/2+2\delta} |\hat{\varphi}_{n+1} (p,q, k_{2 : n})|^2 \\
    & \lesssim (n+1)^{4+2\gamma} N_n^{2+4\gamma} \sum_{\ell_{1:n+1}}  (\ell_1^2 + \dots + \ell_{n+1}^2)^{1/2+2\delta} |\hat{\varphi}_{n+1} (\ell_{1:n+1})|^2,
\end{align*}
from where we deduce as before that $\| w (\NN) \GG_-^{\prec} \varphi \| \lesssim \| w(\NN) (\NN + 1)^{9/2 + 7\gamma} (-\LL_0)^{1/4+\delta} \varphi^\sharp \|$.
\end{proof}

To simplify the notation we write from now for $\gamma \geqslant 0$
\begin{equation}
	\alpha(\gamma) := 9/2 + 7 \gamma.
\end{equation}

\begin{lemma}\label{lem:dom-def}
 For a given weight $w$ and a cutoff as in Proposition~\ref{prop:domain} (for $\gamma = 0$) we set
  \[ \CD_w (\LL) \assign \{ \KK \varphi^{\sharp} :
     \varphi^{\sharp} \in w (\NN)^{- 1} (-\LL_0)^{- 1} \Gamma
     L^2 \cap w (\NN)^{- 1} (1 + \NN)^{- 9/2}
     (-\LL_0)^{- 1 / 2} \Gamma L^2 \} . \]
  Then $\CD_w (\LL)$ is dense in $w (\NN)^{- 1}
  \Gamma L^2$. More precisely, for all
  \[
  		\psi \in w (\NN)^{- 1} (-\LL_0)^{- 1} \Gamma L^2 \cap w (\NN)^{- 1} (1 + \NN)^{- 9 / 2} (-\LL_0)^{- 1 / 2} \Gamma L^2
  \]
  and for all $M \geqslant 1$ there exists a $\varphi^M \in \CD_w(\LL)$ such that
  \begin{equation}\label{eq:quantitative-density} 
  \begin{aligned}
    \| w (\NN) (-\LL_0)^{1/2} (\varphi^M - \psi) \| & \lesssim M^{- 1 / 2} \| w
    (\NN) (-\LL_0)^{1/2} \psi \|,\\ 
    \| w (\NN) (-\LL_0)^{1 / 2} \varphi^M \| & \lesssim \| w
    (\NN) (-\LL_0)^{1 / 2} \psi \|, \\
    \| w (\NN) \LL \varphi^M \| & \lesssim M^{1 / 2} (\| w
    (\NN) (-\LL_0) \psi \| + \| w (\NN) (\NN+
    1)^{9 / 2} (-\LL_0)^{1 / 2} \psi \|).
  \end{aligned}
  \end{equation}
  If $w \equiv 1$ we simply write $\CD(\LL)$.
\end{lemma}

\begin{proof}
  Let $\psi$ be as in the statement of the lemma. Since such $\psi$ are dense in $w (\NN)^{- 1} \Gamma L^2$ it
  suffices to construct $\varphi^M$ such that the inequalities~\eqref{eq:quantitative-density} hold. For this purpose we apply Lemma~\ref{lem:controlled-fct} to find a unique function $\varphi^M
  \in w (\NN)^{- 1} \Gamma L^2$ that satisfies
  \[ \hat{\varphi}_n^M (k_{1 : n}) = \1_{| k |_{\infty} \geqslant M
     N_n} \CF((-\LL_0)^{-1} \GG \varphi^M)_n(k_{1:n}) + \hat{\psi}_n
     (k_{1 : n}), \]
  and for which the first two estimates in~\eqref{eq:quantitative-density} hold by Lemma~\ref{lem:controlled-fct}. To see that $\varphi^M \in \CD_w (\LL)$ note
  that
  \[ \hat{\varphi}_n^M (k_{1 : n}) = \CF((-\LL_0)^{-1} \GG^\succ \varphi^M)_n(k_{1:n}) +
     \hat{\varphi}_n^{M,\sharp} (k_{1 : n}), \]
  where
  \[ \hat{\varphi}_n^{M,\sharp} (k_{1 : n}) = \hat{\psi}_n (k_{1 : n})
     - \1_{N_n \leqslant | k |_{\infty} < M N_n} \CF( (-\LL_0)^{-1} \GG \varphi^M)_n(k_{1:n}). \]
  In particular we have $\LL \varphi^M = \GG^\prec \varphi^M + \LL_0 \varphi^{M,\sharp}$, and by Proposition~\ref{prop:domain} it suffices to estimate $\varphi^{M,\sharp}$ in $w (\NN)^{- 1} (-\LL_0)^{- 1}  \Gamma L^2 \cap w(\NN)^{- 1} (1 + \NN)^{- 9/2} (-\LL_0)^{- 1 / 2}  \Gamma L^2$. The first contribution $\psi$ satisfies the required bounds by assumption, so it suffices to show that the second contribution, denote it as $\psi^M$,
  satisfies
  \begin{equation}\label{eq:psiM-estimate}
  \begin{aligned}
  		\| w (\NN) (-\LL_0) \psi^M \| & \lesssim M^{1 / 2}  \| w (\NN) (1 + \NN)^{9 / 2} (-\LL_0)^{1 / 2} \psi \|,  \\
     	\| w (\NN) (\NN+1)^{9/2} (-\LL_0)^{1 / 2} \psi^M \| & \lesssim \| w (\NN) (1+\NN)^{9/2} (-\LL_0)^{1 / 2} \psi \|.
  \end{aligned}
  \end{equation}
  But
  \begin{align*}
  	\CF((-\LL_0) \psi^M)_n(k_{1:n}) = - \1_{N_n \leqslant | k |_{\infty} < M N_n} \CF( \GG \varphi^M)_n(k_{1:n}),
  \end{align*}
  so that we can estimate this term similarly as in~\eqref{eq:G-m-dependent}. If the cutoff $M N_n$ was independent of $n$, we would get $\| w(\NN) (-\LL_0) \psi^M \| \lesssim (M N_n)^{1/2} \| w(\NN) (1+\NN)(-\LL_0)^{1/2} \varphi^M \|$ from~\eqref{eq:G-m-dependent}, so after including the factor $N_n \simeq (1+n)^3$ into the weight we get
  \[
  	\| w(\NN) (-\LL_0) \psi^M \| \lesssim M^{1/2} \| w(\NN) (1+\NN)^{5/2} (-\LL_0)^{1/2} \varphi^M \|,
  \]
  and then the first estimate of~\eqref{eq:psiM-estimate} follows from~\eqref{eq:controlled-fct-estimate}. Similarly
  \begin{align*}
  	|\CF((-\LL_0)^{1/2} \psi^M)_n(k_{1:n})|^2 & \simeq (k_1^2+\dots+k_n^2)^{-1} \1_{N_n \leqslant | k |_{\infty} < M N_n} |\CF( \GG \varphi^M)_n(k_{1:n})|^2 \\
  	& \lesssim N_n^{-1} |\CF((-\LL_0)^{-1/4} \GG_- \varphi^M)_n(k_{1:n})|^2 \\
  	&\quad + N_n^{-2/3} |\CF((-\LL_0)^{-1/3} \GG_+ \varphi^M)_n(k_{1:n})|^2,
  \end{align*}
  and since $N_n \simeq (1+n)^3$ we get with~\eqref{eq:Gminus-distributional}, \eqref{eq:Gplus-distributional} that
  \[
  	\| w(\NN) (1+\NN)^{9/2} (-\LL_0)^{1/2} \psi^M \| \lesssim \| w(\NN) [(1+\NN)^{4} (-\LL_0)^{1/2} + (1 + \NN)^{9/2} (-\LL_0)^{-5/12}] \varphi^M \|,
  \]
  which together with (\ref{eq:controlled-fct-estimate}) yields~\eqref{eq:psiM-estimate} and then~\eqref{eq:quantitative-density}.
\end{proof}

\begin{remark}
As discussed before, the analysis above works also for $\LL^m$ and we define
\[ \CD_w (\LL^m) \assign \{ \KK^m \varphi^{\sharp} :
   \varphi^{\sharp} \in w (\NN)^{- 1} (-\LL_0)^{- 1} \Gamma
   L^2 \cap w (\NN)^{- 1} (1 + \NN)^{- 9 / 2}
   (-\LL_0)^{- 1 / 2} \Gamma L^2 \} . \]
\end{remark}

\begin{remark}
	The same construction works for the operator $\LL^{(\lambda)} = \LL_0 + \lambda \GG$ for $\lambda \in \R$. For $\lambda \neq 1$ the intersection of the resulting domain $\CD(\LL^{(\lambda)})$ with $\CD(\LL)$ consists only of constants.
\end{remark}

\begin{lemma}\label{lem:dissipative}
For any $\varphi\in\CD(\LL)$ we have
\[
\langle \varphi, \LL \varphi \rangle= - \| (-\LL_0)^{1/2} \varphi \|^2 \leqslant 0,
\]
in particular the operator $(\LL,\CD(\LL))$ is dissipative.
\end{lemma}
\begin{proof}
Note that $\varphi\in\CD(\LL)$ implies $\LL_0 \varphi, \GG \varphi \in (-\LL_0)^{1/2} \Gamma L^2$ and $\varphi \in (-\LL_0)^{-1/2} (1 + \NN)^{-1} \Gamma L^2$ and therefore we can conclude by approximation in the chain of equalities
\[
\langle \varphi, \LL \varphi \rangle=- \langle \varphi, (-\LL_0) \varphi \rangle+\langle \varphi, \GG \varphi \rangle =- 
\langle \varphi, (-\LL_0) \varphi \rangle = - \| (-\LL_0)^{1/2} \varphi \|^2
\]
since all the inner products are well defined. In particular we used the antisymmetry of the form associated to $\GG$:
\[
\langle \varphi, \GG \varphi \rangle = \lim_{m\to\infty } \langle \varphi, \GG^m \varphi \rangle = - \lim_{m\to\infty } \langle \GG^m \varphi,  \varphi \rangle = - \langle \GG \varphi,  \varphi \rangle.
\]
\end{proof}
\begin{remark}
We can introduce another dissipative operator $\LL^-$ given by $\LL^-= \LL_0 - \GG = \LL^{(-1)}$ on the domain $\CD(\LL^-)$. Then  if
$\varphi \in \CD (\LL)$ and $\psi \in \CD
(\LL^-)$ we have
$ \LL_0 \varphi, \GG \varphi, \LL_0 \psi,
   \GG \psi \in (-\LL_0)^{1 / 2} \Gamma L^2$
and
$\varphi, \psi \in (\NN+ 1)^{- 1} (-\LL_0)^{- 1 / 2}
   \Gamma L^2 $
so the identities 
$ \LL \varphi =\LL_0 \varphi +\GG \varphi$,
   $\LL^- \psi =\LL_0 \psi -\GG \psi$ hold (as distributions)
and
\[ \langle \psi, \LL \varphi \rangle = \langle \psi, \LL_0
   \varphi \rangle + \langle \psi, \GG \varphi \rangle = \langle
   \psi, \LL_0 \varphi \rangle - \langle \GG \psi, \varphi
   \rangle = \langle \LL^- \psi, \varphi \rangle. \]
As a consequence $\LL^- \subseteq \LL^{\ast}$ and symmetrically $\LL \subseteq (\LL^-)^{\ast}$. The closed operators $\LL^{\ast}, (\LL^-)^{\ast}$ are dissipative and satisfy
$$
\LL^{\ast}, (\LL^-)^{\ast} \leqslant \LL_0
$$
in the sense of quadratic forms and on their respective domains.
\end{remark}

\section{The Kolmogorov backward equation}\label{sec:bw-eq}

So far we constructed a dense domain $\CD(\LL)$ for the operator $\LL$. In this section we will analyze the Kolmogorov backward equation $\partial_t \varphi = \LL \varphi$. More precisely we consider the backward equation for the Galerkin approximation~\eqref{eq:burgers-galerkin} with generator $\LL^m$, and we derive uniform estimates in controlled spaces for the solution. By compactness, this gives the existence of strong solutions to the backward equation after removing the cutoff. Uniqueness easily follows from the dissipativity of $\LL$.

\subsection{A priori bounds}\label{sec:bw-apriori}

Recall that $T^m$ is the semigroup generated by the Galerkin approximation $u^m$, the solution to~\eqref{eq:burgers-galerkin}. Here we consider $\varphi^m (t) = T^m_t \varphi^m_0$ for $\varphi^m_0\in \CD (\LL^m)$ and we derive some basic a priori estimates without using the controlled structure that we introduced above. Roughly speaking our aim is to gain some control of the growth in the chaos variable $n$ by making use of the antisymmetry of $\GG$. In the next section we then handle the regularity with respect to $(-\LL_0)$ by using the controlled structure.

  Recall from Corollary~\ref{cor:bw-eq} that $\partial_t \varphi^m =\LL^m
\varphi^m$, which yields
\[ \partial_t \frac{1}{2} \| \varphi^m (t) \|^2 = \langle \varphi^m (t),
   \LL_0 \varphi^m (t) +\GG^m \varphi^m (t) \rangle, \]
and since we saw in Lemma~\ref{lem:Gm} that $\langle \varphi^m, \GG^m \varphi^m \rangle = -\langle \GG^m \varphi^m,
\varphi^m \rangle$, we get $\langle  \varphi^m, \GG^m \varphi^m \rangle = 0$.
However, this
argument is only formal because $\GG^m$ introduces some growth in the
chaos variable $n$, and we do not control the decay of
$\varphi^m$ in $n$. Therefore, it is not clear that the ``integration by
parts'' $\langle \GG^m \varphi^m, \varphi^m \rangle = - \langle
\varphi^m, \GG^m \varphi^m \rangle$ is allowed. To overcome this
difficulty we fix a function $w \colon \N_0 \rightarrow \R_+$ of
compact support and note that
\begin{align*}
  \partial_t \frac{1}{2} \| w (\NN) \varphi^m (t) \|^2 & = \langle
  \varphi^m (t), w (\NN)^2 \LL^m \varphi^m (t) \rangle\\
  & = - \| w (\NN) (-\LL_0)^{1 / 2} \varphi^m (t) \|^2 +
  \langle \varphi^m (t), w (\NN)^2 \GG^m \varphi^m (t) \rangle,
\end{align*}
where we used that $\LL_0$ commutes with $w (\NN)$. Let us
focus on the second term on the right hand side, for which
\[ w (\NN)^2 \GG^m_+ \varphi^m =\GG^m_+ (w
   (1 + \NN)^2 \varphi^m), \qquad w (\NN)^2 \GG_-
   \varphi^m =\GG_-^m (w (\NN- 1)^2 \varphi^m), \]
and therefore Lemma~\ref{lem:Gm} gives
\begin{align*}
  \langle \varphi^m, w (\NN)^2 \GG^m \varphi^m \rangle & = \langle \varphi^m, \GG_+^m (w (\NN+
  1)^2 \varphi^m) \rangle + \langle \varphi^m, \GG_-^m (w
  (\NN- 1)^2 \varphi^m) \rangle\\
  & = - \langle \GG^m_- \varphi^m, w (1 + \NN)^2 \varphi^m
  \rangle - \langle \GG^m_+ \varphi^m, w (\NN- 1)^2 \varphi^m
  \rangle\\
  & = - \langle \GG^m \varphi^m, w
  (\NN)^2 \varphi^m \rangle\\
  &\quad  + \langle \GG^m_- \varphi^m, (w (\NN)^2 - w
  (1 + \NN)^2) \varphi^m \rangle + \langle \GG^m_+ \varphi^m, (w (\NN)^2 - w
  (\NN- 1)^2) \varphi^m \rangle,
\end{align*}
that is
\[ 2 \langle \varphi^m, w (\NN)^2 \GG^m \varphi^m
   \rangle = \langle \GG^m_- \varphi^m, (w (\NN)^2 - w
   (1 + \NN)^2) \varphi^m \rangle + \langle \GG_+^m \varphi^m,
   (w (\NN)^2 - w (\NN- 1)^2) \varphi^m \rangle. \]
Note that these computations
are rigorous since the compact support of $w$ ensures that the inner product
involves only a finite number of chaoses. Let us denote $h (n) = w (n)^2 - w
(n - 1)^2$, then we have for the $\GG^m_-$ term
\[ \langle \GG_-^m \varphi^m, - h (\NN+ 1) \varphi^m \rangle
   = \langle \varphi^m, \GG_+^m (h (\NN+ 1) \varphi^m)
   \rangle = \langle \varphi^m, h (\NN) \GG_+^m \varphi^m
   \rangle = \langle \GG_+^m \varphi^m, h (\NN) \varphi^m
   \rangle, \]
i.e.
\begin{equation}
  \langle \varphi^m, w (\NN)^2 \GG^m \varphi^m \rangle =
  \langle \GG_+^m \varphi^m, h (\NN) \varphi^m \rangle .
\end{equation}
Consider now a function $g \colon \N_0 \rightarrow [0,\infty)$ such that $g(n)=0$ only if $h(n) = h(n+1) = 0$; we will choose the precise form of $g$ later.
From the Cauchy-Schwarz inequality and estimate
(\ref{eq:Gplus-distributional}) we get
\begin{align*}
  | \langle \GG_+^m \varphi^m, h (\NN) \varphi^m \rangle | &
  \leqslant \left\| \frac{h (\NN)}{g (\NN)}
  (-\LL_0)^{- 1 / 2} \GG^m_+ \varphi^m \right\| \| g
  (\NN) (-\LL_0)^{1 / 2} \varphi^m \|\\
  & \lesssim \left\| \frac{h (\NN+ 1)}{g (\NN+ 1)}
  (\NN+ 1) (-\LL_0)^{1 / 4} \varphi^m \right\| \| g
  (\NN) (-\LL_0)^{1 / 2} \varphi^m \|,
\end{align*}
and then Young's inequality gives
\[ \left\| \frac{h (\NN+ 1)}{g (\NN+ 1)} (\NN+ 1)
   (-\LL_0)^{1 / 4} \varphi^m \right\| \lesssim \delta \| g
   (\NN) (-\LL_0)^{1 / 2} \varphi^m \| + \delta^{- 1} \left\|
   \left( \frac{h (\NN+ 1) (\NN+ 1)}{g (\NN+ 1) g
   (\NN)^{1 / 2}} \right)^2 \varphi^m \right\| \]
for all $\delta > 0$, which with another application of Young's inequality
yields
\[ | \langle \GG_+^m \varphi^m, h (\NN) \varphi^m \rangle |
   \leqslant \delta \| g (\NN) (-\LL_0)^{1 / 2} \varphi^m
   \|^2 + C (\delta) \left\| \left( \frac{h (\NN+ 1) (\NN+
   1)}{g (\NN+ 1) g (\NN)^{1 / 2}} \right)^2 \varphi^m
   \right\|^2 . \]
Recall that a \emph{dyadic partition of unity} consists of two functions
$\rho_{- 1}, \rho \in C^{\infty}_c (\R)$ such that with $\rho_i
\assign \rho (2^{- i} \cdot)$ for $i \geqslant 0$ we have $\tmop{supp}
(\rho_i) \cap \tmop{supp} (\rho_j)$ for $| i - j | > 1$ and such that $\sum_{i
\geqslant - 1} \rho_i (x) \equiv 1$; see \cite[Chapter~2.2]{Bahouri2011} for a
construction. In the following we write $i \sim j$ if $2^i \simeq 2^j$, i.e.
if $| i - j | \leqslant L$ for some fixed $L > 0$. Let us take $w (n) = \rho_i
(n)$ for a dyadic partition of unity, and $g = \sum_{j \sim i} \rho_j$. Then
we have for $n \simeq 2^i$
\begin{align*}
  \left| \frac{h (n + 1) (n + 1)}{g (n + 1) g (n)^{1 / 2}} \right| & = | h (n
  + 1) (n + 1) | = | (\rho_i (n + 1)^2 - \rho_i (n)^2) (n + 1) |\\
  & \leqslant (\rho_i (n) + \rho_i (n + 1)) | \rho_i (n + 1) - \rho_i (n) |
  (n + 1)\\
  & \lesssim \sum_{j \sim i} \rho_j (n) \max \{ \| \rho_{- 1}' \|_{\infty},
  \| \rho' \|_{\infty} \} 2^{- i} (n + 1) \lesssim \sum_{j \sim i} \rho_j (n),
\end{align*}
and $h (n + 1) (n + 1) / (g (n + 1) g (n)^{1 / 2}) = 0$ for $n \not\simeq
2^i$, and thus for all $\delta > 0$ there exists $C = C (\delta) > 0$,
independent of $i$, such that
\begin{align*}
  & \frac{1}{2} \| \rho_i (\NN) \varphi^m (t) \|^2 + \int_0^t \|
  \rho_i (\NN) (-\LL_0)^{1 / 2} \varphi^m (s) \|^2 \dd s\\
  & \hspace{20pt} \leqslant \frac{1}{2} \| \rho_i (\NN) \varphi^m_0
  \|^2 + \int_0^t \delta \sum_{j \sim i} \| \rho_j (\NN)
  (-\LL_0)^{1 / 2} \varphi^m (s) \|^2 \dd s + \int_0^t C \sum_{j
  \sim i} \| \rho_j (\NN) \varphi^m (s) \|^2 \dd s.
\end{align*}
From here we get for $\alpha \in \R$ and a new $C = C (\delta, \alpha) >
0$
\begin{align*}
  & \frac{1}{2} \sum_{i \geqslant - 1} 2^{2 i \alpha} \| \rho_i (\NN)
  \varphi^m (t) \|^2 + \int_0^t \sum_{i \geqslant - 1} 2^{2 i \alpha} \|
  \rho_i (\NN) (-\LL_0)^{1 / 2} \varphi^m (s) \|^2 \dd s\\
  & \hspace{30pt} \leqslant \frac{1}{2} \sum_{i \geqslant - 1} 2^{2 i \alpha}
  \| \rho_i (\NN) \varphi^m_0 \|^2 + \delta \int_0^t \sum_{i
  \geqslant - 1} 2^{2 i \alpha} \| \rho_i (\NN) (-\LL_0)^{1 /
  2} \varphi^m (s) \|^2 \dd s\\
  & \hspace{30pt} \quad + C \int_0^t \frac{1}{2} \sum_{i \geqslant - 1} 2^{2
  i \alpha} \| \rho_i (\NN) \varphi^m (s) \|^2 \dd s,
\end{align*}
and taking $\delta < 1$ we deduce the following bounds:

\begin{lemma}
  \label{lem:bw-apriori}For all $\alpha \in \R$ there exists $C = C
  (\alpha) > 0$ such that
  \begin{equation}
    \label{eq:bw-apriori-1} \sum_{i \geqslant - 1} 2^{2 i \alpha} \| \rho_i
    (\NN) \varphi^m (t) \|^2 \leqslant e^{t C} \sum_{i \geqslant - 1}
    2^{2 i \alpha} \| \rho_i (\NN) \varphi^m_0 \|^2,
  \end{equation}
  as well as
  \begin{equation}
    \label{eq:bw-apriori-2} \int_0^{\infty} e^{- t C} \sum_{i \geqslant - 1}
    2^{2 i \alpha} \| \rho_i (\NN) (-\LL_0)^{1 / 2} \varphi^m
    (t) \|^2 \dd t \leqslant \sum_i 2^{2 i \alpha} \| \rho_i (\NN)
    \varphi^m_0 \|^2 .
  \end{equation}
\end{lemma}

\begin{proof}
  The first bound follows from our previous estimates and Gronwall's lemma.
  For the second bound, observe that
  \[ \partial_t \left( e^{- t C} \frac{1}{2} \| w (\NN) \varphi^m (t)
     \|^2 \right) = e^{- t C} \partial_t \frac{1}{2} \| w (\NN)
     \varphi^m (t) \|^2 - C e^{- t C} \frac{1}{2} \| w (\NN) \varphi^m
     (t) \|^2, \]
  and thus our estimates from above yield
  \begin{align*}
    & e^{- t C} \frac{1}{2} \sum_{i \geqslant - 1} 2^{2 i \alpha} \| \rho_i
    (\NN) \varphi^m (t) \|^2 + \int_0^t e^{- s C} \sum_{i \geqslant -
    1} 2^{2 i \alpha} \| \rho_i (\NN) (-\LL_0)^{1 / 2}
    \varphi^m (s) \|^2 \dd s\\
    & \hspace{30pt} \leqslant \frac{1}{2} \sum_{i \geqslant - 1} 2^{2 i
    \alpha} \| \rho_i (\NN) \varphi^m_0 \|^2 + \delta \int_0^t e^{-
    s C} \sum_{i \geqslant - 1} 2^{2 i \alpha} \| \rho_i (\NN)
    (-\LL_0)^{1 / 2} \varphi^m (s) \|^2 \dd s.
  \end{align*}
  Then we take $\delta = 1/2$, bring the integral from the right to the left, and send $t \rightarrow \infty$ to deduce
  (\ref{eq:bw-apriori-2}).
\end{proof}

\begin{remark}
The norms appearing in the previous lemma can be brought to a more familiar
``Sobolev'' form with the help of the following simple result: For all $\alpha
\in \R$ and $\varphi \in \Gamma L^2$ we have
\begin{align*}
  \sum_{i \geqslant - 1} 2^{2 i \alpha} \| \rho_i (\NN) \varphi \|^2 &
  = \sum_{i \geqslant - 1} 2^{2 i \alpha} \sum_{n \geqslant 0} n! \rho_i (n)^2
  \| \varphi_n \|_{L^2 (\T^n)}^2 \simeq \sum_{n \geqslant 0} n! (1 +
  n)^{2 \alpha} \sum_{i \geqslant - 1} \rho_i (n)^2 \| \varphi_n \|_{L^2
  (\T^n)}^2\\
  & \simeq \sum_{n \geqslant 0} n! (1 + n)^{2 \alpha} \| \varphi_n \|_{L^2
  (\T^n)}^2 = \| (1 + \NN)^{\alpha} \varphi \|^2,
\end{align*}
where we used that $\sum_i \rho_i^2 (n) \simeq 1$. The reason for not directly
working with this Sobolev type norm is that the dyadic partition of unity
allows us to localize in $n$ and therefore to rigorously justify the
operations on $\GG^m_+$ and $\GG^m_-$ above. Compared to a
hard cutoff, the smooth dyadic partition has the advantage that the transition
from the support of $\rho_i$ to its complement is well behaved, while for a
hard cutoff it gives a too large contribution and we cannot close our
estimates.
\end{remark}

\begin{corollary}
  \label{cor:bw-apriori-2}We have for $\varphi^m$, $\alpha$, and $C$ as
  in Lemma~\ref{lem:bw-apriori}
  \begin{equation}
    \label{eq:bw-apriori-3} \| (1 + \NN)^{\alpha} \partial_t \varphi^m
    (t) \|^2 = \| (1 + \NN)^{\alpha} \LL^m \varphi^m (t) \|^2
    \lesssim e^{t C} \| (1 + \NN)^{\alpha} \LL^m \varphi^m_0
    \|^2
  \end{equation}
  and
  \begin{equation}
    \label{eq:bw-apriori-4} \| (1 + \NN)^{\alpha} (-\LL_0)^{1
    / 2} \varphi^m (t) \|^2 \lesssim t e^{t C} \| (1 + \NN)^{\alpha}
    \LL^m \varphi^m_0 \|^2 + \| (1 + \NN)^{\alpha}
    (-\LL_0)^{1 / 2} \varphi^m_0 \|^2 .
  \end{equation}
\end{corollary}

\begin{proof}
  Recall that $T^m_t \varphi^m_0 = \varphi^m(t)$. We just showed
  \[ e^{- t C} \| (1 + \NN)^{\alpha} T^m_t \varphi^m_0 \| +
     \int_0^{\infty} e^{- t C} \| (1 + \NN)^{\alpha}
     (-\LL_0)^{1 / 2} T^m_t \varphi^m_0 \|^2 \dd t \lesssim \|
     (1 + \NN)^{\alpha} \varphi^m_0 \|, \]
  and therefore also
  \begin{align*}
    \int_0^{\infty} e^{- t C} \| (1 + \NN)^{\alpha}
    (-\LL_0)^{1 / 2} \partial_t T^m_t \varphi^m_0 \|^2 \dd t & =
    \int_0^{\infty} e^{- t C} \| (1 + \NN)^{\alpha}
    (-\LL_0)^{1 / 2} T_t^m \LL^m \varphi^m_0 \|^2 \dd
    t\\
    & \lesssim \| (1 + \NN)^{\alpha} \LL^m \varphi^m_0
    \|^2,
  \end{align*}
  and
  \begin{align*}
    & \| (1 + \NN)^{\alpha} (-\LL_0)^{1 / 2} T^m_t \varphi^m_0 \|^2  \\
    &\hspace{40pt} \lesssim \left\| \int_0^t (1 + \NN)^{\alpha}
    (-\LL_0)^{1 / 2} \partial_s T_s^m \varphi^m_0 \dd s
    \right\|^2 + \| (1 + \NN)^{\alpha} (-\LL_0)^{1 / 2}
    \varphi^m_0 \|^2\\
    &\hspace{40pt} \leqslant t \int_0^t \| (1 + \NN)^{\alpha} (-\LL_0)^{1
    / 2} \partial_s T_s^m \varphi^m_0 \|^2 \dd s + \| (\NN+
    1)^{\alpha} (-\LL_0)^{1 / 2} \varphi^m_0 \|^2\\
    &\hspace{40pt} \leqslant t e^{t C} \int_0^t e^{- s C} \| (1 + \NN)^{\alpha}
    (-\LL_0)^{1 / 2} \partial_s T^m_s \varphi^m_0 \|^2 \dd s + \| (1 + \NN)^{\alpha} (-\LL_0)^{1 / 2}
    \varphi^m_0 \|^2\\
    &\hspace{40pt} \lesssim t e^{t C} \| (1 + \NN)^{\alpha} \LL^m
    \varphi^m_0 \|^2 + \| (1 + \NN)^{\alpha} (-\LL_0)^{1 /
    2} \varphi^m_0 \|^2,
  \end{align*}
  which is the claimed estimate.
\end{proof}

\subsection{Controlled solutions}\label{sec:bw-controlled}

The a priori bounds (\ref{eq:bw-apriori-3}) and (\ref{eq:bw-apriori-4}) allow
us to control $\| \varphi^m (t) \|$, $\| \partial_t \varphi^m (t) \|$, and $\| \LL^m \varphi^m (t) \|$
uniformly in $m$ and locally uniformly in $t$. We want to use this to construct solutions of the limiting backward equation $\partial_t
\varphi =\LL \varphi$ that are in the domain $\CD
(\LL)$ from Section~\ref{sec:controlled}.
Therefore, let us set
\begin{equation}\label{eq:phi-m-sharp-def}
	\varphi^{m, \sharp} \assign \varphi^m - (-\LL_0)^{- 1}
   \GG^{m, \succ} \varphi^m,
\end{equation}
so that $\varphi^m =\KK^m \varphi^{m, \sharp}$.

\begin{convention}
	Throughout this section we consider a cutoff $N_n$ in Lemma~\ref{lem:controlled-fct} that is adapted to the weights $(1+\NN)^\beta$ for all $\beta$ that we encounter below.	
\end{convention}

\begin{lemma}
  \label{lem:phi-sharp-apriori}The a priori estimates from the previous
  section give
  \begin{equation}\label{eq:phi-sharp-apriori}
    \| (1 + \NN)^{\alpha} (-\LL_0)^{1 / 2} \varphi^{m, \sharp}
    (t) \| \lesssim (t e^{t C} + 1)^{1/2} (\| (1 + \NN)^{\alpha} (-\LL_0) \varphi^{m,\sharp}_0 \| + \| (1 + \NN)^{\alpha+9/2} (-\LL_0)^{1/2} \varphi^{m,\sharp}_0 \| ).
  \end{equation}
\end{lemma}

\begin{proof}
It follows from~\eqref{eq:bw-apriori-4} and Lemma~\ref{lem:controlled-fct} that
  \begin{align*}
  	\| (1 + \NN)^{\alpha} (-\LL_0)^{1 / 2} \varphi^{m, \sharp} (t) \|^2 & \lesssim \| (1 + \NN)^{\alpha} (-\LL_0)^{1 / 2} \varphi^{m} (t) \|^2 + \| (1 + \NN)^{\alpha} (-\LL_0)^{- 1/2} \GG^{m, \succ} \varphi^m (t) \|^2 \\
  	& \lesssim t e^{t C} \| (1 + \NN)^{\alpha} \LL^m \varphi^m_0 \|^2 + \| (1 + \NN)^{\alpha} (-\LL_0)^{1 / 2} \varphi^m_0 \|^2 \\
  	& \lesssim (t e^{t C} + 1) (\| (1 + \NN)^{\alpha} (-\LL_0) \varphi^{m,\sharp}_0 \|^2 + \| (1 + \NN)^{\alpha+9/2} (-\LL_0)^{1/2} \varphi^{m,\sharp}_0 \|^2 ),
  \end{align*}
  where in the last step we applied Proposition~\ref{prop:domain}.
\end{proof}

Unfortunately this estimate is not enough to show that $\varphi^m \in
\CD (\LL^m)$, which requires a bound on $\| (-\LL_0)
\varphi^{m, \sharp} \| + \| (1 + \NN)^{9 / 2} (-\LL_0)^{1 / 2}
\varphi^{m, \sharp} \|$. And in fact we will need even more regularity to deduce compactness in the right spaces. So let us analyze the equation for $\varphi^{m,
\sharp}$:
\begin{equation}\label{eq:phi-m-sharp-eq}
\begin{aligned}
  \partial_t \varphi^{m, \sharp} & =\LL^m \varphi^m -
  (-\LL_0)^{- 1} \GG^{m, \succ} \partial_t \varphi^m\\
  & =\LL_0 \varphi^{m, \sharp} + \GG^{m,\prec} \varphi^m - (-\LL_0)^{- 1}
  \GG^{m,\succ} \partial_t \varphi^m .
\end{aligned}
\end{equation}
The second term on the right hand side can be controlled with~\eqref{eq:domain}, which gives for $\gamma \geqslant 0$ and $\delta > 0$
\[
	\|(1+\NN)^\alpha (-\LL_0)^\gamma \GG^{m,\prec} \varphi^m \| \lesssim \| (1+\NN)^{\alpha + \alpha(\gamma)} (-\LL_0)^{1/4+\delta} \varphi^{m,\sharp} \|,
\]
so together with our a priori bound~\eqref{eq:phi-sharp-apriori} we get
\begin{equation}\label{eq:bw-controlled-1}
\begin{aligned}
	\sup_{t \in [0,T]} \|(1+\NN)^\alpha (-\LL_0)^\gamma \GG^{m,\prec} \varphi^m(t) \| \lesssim_T & \| (1 + \NN)^{\alpha + \alpha(\gamma)} (-\LL_0) \varphi^{m,\sharp}_0 \| \\
	& + \| (1 + \NN)^{\alpha+\alpha(\gamma)+9/2} (-\LL_0)^{1/2} \varphi^{m,\sharp}_0 \|.
\end{aligned}
\end{equation}
The remaining term $(-\LL_0)^{- 1}\GG^{m,\succ} \partial_t \varphi^m$ is more tricky. We can plug in the explicit form of the time derivative, $\partial_t \varphi^m = \GG^{m,\prec} \varphi^m + \LL_0 \varphi^{m,\sharp}$, but then we have a problem with the term $\LL_0 \varphi^{m,\sharp}$ because it is of the same order as the leading term of the equation for $\varphi^{m,\sharp}$. Therefore, we would like to gain a bit of regularity in $(-\LL_0)$ from $(-\LL_0)^{- 1}\GG^{m,\succ}$, and indeed this is possible by slightly adapting the proof of Lemma~\ref{lem:controlled-fct}; see Lemma~\ref{lem:controlled-fct-adapted} in the appendix for details. This gives for $\gamma \in (1/2, 3/4)$
\begin{align*}
	\| (1 + \NN)^\alpha (-\LL_0)^\gamma (-\LL_0)^{- 1} \GG^{m,\succ} \partial_t \varphi^m \| & \lesssim \| (1 + \NN)^{\alpha + 3/2} (-\LL_0)^{\gamma-1/4} (\GG^{m,\prec} \varphi^m + \LL_0 \varphi^{m,\sharp}) \| \\
	& \lesssim \| (1 + \NN)^{\alpha + 3/2 + \alpha(\gamma-1/4)} (-\LL_0)^{1/4+\delta}  \varphi^{m,\sharp} \| \\
	&\quad + \|(1+\NN)^{\alpha+3/2} (-\LL_0)^{\gamma + 3/4} \varphi^{m,\sharp}\|.
\end{align*}
Recall that $\alpha(\gamma) = 9/2 + 7\gamma$, and therefore $3/2 + \alpha(\gamma-1/4) \leqslant \alpha(\gamma)$ and the first term on the right hand side is bounded by the same expression as in~\eqref{eq:bw-controlled-1}. For the remaining term we apply Young's inequality for products: There exists $p>0$ such that for all $\varepsilon \in (0,1)$
\begin{equation}\label{eq:bw-controlled-2}
	 \|(1+\NN)^{\alpha+3/2} (-\LL_0)^{\gamma + 3/4} \varphi^{m,\sharp}\| \lesssim \varepsilon^{-p} \| (1+\NN)^p (-\LL_0)^{1/2} \varphi^{m,\sharp} \| + \varepsilon \|(1+\NN)^\alpha (-\LL_0)^{\gamma + 7/8} \varphi^{m,\sharp}\|.
\end{equation}
The first term on the right hand side is under control by our a priori estimates, and the second term on the right hand side can be estimated using the regularizing effect of the semigroup $(S_t)$ generated by $\LL_0$:

\begin{lemma}
	Let $\gamma \in (3/8, 5/8)$. There exists $p=p(\alpha,\gamma)$ such that for all $T>0$
	\begin{equation}
		\sup_{t \in [0,T]} \big(\|(1+\NN)^\alpha (-\LL_0)^{1+\gamma} \varphi^{m,\sharp}(t) \| + \|(1+\NN)^\alpha (-\LL_0)^{\gamma} \partial_t \varphi^{m,\sharp}(t) \|\big)\lesssim_T \|(1+\NN)^{p} (-\LL_0)^{1+\gamma} \varphi^{m,\sharp}_0\|.
	\end{equation}
\end{lemma}

\begin{proof}
	The variation of constants formula gives $\varphi^{m,\sharp}(t) = S_t \varphi^{m,\sharp}_0 + \int_0^t S_{t-s} (\partial_s - \LL_0) \varphi^{m,\sharp}(s) \dd s$, and by writing the explicit representation of $S_t$ and $\LL_0$ in Fourier variables we easily see that
	\begin{equation*}
		\| (1+\NN)^\alpha (-\LL_0)^\beta S_t \psi \| \lesssim t^{-\beta} \|(1+\NN)^\alpha \psi \|	
	\end{equation*}
	for all $\beta \geqslant 0$. Since $\gamma + 1/8 \in (1/2, 3/4)$ we can combine this with our previous estimates, and in that way we obtain for some $K, K_T > 0$ and for $t \in [0,T]$
	\begin{align*}
		\|(1+\NN)^\alpha (-\LL_0)^{1+\gamma} \varphi^{m,\sharp}(t) \| & \lesssim \|(1+\NN)^\alpha (-\LL_0)^{1+\gamma} \varphi^{m,\sharp}_0 \| \\
		&\quad + \int_0^t (t-s)^{-1+1/8} \|(1+\NN)^\alpha (-\LL_0)^{\gamma + 1/8} (\partial_s - \LL_0) \varphi^{m,\sharp}(s) \| \dd s \\
		& \leqslant K \|(1+\NN)^\alpha (-\LL_0)^{1+\gamma} \varphi^{m,\sharp}_0 \|  + K_T (1+\varepsilon^{-p}) \| (1 + \NN)^p (-\LL_0) \varphi^{m,\sharp}_0\| \\
		& \quad + K T^{1/8} \varepsilon \sup_{s \in [0,T]} \|(1 + \NN)^\alpha (-\LL_0)^{1+\gamma} \varphi^{m,\sharp}(s) \|.
	\end{align*}
	The right hand side does not depend on $t$, and therefore we can take the supremum over $t \in [0,T]$, and then we choose $\varepsilon > 0$ small enough so that $K T^{1/8} \varepsilon \leqslant 1/2$ and we bring the last term on the right hand side to the left and thus we obtain the claimed bound for the spatial regularity. For the temporal regularity, i.e. for $\partial_t \varphi^{m,\sharp}$, we simply use that
	\[
		\partial_t \varphi^{m,\sharp} = \LL_0 \varphi^{m,\sharp} + (\partial_t - \LL_0) \varphi^{m,\sharp}
	\]
	and apply the previous bounds to the two terms on the right hand side.
\end{proof}

For $s, t \in [0, T]$ we now interpolate the
two estimates
\[ \| (1 + \NN)^{\alpha} (-\LL_0)^\gamma (\varphi^{m, \sharp} (t) - \varphi^{m, \sharp}
   (s)) \| \lesssim_T | t - s | \times \| (1 + \NN)^{p}
   (-\LL_0)^{1 + \gamma} \varphi^{m, \sharp}_0 \| \]
and
\[ \| (1 + \NN)^{\alpha} (-\LL_0)^{1 + \gamma} (\varphi^{m,
   \sharp} (t) - \varphi^{m, \sharp} (s)) \| \lesssim_T \| (1 + \NN)^{p} (-\LL_0)^{1 + \gamma} \varphi^{m, \sharp}_0
   \| \]
to obtain some $\kappa \in (0,1)$ such
that
\[ \| (1 + \NN)^{\alpha} (-\LL_0)^{1 + \gamma/2 }
   (\varphi^{m, \sharp} (t) - \varphi^{m, \sharp} (s)) \| \lesssim | t - s
   |^{\kappa} \times \| (1 + \NN)^{p} (-\LL_0)^{1 +
   \gamma} \varphi^{m, \sharp}_0 \| . \]
   
In conclusion, if for $\alpha>0$ we introduce the set
\begin{equation}
\label{eq:kol-set}
\CK_\alpha := \bigcup_{\gamma \in (3/8, 5 / 8)} \KK (1 + \NN)^{- p(\alpha,\gamma)}
  (-\LL_0)^{- 1 - \gamma} \Gamma L^2 \subseteq \Gamma L^2
\end{equation}   
where $p(\alpha,\gamma)$ is as above, then we can state the existence of strong solutions to the
Kolmogorov backward equation for initial conditions in $\CK := \CK_{9/2+} :=\bigcup_{\alpha>9/2}\CK_{\alpha}$:

\begin{theorem}
  \label{thm:bw-eq}Let $\alpha \geqslant 0$ and $\varphi_0 \in \CK_\alpha$. Then there exists a solution 
  \[
  		\varphi \in \bigcup_{\delta > 0} C (\R_+, (1 + \NN)^{- \alpha + \delta} (-\LL_0)^{- 1 } \Gamma L^2)
  \]
  of the backward equation
  \begin{equation}\label{eq:bw-eq}
  	\partial_t \varphi =\LL \varphi
  \end{equation}
  with initial condition
  $\varphi (0) =\varphi_0$. For $\alpha > 9/2$ we have $\varphi \in C(\R_+, \CD(\LL)) \cap C^1(\R, \Gamma L^2)$ and by dissipativity of $\LL$ this solution is unique. 
  \end{theorem}
  
\begin{proof} Take $\varphi_0 \in \CK_\alpha$ and
denote $\varphi^{\sharp}_0 = \KK^{-1}\varphi_0 \in (1 + \NN)^{- p}
  (-\LL_0)^{- 1 - \gamma} \Gamma L^2$  for some 
$\gamma \in (3/8, 5 / 8)$ and $p = p(\alpha,\gamma)$. 
Consider for $m \in \N$ the solution $\varphi^m$ to $\partial_t \varphi^m
  =\LL^m \varphi^m$ with initial condition $\varphi^m (0)
  =\KK^m \varphi^{\sharp}_0$. It follows from a diagonal sequence argument that bounded sets in $(1 + \NN)^{-\alpha} (-\LL_0)^{-1-\gamma/2} \Gamma L^2$ are relatively compact in $(1 + \NN)^{-\alpha+\delta} (-\LL_0)^{-1} \Gamma L^2$ for $\delta >0$, and thus $(\varphi^{m,\sharp})_m$ is relatively compact in $C (\R_+, (1 + \NN)^{- \alpha + \delta} (-\LL_0)^{- 1 } \Gamma L^2)$ (equipped with the topology of uniform convergence on compacts) by the infinite-dimensional version of the Arzel\`a-Ascoli theorem. If $\varphi^{\sharp}$ is a limit point we let
  $\varphi =\KK \varphi^{\sharp}$.
  To see that $\partial_t \varphi =\LL
  \varphi$, note that (along the convergent subsequence, which we omit from
  the notation for simplicity)
  \begin{align*}
    \varphi(t) - \varphi(0) & = \lim_{m \rightarrow \infty} (\varphi^m(t) - \varphi^m(0)) =
    \lim_{m \rightarrow \infty} \int_0^t \LL^m \varphi^m(s) \dd s \\
    & = \lim_{m \rightarrow
    \infty} \int_0^t (\LL_0 \varphi^{m, \sharp}(s) +\GG^{m,\prec} \KK^m \varphi^{m,\sharp}(s)) \dd s \\
    & = \lim_{m \rightarrow
    \infty} \int_0^t (\LL_0 \varphi^{\sharp}(s) +\GG^{m,\prec} \KK^m \varphi^{\sharp}(s)) \dd s \\
    & = \int_0^t (\LL_0 \varphi^{\sharp}(s) +\GG^{\prec} \KK \varphi^{\sharp}(s)) \dd s,
  \end{align*}
  where the second-to-last step follows from our uniform bounds on
  $\LL_0, \GG^{m,\prec}, \KK^m$ and the convergence of $\varphi^{m, \sharp}$ to $\varphi^{\sharp}$, and the last
  step follows from our bounds for $\GG^{\prec}, \KK$ together with the dominated convergence theorem. If $\alpha > 9/2$, then $\varphi \in \CD(\LL)$ by definition, see Lemma~\ref{lem:dom-def}. Moreover, in that case $\LL \varphi \in C(\R_+, \Gamma L^2)$ and since $\varphi(t) - \varphi(s) = \int_s^t \LL \varphi(r) \dd r$ we get $\varphi \in C^1(\R_+, \Gamma L^2)$. In this last case we can compute as follows 
  $$
  \partial_t \|\varphi(t)\|^2 =2  \langle\varphi(t),\LL \varphi(t)  \rangle \leqslant 0,
  $$
  using the dissipativity of $\LL$ (Lemma~\ref{lem:dissipative}). Therefore we conclude that for any solution we have $\|\varphi(t)\|\leqslant \|\varphi_0\|$ which together with the linearity of the equation gives uniqueness.
\end{proof}

\begin{remark}
	We focused on the backward equation, but by similar (and actually slightly easier) arguments we can also solve the resolvent equation $(\lambda - \LL) \varphi = \psi$ for $\lambda > C/2$ and $\psi \in \CK$, where $C>0$ is the constant from Corollary~\ref{cor:bw-apriori-2}. Since $\CK$ and $\CD(\LL)$ are dense and $\LL$ is dissipative by Lemma~\ref{lem:dissipative}, it  follows from Theorem~1.2.12 of \cite{Ethier1986} that $\LL$ generates a strongly continuous contraction semigroup on $L^2(\mu)$. Then we can apply Kolmogorov's extension theorem to construct, for all initial distributions with $L^2$ density with respect to $\mu$, a Markov process corresponding to this semigroup. However, it seems a bit subtle how to get the continuity of trajectories or the link with the martingale problem in this way. To be in the setting of~\cite{Ethier1986} we would need a semigroup on $C_b(E)$ for a locally compact and separable state space $E$, but since we are in infinite dimensions our state space cannot be locally compact. A canonical state space would be $H^{-1/2-}(\T)$, but it also seems difficult to show that $T_t$ maps $C_b(H^{-1/2-})$ to itself, let alone that it defines a semigroup on that space. So instead we will construct the process directly by a tightness argument based on the martingale problem.
\end{remark}

\section{The martingale problem}\label{sec:mp}

\begin{definition}\label{def:mart-prob}
We say that a process $(u_t)_{t \geqslant 0}$ with trajectories in $C
(\R_+, \CS')$, where $\CS'$ are the Schwartz
distributions on $\T$, \emph{solves the martingale problem for
$\LL$ with initial distribution $\nu$} if $u_0 \sim \nu$, if
$\tmop{law} (u_t) \ll \eta$ for all $t \geqslant 0$, and if for all $\varphi
\in \CD (\LL)$ and $t \geqslant 0$ we have $\int_0^t | \LL \varphi (u_s)
| \dd s < \infty$ almost surely and the process
\[ \varphi (u_t) - \varphi (u_0) - \int_0^t \LL \varphi (u_s) \dd
   s, \qquad t \geqslant 0, \]
is a martingale in the filtration generated by $(u_t)$. Note that since
$\varphi$ and $\LL \varphi$ are not cylinder functions we need the
condition $\tmop{law} (u_t) \ll \eta$ in order for $\varphi (u_t)$ and
$\LL \varphi (u_t)$ to be well defined.
\end{definition}

Due to our lack of  control for $\LL\varphi$ outside of $\Gamma L^2$, the following class of processes will play a major role in our study of the martingale problem.
\begin{definition}\label{def:incompressible}
We say that a process $(u_t)_{t\geqslant 0}$ with values in $\CS'$ is  \emph{incompressible} if for all $T>0$ there exists $C(T) > 0$ such that for all $\varphi \in \Gamma L^2$
\[
\sup_{t \le T} \E[|\varphi(u_t)|] \leqslant C(T) \| \varphi \|.
\]
\end{definition}

We will establish the existence of incompressible solutions to the martingale problem  by a compactness argument. The duality of martingale problem and backward equation gives uniqueness of incompressible solutions to the martingale problem. Since the domain of $\LL$ is rather complicated, we then study a ``cylinder function martingale problem'', a generalization of the energy solutions of~\cite{Goncalves2014,Gubinelli2013,Gubinelli2018Energy}, and we show that every solution to the cylinder function martingale problem solves the martingale problem for $\LL$ and in particular its law is unique.

\subsection{Existence of solutions}\label{sec:mp-ex}

In the following we show that under ``near-stationary'' initial
conditions the Galerkin approximations $(u^m)_m$ solving~\eqref{eq:burgers-galerkin} are tight in $C
(\R_+, \CS')$, and that any weak limit is an incompressible solution to the martingale
problem for the generator $\LL$ in the sense of Definitions~\ref{def:mart-prob} and~\ref{def:incompressible}. 
The following elementary inequality
will be used throughout this section:

\begin{lemma}
  \label{lem:switch-measure}Let $u^m$ be a solution
  to~(\ref{eq:burgers-galerkin}) with $\tmop{law} (u^m_0) \ll \mu$ with
  density $\eta \in L^2 (\mu)$. Then we have for any $\Psi\colon C (\R_+,
  \CS') \rightarrow \R$
  \[ \E [\Psi (u^m)] \leqslant \| \eta \| \E_{\mu} [\Psi
     (u^m)^2]^{1 / 2}, \]
  where $\P_{\mu}$ denotes the distribution of $u^m$ under the
  stationary initial condition $u^m_0 \sim \mu$. In particular $u^m$ is incompressible.
\end{lemma}

\begin{proof}
  The Cauchy-Schwarz inequality and Jensen's inequality yield
  \[ \E [\Psi (u^m)] = \int \E_u [\Psi (u^m)] \eta (u) \mu
     (\dd u) \leqslant \| \eta \| \left( \int \E_u [\Psi (u^m)]^2
     \mu (\dd u) \right)^{1 / 2} \leqslant \| \eta \| \E_{\mu}
     [\Psi (u^m)^2]^{1 / 2} . \]
  
\end{proof}

Recall that $D_x$ denotes the Malliavin derivative.

\begin{lemma}
  Let $u^m$ be a solution to~(\ref{eq:burgers-galerkin}) with $\tmop{law}
  (u^m_0) \ll \mu$ with density $\eta \in L^2 (\mu)$. Let $\varphi \in \CD (\LL^m)$ and consider $M^{m, \varphi}_t \assign \varphi (u^m_t) - \varphi (u^m_0) - \int_0^t
     \LL^m \varphi (u^m_s) \dd s$. Then $M^{m, \varphi}$ is a continuous martingale with quadratic variation
  \begin{equation}
    \langle M^{m, \varphi} \rangle_t = \int_0^t \CE \varphi (u_m (s))
    \dd s, \qquad \tmop{where} \qquad \CE \varphi = 2
    \int_{\T} | \partial_x D_x \varphi |^2 \dd x.
  \end{equation}
  Moreover, for $w \colon \N_0 \rightarrow \R_+$ we have
  \begin{equation}\label{eq:martingale-qv-bound}
    \| w (\NN) (\CE \varphi)^{1 / 2} \| = \sqrt{2} \| w
    (\NN- 1) (-\LL_0)^{1 / 2} \varphi \|.
  \end{equation}
\end{lemma}

\begin{proof}
  For cylinder functions $\varphi$ the claim follows from It{\^o}'s formula,
  and in that case the Burkholder-Davis-Gundy inequality gives for all $T > 0$
  \[ \E [\sup_{t \leqslant T} | M^{m, \varphi}_t |] \lesssim
     \E [\langle M^{m, \varphi} \rangle_T^{1 / 2}] \leqslant \| \eta
     \| \E_{\mu} [\langle M^{m, \varphi} \rangle_T]^{1/2} = \| \eta \| T^{1/2} \|
     (\CE \varphi)^{1 / 2} \|. \]
  The ``energy'' on the right hand side can be computed as
  \begin{align*}
    \| w (\NN) (\CE \varphi)^{1 / 2} \|^2 & = 2
    \int_{\T} \| w (\NN) \partial_x D_x \varphi \|^2 \dd
    x\\
    & = 2 \int_{\T} \left( \sum_{n = 1}^{\infty} (n - 1) !w (n -
    1)^2 n^2 \| \partial_x \varphi_n (x, r_{2 : n}) \|_{L^2_r (\T^{n
    - 1})}^2 \right) \dd x\\
    & = 2 \sum_{n = 1}^{\infty} n!w (n - 1)^2 n \sum_{k_{1 : n}} | 2 \pi k_1
    |^2 | \hat{\varphi}_n (k_{1 : n}) |^2\\
    & = 2 \sum_{n = 1}^{\infty} n!w (n - 1)^2 \sum_{k_{1 : n}} (| 2 \pi k_1
    |^2 + \cdots + | 2 \pi k_n |^2) | \hat{\varphi}_n (k_{1 : n}) |^2\\
    & = 2 \| w (\NN- 1) (-\LL_0)^{1 / 2} \varphi \|^2,
  \end{align*}
  so since $\CD (\LL^m) \subset (-\LL_0)^{- 1/2}(1 + \NN)^{-1}
  \Gamma L^2$ and cylinder functions are dense in $(-\LL_0)^{- 1/2} (1 + \NN)^{-1}
  \Gamma L^2$ by Proposition~\ref{prop:domain}, we deduce that if $(\varphi^M)_M \subset \CC$ converges
  in $(-\LL_0)^{- 1/2} (1 + \NN)^{-1} \Gamma L^2$ to $\varphi \in \CD
  (\LL^m)$, then $M^{m, \varphi^M}$ converges to a continuous
  martingale $M^{m, \varphi}$ with quadratic variation $\langle M^{m, \varphi}
  \rangle_t = \int_0^t \CE \varphi (u_m (s)) \dd s$. On the other
  hand it follows from the bounds in Lemma~\ref{lem:G-apriori} that $\varphi^M
  (u^m_t) - \varphi^M (u^m_0) - \int_0^t \LL^m \varphi^M (u^m_s)
  \dd s$ converges to $\varphi (u^m_t) - \varphi (u^m_0) - \int_0^t
  \LL^m \varphi (u^m_s) \dd s$, and therefore
  \[ M^{m, \varphi}_t = \varphi (u^m_t) - \varphi (u^m_0) - \int_0^t
     \LL^m \varphi (u^m_s) \dd s. \]
\end{proof}

To prove tightness we need to control higher moments, and for this purpose the following classical result is useful.

\begin{remark}
  \label{rmk:higher-moments}Let $p \geqslant 2$ and define $c_p \assign \sqrt{p - 1}$.
  It follows from the hypercontractivity of the Ornstein-Uhlenbeck semigroup,
  see \cite[Theorem 1.4.1]{Nualart2006}, that $\| | \varphi |^{p / 2} \|^2 \leqslant \|
  c_p^{\NN} \varphi \|^p$ for all $\varphi \in \Gamma L^2$.
\end{remark}

In Lemma~\ref{lem:dom-def} we defined a domain $\CD_w(\LL)$ of functions that are mapped to $w(\NN)^{-1} \Gamma L^2$ by $\LL$. If $w (n) = c_p^n$ for the constant $c_p > 0$ of Remark~\ref{rmk:higher-moments}, we write $\CD_p (\LL) := \CD_w(\LL)$ from now on.

\begin{theorem}\label{thm:mp-ex}
  Let $\eta \in L^2 (\mu)$ and let $u^m$ be the solution
  to~(\ref{eq:burgers-galerkin}) with $\tmop{law} (u^m_0) \sim \eta \dd
  \mu$. Then $(u^m)_{m \in \N}$ is tight in $C (\R_+,
  \CS')$ and any weak limit $u$ is incompressible and solves the martingale problem for
  $\LL$ with initial distribution $\eta \dd \mu$.
\end{theorem}

\begin{proof}
  1. We first consider $p \geqslant 2$ and $\varphi \in \CD_{2 p}
  (\LL^m)$ and derive an estimate for $\E [| \varphi (u^m_t)
  - \varphi (u^m_s) |^p]$. For that purpose we split $\varphi (u^m_t) -
  \varphi (u^m_s) = \int_s^t \LL^m \varphi (u^m_r) \dd r + M^{m,
  \varphi}_t - M^{m, \varphi}_s$, and observe that by
  Lemma~\ref{lem:switch-measure} and Remark~\ref{rmk:higher-moments}
  \[ \E \left[ \left| \int_s^t \LL^m \varphi (u^m_r) \dd
     r \right|^p \right] \lesssim \E_{\mu} \left[ \left| \int_s^t
     \LL^m \varphi (u^m_r) \dd r \right|^{2 p} \right]^{1 / 2}
     \leqslant | t - s |^p \| | \LL^m \varphi |^p \| \leqslant | t -
     s |^p \| c_{2 p}^{\NN} \LL^m \varphi \|^p . \]
  The martingale term can be bounded with the help of the
  Burkholder-Davis-Gundy inequality and~\eqref{eq:martingale-qv-bound}:
  \begin{align*}
    \E [| M^{m, \varphi}_t - M^{m, \varphi}_s |^p] & \lesssim
    \E \left[ \left( \int_s^t \CE \varphi (u^m_s) \dd s
    \right)^{p / 2} \right] \lesssim \E_{\mu} \left[ \left( \int_s^t
    \CE \varphi (u^m_s) \dd s \right)^p \right]^{1 / 2} \\
    & \lesssim |  t - s |^{p / 2} \| (\CE \varphi)^{p / 2} \| \leqslant | t - s |^{p / 2} \| c_{2 p}^{\NN} (\CE
    \varphi)^{1 / 2} \|^p \\
    & \lesssim | t - s |^{p / 2} \| c_{2 p}^{\NN} (-\LL_0)^{1 / 2} \varphi \|^p .
  \end{align*}
\noindent  2. Let now $\varphi \in c_{2p}^{-\NN} (-\LL_0)^{- 1} \Gamma L^2 \cap c_{2p}^{-\NN} (1 + \NN)^{- 9 / 2} (-\LL_0)^{- 1 / 2} \Gamma L^2$. We apply Step 1 and~\eqref{eq:quantitative-density} to find for all $M \geqslant 1$ a $\varphi^M \in
  \CD_{2 p} (\LL^m)$ with
  \begin{align*}
    \E [| \varphi (u^m_t) - \varphi (u^m_s) |^p] & \lesssim
    \E [| \varphi (u^m_t) - \varphi^M (u^m_t) |^p] +\E [|
    \varphi (u^m_s) - \varphi^M (u^m_s) |^p] \\
    & \quad +\E [| \varphi^M (u^m_t)    - \varphi^M (u^m_s) |^p]\\
    & \lesssim \| | \varphi - \varphi^M |^p \| + | t - s |^{p / 2} \| c_{2
    p}^{\NN} (-\LL_0)^{1 / 2} \varphi^M \|^p + | t - s |^p \|
    c_{2 p}^{\NN} \LL^m \varphi^M \|^p\\
    & \lesssim \| c_{2 p}^{\NN} (\varphi - \varphi^M) \|^p + | t - s
    |^{p / 2} \| c_{2 p}^{\NN} (-\LL_0)^{1 / 2} \varphi^M
    \|^p + | t - s |^p \| c_{2 p}^{\NN} \LL^m \varphi^M
    \|^p\\
    & \lesssim M^{- p / 2} \| c_{2 p}^{\NN} \varphi \|^p + | t - s
    |^{p / 2} \| c_{2 p}^{\NN} (-\LL_0)^{1 / 2} \varphi
    \|^p\\
    & \quad + | t - s |^p M^{p / 2} (\| c_{2 p}^{\NN}
    (-\LL_0) \varphi \|^p + \| c_{2 p}^{\NN} (1 + \NN)^{9 / 2} (-\LL_0)^{1 / 2} \varphi \|^p) .
  \end{align*}
  For $| t - s | \leqslant 1$ we choose $M = | t - s |^{- 1}$ and see that the
  right hand side is of order $| t - s |^{p / 2}$. The law of the initial
  condition $\varphi (u^m_0)$ does not depend on $m$, and thus it follows from
  Kolmogorov's continuity criterion that the sequence of real valued processes
  $(\varphi (u^m))_m$ is tight in $C (\R_+, \R)$ whenever $p
  > 2$ and $\varphi \in w (\NN)^{- 1} (-\LL_0)^{- 1} \Gamma L^2 \cap w (\NN)^{- 1} (1 + \NN)^{- 9 / 2} (-\LL_0)^{- 1 / 2} \Gamma L^2$.
  This space contains in particular all functions of the form $\varphi (u) = u (f)$ with $f \in
  C^{\infty} (\T)$, where $u (f)$ denotes the application of the
  distribution $u \in \CS'$ to the test function $f$. Therefore, we
  can apply Mitoma's criterion~\cite{Mitoma1983} to deduce that the sequence $(u^m)$ is
  tight in $C (\R_+, \CS')$.
  
  3. It remains to show that any weak limit $u$ of $(u^m)$ solves the
  martingale problem for $\LL$ with initial distribution $\eta \dd
  \mu$. As $u^m_0 \sim \eta \dd \mu$, also any weak limit has initial
  distribution $\eta \dd \mu$. To show that $u$ solves the martingale
  problem, first observe that for any $\varphi \in \Gamma L^2$ we have
  \[ \E [| \varphi (u_t) |] \leqslant \liminf_{m \rightarrow \infty}
     \E [| \varphi (u^m_t) |] \leqslant \liminf_{m \rightarrow
     \infty} \| \eta \| \| \varphi \|, \]
  and therefore we have for any bounded cylinder function $\varphi^M$
  \begin{align*}
     \limsup_{m \rightarrow \infty} | \E [\varphi (u_t)]&
    -\E [\varphi (u^m_t)] | \leqslant  
    \E [| (\varphi - \varphi^M) (u_t) |] + 
    \\ & 
    +\limsup_{m \rightarrow \infty} \left\{  \left| \E
    [\varphi^M (u_t)] -\E [\varphi^M (u_t^m)] \right| +\E [|
    (\varphi - \varphi^M) (u_t^m) |] \right\} \\ &  \lesssim \| \varphi - \varphi^M
    \|,
  \end{align*}
  which shows that the left hand side equals zero because bounded cylinder
  functions are dense in $\Gamma L^2$. The same argument also shows that
  $\limsup_{m \rightarrow \infty} \left| \E \left[ \int_s^t \varphi
  (u_r) \dd r \right] -\E \left[ \int_s^t \varphi (u_r^m) \dd r
  \right] \right| = 0$ and then that for $\varphi \in \CD
  (\LL)$ and $G$ bounded and continuous on $C ([0, s], \CS')$
  \begin{align*}
    & \E \left[ \left( \varphi (u_t) - \varphi (u_s) - \int_s^t
    \LL \varphi (u_r) \dd r \right) G ((u_r)_{r \in [0, s]})
    \right]\\
    & \hspace{20pt} = \lim_{m \rightarrow \infty} \E \left[ \left(
    \varphi (u_t^m) - \varphi (u_s^m) - \int_s^t \LL \varphi (u_r^m)
    \dd r \right) G ((u_r^m)_{r \in [0, s]}) \right] .
  \end{align*}
  This is not quite sufficient, because $u^m$ solves the martingale problem
  for $\LL^m$ and not for $\LL$. But since $\varphi \in
  \CD (\LL)$ there exists $\varphi^{\sharp}$ with $\varphi
  =\KK \varphi^{\sharp}$, so let us define $\varphi^m =\KK^m
  \varphi^{\sharp}$. It follows from the dominated convergence theorem and the
  proof of Lemma~\ref{lem:controlled-fct} that $\| \varphi^m - \varphi \|
  \rightarrow 0$ as $m \rightarrow \infty$. Moreover, $\LL^m
  \varphi^m =\LL_0 \varphi^{\sharp} +\GG^{m, \prec}
  \KK^m \varphi^{\sharp}$, and therefore another application of the dominated
  convergence theorem in the proof of Proposition~\ref{prop:domain} shows that
  $\| \LL^m \varphi^m -\LL \varphi \| \rightarrow 0$. Hence
  \begin{align*}
    & \lim_{m \rightarrow \infty} \E \left[ \left( \varphi (u_t^m) -
    \varphi (u_s^m) - \int_s^t \LL \varphi (u_r^m) \dd r \right) G
    ((u_r^m)_{r \in [0, s]}) \right]\\
    & \hspace{20pt} = \lim_{m \rightarrow \infty} \E \left[ \left(
    \varphi^m (u_t^m) - \varphi^m (u_s^m) - \int_s^t \LL^m \varphi^m
    (u_r^m) \dd r \right) G ((u_r^m)_{r \in [0, s]}) \right] = 0,
  \end{align*}
  which concludes the proof.
\end{proof}

\begin{remark}
  For simplicity we restricted our attention to $\eta \in L^2 (\mu)$. But it
  is clear that the same arguments show the existence of solutions to the martingale problem for initial conditions $\eta \dd \mu$ with $\eta \in L^q (\mu)$ for $q > 1$.
  The key requirement is that we can control expectations of $u^m$ in terms of
  higher moments under the stationary measure $\P_{\mu}$, which also
  works for $\eta \in L^q (\mu)$. The only difference is that for $q < 2$ we
  would have to adapt the definition of incompressibility and restrict our domain in the martingale problem from
  $\CD (\LL)$ to $\CD_{q'} (\LL)$, where
  $q'$ is the conjugate exponent of $q$. On the other hand the uniqueness proof below really needs $\eta \in L^2$ because we can only control the solution to the backward equation in spaces with polynomial weights, but not with exponential weights.
\end{remark}

\subsection{Uniqueness of solutions}

Let $\eta \in \Gamma L^2$ be a probability density (with respect to $\mu$). Let the process $(u_t)_{t \geqslant 0} \in C (\R_+,
\CS')$ be incompressible and solve the martingale problem for $\LL$ with initial
distribution $u_0 \sim \eta \dd \mu$. Here we use the duality of martingale problem and backward equation to show that the law of $u$ is unique and that it is a Markov process with invariant measure $\mu$.

In Lemma~\ref{lem:time-dependent-mp} in the appendix we show that for $\varphi \in C (\R_+, \CD (\LL)) \cap C^1 (\R_+, \Gamma L^2)$ the process $\varphi (t, u_t) - \varphi (0, u_0) - \int_0^t (\partial_s +\LL) \varphi (s, u_s) \dd s$, for $t \geqslant 0$, is a martingale. This will be an important tool in the following theorem:

\begin{theorem}
  \label{thm:mp-uniq}Let $\eta \in \Gamma L^2$ with $\eta \geqslant 0$ and
  $\int \eta \dd \mu = 1$. Let $u$ be an incompressible solution to the martingale
  problem for $\LL$ with initial distribution $u_0 \sim \eta \dd
  \mu$. Then $u$ is a Markov process and its law is
  unique. Moreover, $\mu$ is a stationary measure for $u$.
\end{theorem}

\begin{proof}
  Let $\varphi_0 \in \CK$ and let $\varphi \in C
  (\R_+, \CD (\LL)) \cap C^1 (\R_+, \Gamma
  L^2)$ be the solution to $\partial_t \varphi =\LL \varphi$ with
  initial condition $\varphi (0) =\varphi_0$ that is given
  by Theorem~\ref{thm:bw-eq}. Then we get for $t \geqslant 0$ from Lemma~\ref{lem:time-dependent-mp} that
  \begin{align*}
    \E [\varphi_0 (u_t)] & =\E [\varphi
    (t - t, u_t)] =\E \left[ \varphi (t - 0, u_0) + \int_0^t (- \partial_t
    \varphi (t - s, u_s) +\LL \varphi (t - s, u_s)) \dd s
    \right]\\
    & =\E [\varphi (t, u (0))] = \langle \varphi (t), \eta \rangle
  \end{align*}
  is uniquely determined. Here we used that $\| - \partial_t \varphi (t - s)
  +\LL \varphi (t - s) \| = 0$ implies by assumption also
  $\E [| - \partial_t \varphi (t - s, u_s) +\LL \varphi (t -
  s, u_s) |] = 0$. It is easy to see that $\CK$ is dense in $\CD (\LL)$,
  and since $\CD (\LL)$ is dense in $\Gamma L^2$ and
  $\E [| \psi (u_t) - \tilde{\psi} (u_t) |] \lesssim \| \psi -
  \tilde{\psi} \|$, the law of $u_t$ is uniquely determined.
  
  Next let $\psi_1$ be bounded and measurable and let $\psi_2 \in \CK$. Let $0 \leqslant t_1 <
  t_2$ and let $\partial_t \varphi_2 =\LL \varphi_2$ with initial
  condition $\varphi_2 (0) =\psi_2$. Then
  \begin{align*}
    \E [\psi_1 (u_{t_1}) \psi_2 (u_{t_2})] & =\E [\psi_1
    (u_{t_1}) \varphi_2 (t_2 - t_2, u_{t_2})]\\
    & =\E \left[ \psi_1 (u_{t_1}) \left\{ \varphi_2 (t_2 - t_1,
    u_{t_1}) + \int_{t_1}^{t_2} (- \partial_t +\LL) \varphi_2 (t_2 -
    s, u_s) \dd s \right\} \right]\\
    & =\E [\psi_1 (u_{t_1}) \varphi_2 (t_2 - t_1, u_{t_1})] .
  \end{align*}
  Since we already saw that the law of $u (t_1)$ is uniquely determined, also
  the law of $(u_{t_1}, u_{t_2})$ is unique (by a monotone class argument).
  Iterating this, we get the uniqueness of $\tmop{law} (u_{t_1}, \ldots, u_{t_n})$
  for all $0 \leqslant t_1 < \ldots < t_n$, and therefore the uniqueness of
  $\tmop{law} (u_t : t \geqslant 0)$.
  
  To see the Markov property let $0 \leqslant t < s$, let $X$ be an
  $\CF_t = \sigma (u_r : r \leqslant t)$ measurable bounded random
  variable, and let $\varphi_0 \in \CK$. Then for the solution
  $\varphi$ to the backward equation with initial condition $\varphi (0)
  =\varphi_0$:
  \begin{align*}
    \E [X \varphi_0 (u_s)] & =\E [X
    \varphi (s - s, u_s)] =\E \left[ X \left( \varphi (s - t, u_t) +
    \int_t^s (- \partial_t +\LL) \varphi (s - r, u_r) \dd r
    \right) \right]\\
    &  =\E [X \varphi (s - t, u_t)],
  \end{align*}
  which shows that $\E [\varphi_0 (u_s)
  |\CF_t] = \varphi (s - t, u_t) =\E [\varphi_0 (u_s) |u_t]$. Now the Markov property follows by another
  density argument.
  
  To see that $u$ is stationary with respect to $\mu$ it suffices to consider the specific approximation that we used in the existence proof, i.e. the Galerkin approximation with initial distribution $\mathrm{law}(u^m_0) = \mu$. This is a stationary process and it converges to the solution of the martingale problem, which therefore is also stationary and has initial distribution $\mu$.
\end{proof}

\begin{remark}
The strong Markov property seems difficult to obtain with our tools: If $\tau$ is a stopping
time, then there is no reason why the law of $u_{\tau}$ should be absolutely
continuous with respect to $\mu$, regardless of the initial distribution of
$u$. Since such absolute continuity is crucial for our method, it is not clear
how to deal with $(u_{\tau + t})_{t \geqslant 0}$ (although formally of course
the same arguments as above apply).
\end{remark}

\begin{definition}
We let  $(T_t)_t$ be the semigroup on $\Gamma L^2$ given by, for $\varphi \in \Gamma L^2$,
\[
	\langle T_t \varphi, \eta \rangle = \E_{\eta \dd \mu} [\varphi(u_t)],\qquad \eta \in \Gamma L^2, \eta \geqslant 0, \int \eta \dd \mu = 1,
\]
where $u$ solves the martingale problem for $\LL$ with initial condition $\mathrm{law}(u_0) = \eta \dd \mu$; for more general $\eta \in \Gamma L^2$ we define $\langle T_t \varphi, \eta \rangle$ by linearity, so by the Riesz representation theorem we have indeed $T_t \varphi \in \Gamma L^2$. 
\end{definition}

\begin{proposition}\label{prop:hille-yosida}
The semigroup $(T_t)_t$ is a strongly continuous contraction semigroup on $\Gamma L^2$ and
\[
T_t \varphi = \varphi + \int_0^t T_s \LL \varphi ds,\qquad t\geqslant 0
\]
for all $\varphi \in \CD(\LL)$.
The Hille-Yosida generator $\hat\LL$ of $(T_t)_t$ is an extension of $\LL$.
\end{proposition}
\begin{proof}
Since $\mu$ is stationary for $u$ we have $\|T_t \varphi \| \leqslant \|\varphi\|$ for all $t \geqslant 0$, i.e. $(T_t)$ is a contraction semigroup. From the martingale problem it follows also that for $\varphi \in \CD(\LL)$
\[
T_t \varphi = \varphi + \int_0^t T_s \LL \varphi ds,\qquad t\geqslant 0
\]
and therefore we get the strong continuity in $t$, which by approximation extends to $t \mapsto T_t \psi$ for all $\psi \in \Gamma L^2$. We conclude that $\partial_t T_t \varphi |_{t=0} = \LL \varphi$ and thus $\hat\LL$ is an extension of $\LL$. 
\end{proof}

\subsection{Exponential ergodicity}\label{sec:ergodic}

The Burgers generator formally satisfies a spectral gap estimate and should thus be exponentially $L^2$ ergodic. Indeed, its symmetric part is $\LL_0$ for which the spectral gap is known, and its antisymmetric part $\GG$ should not contribute to the spectral gap estimate, see e.g.~\cite[Definition~2.1]{Guionnet2003}. Having identified a domain for $\LL$, we can make this formal argument rigorous. We remark that the ergodicity of Burgers equation was already shown in~\cite{Hairer2018Strong}, even in a stronger sense. The only new result here is the exponential speed of convergence (and our proof is very simple).

Consider  $\varphi \in \CK$ and let $(\varphi(t))$ be the unique solution to the backward equation that we constructed in Theorem~\ref{thm:bw-eq} starting from $\varphi(0) = \varphi$. From Proposition~\ref{prop:hille-yosida} we know that $T_t \varphi = \varphi(t)$ and from Lemma~\ref{lem:dissipative} we obtain
\begin{align*}
	\partial_t \frac12 \| \varphi(t) \|^2 & = - \| (-\LL_0)^{1/2} \varphi(t) \|^2.
\end{align*}
Assume that $\int \varphi \dd \mu = \varphi_0 = 0$ for the zero-th chaos component, which by construction holds whenever $(\KK^{-1} \varphi)_0  = 0$. Using the stationarity of $(u_t)$ with respect to $\mu$ we see that then also $(\varphi(t))_0=0$. Recall that $\CF(\varphi(t))_n(k_{1:n}) = 0$ whenever $k_i = 0$ for some $i$, which leads to
\[
	\| (-\LL_0)^{1/2} \varphi(t) \|^2 \geqslant |2\pi|^2 \| \varphi(t) \|^2,
\]
and thus $\partial_t  \| \varphi(t) \|^2 \leqslant - 8 \pi^2 \|\varphi(t) \|^2$, and then from Gronwall's inequality
\begin{equation}\label{eq:ergodic}
	\| T_t \varphi \| \leqslant e^{-4 \pi^2 t} \|\varphi \|.
\end{equation}
This holds for all $\varphi \in \CK$ with $\int \varphi \dd \mu = 0$, but since the left and right hand side can be controlled in terms of $\|\varphi\|$ it extends to all $\varphi \in \Gamma L^2$ with $\int \varphi \dd \mu = 0$. There are two main consequences:
\begin{enumerate}
\item The measure $\mu$ is ergodic: Recall that the set of invariant measures of a Markov process is convex, and the extremal points are the mutually singular ergodic measures. Moreover, $\mu$ is ergodic if and only if for all $A \subset \CS'$ with $T_t \1_A \overset{\mu-\text{a.s.}}{=} \1_A$ for $t \geqslant 0$ we have $\mu(A) \in \{0,1\}$, see \cite[Theorem~2.12]{Eberle2017}. But from~\eqref{eq:ergodic} we know that $T_t \1_A \to \mu(A)$ as $t \to \infty$, and the only possibility to have $\1_A \overset{\mu-\text{a.s.}}{=} \mu(A)$ is if $\mu(A) \in \{0,1\}$. Therefore, $\mu$ is ergodic and in particular there is no invariant measure that is absolutely continuous with respect to $\mu$, other than $\mu$ itself.

\item We can solve the Poisson equation $\hat\LL \varphi = \psi$ for all $\psi \in \Gamma L^2$ with $\int \psi \dd \mu = 0$ by setting $\varphi = \int_0^\infty T_t \psi \dd t$, which is well defined by~\eqref{eq:ergodic}. Here $\hat \LL$ is the Hille-Yosida generator and we do not necessarily have $\varphi \in \CD(\LL)$.
\end{enumerate}

\subsection{Martingale problem with cylinder functions}

The martingale approach to Burgers equation is particularly useful for proving
that the equation arises in the scaling limit of particle systems. The disadvantage of the martingale problem
based on controlled functions is that, given a microscopic system for which we
want to prove that it scales to Burgers equation, it may be difficult to find
similar controlled functions before passing to the limit. Instead it is often
more natural to derive a characterization of the scaling limit based on
cylinder test functions. Here we show that in some cases this characterization
already implies that the limit solves our martingale problem for the
controlled domain of the generator, and therefore it is unique in law. The
biggest restriction is that we have to assume that the process allows for the
It{\^o} trick:

\begin{definition}
  \label{def:cylinder-mp}A process $(u_t)_{t \geqslant 0}$ with trajectories
  in $C (\R_+, \CS')$ \emph{solves the cylinder function
  martingale problem for $\LL$ with initial distribution $\nu$} if
  $u_0 \sim \nu$, and if the following conditions are satisfied:
  \begin{enumerate}[i.]
    \item $\E [| \varphi (u_t) |] \lesssim \| \varphi \|$ locally
    uniformly in $t$, namely $u$ is incompressible;
    
    \item there exists an approximation of the identity $(\rho^{\varepsilon})$
    such that for all $f \in C^{\infty} (\T)$ the process
    \[ M_t^f = u_t (f) - u_0 (f) - \lim_{\varepsilon \to 0} \int_0^t \LL^{\varepsilon} u_s (f)
       \dd s \]
    is a continuous martingale in the filtration generated by $(u_t)$, where
    \[ \LL^{\varepsilon} u (f) =\LL_0 u (f) + \langle
       \partial_x (u \ast \rho^{\varepsilon})^2, f \rangle_{L^2 (\T)}
       ; \]
    moreover $M^f$ has quadratic variation $\langle M^f \rangle_t = 2 t \| \partial_x f \|_{L^2}^2$.
    \item the It{\^o} trick works: for all cylinder functions $\varphi$ and
    all $p \geqslant 1$ we have
    \[ \E \left[ \sup_{t \leqslant T} \left| \int_0^t \varphi (u_s)
       \dd s \right|^p \right] \lesssim T^{p / 2} \| c_{2p}^{\NN}
       (-\LL_0)^{- 1 / 2} \varphi \|^p . \]
  \end{enumerate}
\end{definition}

\begin{remark}
	In \cite{Goncalves2014, Gubinelli2013} so called \emph{stationary energy
solutions} to the Burgers equation are defined. The definition in~\cite{Gubinelli2013} makes the following alternative assumptions:
\begin{itemize}
	\item[i'.] For all times $t \geqslant 0$ the law of $u_t$ is $\mu$;
	\item[ii'.] the conditions in ii. above hold, and additionally the process $\lim_{\varepsilon \to 0} \int_0^t \LL^{\varepsilon} u_s (f) \dd s$ has vanishing quadratic variation;
	\item[iii'.] for $T \geqslant 0$ let $\hat u_t = u_{T-t}$; then $\hat{M}_t^f = \hat u_t (f) - \hat u_0 (f) + \lim_{\varepsilon \to 0} \int_0^t \LL^{\varepsilon} \hat u_s (f) \dd s$ is a continuous martingale in the filtration generated by $(\hat u_t)$, with quadratic variation $\langle \hat M^f \rangle_t = 2 t \| \partial_x f \|_{L^2}^2$.
\end{itemize}
Clearly i'. and ii'. are stronger than i. and ii., and it is shown~\cite[Proposition~3.2]{Gubinelli2018Energy} that any process satisfying i'., ii'., iii'. also satisfies the first inequality in
\begin{equation}\label{eq:pre-KV}
	\E \left[ \sup_{t \leqslant T} \left| \int_0^t \LL_0 \varphi (u_s) \dd s \right|^p \right] \lesssim T^{p / 2} \| (\CE \varphi)^{p/4} \|^2  \lesssim T^{p / 2} \|  c_{p}^\NN (\CE \varphi)^{1/2} \|^p \simeq T^{p / 2} \|  c_{p}^\NN (-\LL_0)^{1/2} \varphi\|^p,
\end{equation}
where the second inequality uses Remark~\ref{rmk:higher-moments} and the third inequality is from \eqref{eq:martingale-qv-bound}. If $\int \varphi \dd \mu = 0$, we can solve $-\LL_0 \psi = \varphi$ and then~\eqref{eq:pre-KV} applied to $\psi$ gives
\[
	\E \left[ \sup_{t \leqslant T} \left| \int_0^t \varphi (u_s) \dd s \right|^p \right] \lesssim T^{p / 2} \|  c_{p}^\NN (-\LL_0)^{-1/2} \varphi \|^p,
\]
i.e. a stronger version of iii. Therefore, we also have uniqueness in law for any process which satisfies i'., ii'. and iii'., or alternatively i., ii., and~\eqref{eq:pre-KV}.

Note that the constant $c_{2p}^\NN$ in iii. is not a typo. This is what we get if we consider a non-stationary process whose initial condition has an $L^2$-density with respect to $\mu$ and we apply Lemma \ref{lem:switch-measure} to pass to a stationary process that has the properties above.
\end{remark}

In the following we assume that $u$ solves the cylinder function martingale
problem for $\LL$ with initial distribution $\nu$, and we fix the
filtration $\CF_t = \sigma (u_s : s \in [0, t])$, $t \geqslant 0$.

\begin{lemma}
  \label{lem:cylinder-mp-testfct}Let $\varphi (u) = \Phi (u (f_1), \ldots, u
  (f_k)) \in \CC$ be a cylinder function. Then the process
  \[ M^{\varphi}_t = \varphi (u_t) - \varphi (u_0) - \lim_{m \rightarrow
     \infty} \int_0^t \LL^m \varphi (u_s) \dd s \]
  is a continuous martingale with respect to $(\CF_t)$, where
  \[ \LL^m \varphi (u) =\LL_0 \varphi (u) + \sum_{i = 1}^k
     \partial_i \Phi (u (f_1), \ldots, u (f_k)) \langle B_m (u), f_i
     \rangle_{L^2 (\T)} \]
  for $B_m (u) \assign \partial_x \Pi_m (\Pi_m u)^2$.
\end{lemma}

\begin{proof}
  Let us write
  \begin{align*}
    u^m_t (f) & \assign u_0 (f) + \int_0^t u_s (\Delta f) \dd s + A^{m,
    f}_t + M^f_t\\
    & \assign u_0 (f) + \int_0^t u_s (\Delta f) \dd s + \int_0^t \langle
    B_m (u_s), f \rangle_{L^2 (\T)} \dd s + M^f_t
  \end{align*}
  for $f \in C^{\infty} (\T)$. Then by It{\^o}'s formula the process
  \[ \varphi (u^m_t) - \varphi (u^m_0) - \int_0^t \LL_0 \varphi
     (u^m_s) \dd s - \int_0^t \sum_{i = 1}^k \partial_i \Phi (u^m_s (f_1),
     \ldots, u^m_s (f_k)) \dd A^{m, f_i}_s \]
  is a martingale. In \cite[Corollary 3.17]{Gubinelli2018Energy} it is shown that for
  all $\alpha < 3 / 4$ and all $T > 0$ and $p \in [1, \infty)$ we have
  $\E [\| A^{m, f_i} - A^{f_i} \|_{C^{\alpha} ([0, T],
  \R)}^p] \rightarrow 0$ for the limit $A^{f_i}$ of $A^{m, f_i}$; here $C^\alpha([0,T], \R)$ is the space of $\alpha$-H\"older continuous functions. Strictly speaking in \cite{Gubinelli2018Energy} only the approximation $\partial_x (\Pi_m u)^2$ is considered, but it is easy to generalize the analysis to $\partial_x \Pi_m (\Pi_m u)^2$. In particular
  \[ \lim_{m \rightarrow \infty} \E \left[ \left| \varphi (u^m_t) -
     \varphi (u^m_0) - \int_0^t \LL_0 \varphi (u^m_s) \dd s -
     \left( \varphi (u_t) - \varphi (u_0) - \int_0^t \LL_0 \varphi
     (u_s) \dd s \right) \right|^p \right] = 0. \]
  Moreover, we can interpret $\int_0^t \sum_{i = 1}^k \partial_i \Phi (u^m_s
  (f_1), \ldots, u^m_s (f_k)) \dd A^{m, f_i}_s$ as a Young integral and e.g.
  by Theorem~1.16 in \cite{Lyons2007} together with the Cauchy-Schwarz
  inequality we have
  \begin{align*}
    & \E \left[ \left| \int_0^t \sum_{i = 1}^k \partial_i \Phi
    (u^m_s (f_1), \ldots, u^m_s (f_k)) \dd A^{m, f_i}_s - \int_0^t \sum_{i
    = 1}^k \partial_i \Phi (u_s (f_1), \ldots, u_s (f_k)) \dd A^{m, f_i}_s
    \right| \right]\\
    & \lesssim \sum_{i = 1}^k \E [\| \partial_i \Phi (u^m (f_1),
    \ldots, u^m (f_k)) - \partial_i \Phi (u (f_1), \ldots, u (f_k))
    \|_{C^{\beta} ([0, T], \R)}^2]^{1 / 2} \E [\| A^{m, f_i}
    \|_{C^{\alpha} ([0, T], \R)}^2]^{1 / 2}
  \end{align*}
  whenever $\beta > 1 - \alpha$ and $\alpha < 3 / 4$. Since $\partial_i \Phi$
  is locally Lipschitz continuous with polynomial growth of the derivative and
  we can take $\beta < \alpha$, the convergence of the first expectation to
  zero follows from the convergence of $u^m$ to $u$ in $L^p (C^{\alpha} ([0,
  T], \R))$. The second expectation $\E [\| A^{m, f_i}
  \|_{C^{\alpha} ([0, T], \R)}^2]$ is uniformly bounded in $m$ by the
  considerations above, and therefore the difference converges to zero. Very
  similar arguments yield
  \[ \lim_{m \rightarrow \infty} \E \left[ \left| \int_0^t \sum_{i =
     1}^k \partial_i \Phi (u_s (f_1), \ldots, u_s (f_k)) \dd A^{m, f_i}_s -
     \int_0^t \sum_{i = 1}^k \partial_i \Phi (u_s (f_1), \ldots, u_s (f_k))
     \dd A^{f_i}_s \right| \right] = 0, \]
  and since all the convergences are in $L^1$ we get that
  \begin{align*}
    M^{\varphi}_t & = \varphi (u_t) - \varphi (u_0) - \int_0^t \LL_0
    \varphi (u_s) \dd s - \int_0^t \sum_{i = 1}^k \partial_i \Phi (u_s
    (f_1), \ldots, u_s (f_k)) \dd A^{f_i}_s\\
    & = \varphi (u_t) - \varphi (u_0) - \int_0^t \LL_0 \varphi (u_s)
    \dd s - \lim_{m \rightarrow \infty} \int_0^t \sum_{i = 1}^k \partial_i
    \Phi (u_s (f_1), \ldots, u_s (f_k)) \dd A^{m, f_i}_s
  \end{align*}
  is a continuous martingale.
\end{proof}

While it is not obvious from the proof, here we already used that the It{\^o}
trick works for $(u_t)$. Indeed, Corollary 3.17 of \cite{Gubinelli2018Energy}
crucially relies on this.

\begin{theorem}\label{thm:cylinder-mp}
  Let $u$ solve the cylinder function martingale problem for $\LL$
  with initial distribution $\nu$. Then $u$ solves the martingale problem for
  $\LL$ in the sense of Section~\ref{sec:mp-ex}, and in particular
  its law is unique by Theorem~\ref{thm:mp-uniq}.
\end{theorem}

\begin{proof}
  Let $\varphi \in \CD (\LL)$ and define $\varphi^M$ as the
  projection of $\varphi$ onto the chaos components of order $\leqslant M$,
  and in each chaos we project onto the Fourier modes $| k |_{\infty}
  \leqslant M$. In particular, $\varphi^M \in \CC$ and by
  Lemma~\ref{lem:cylinder-mp-testfct} the process
  \[ M^{\varphi^M}_t = \varphi^M (u_t) - \varphi^M (u_0) - \lim_{m \rightarrow
     \infty} \int_0^t \LL^m \varphi^M (u_s) \dd s \]
  is a martingale. By construction $\E [| \varphi^M (u_t) - \varphi^M
  (u_0) - \varphi (u_t) - \varphi (u_0) |] \rightarrow 0$ as $M \rightarrow
  \infty$, so if we can show that
  \[ \lim_{M \rightarrow \infty} \E \left[ \left| \lim_{m \rightarrow
     \infty} \int_0^t \LL^m \varphi^M (u_s) \dd s - \int_0^t
     \LL \varphi (u_s) \dd s \right| \right] = 0, \]
  then the proof is complete. Since we saw in the proof of
  Lemma~\ref{lem:cylinder-mp-testfct} that the integral $\int_0^t
  \LL^m \varphi^M (u_s) \dd s$ converges in $L^1$, we can take out
  the limit in $m$ from the expectation (or we could just apply Fatou's
  lemma), so that it suffices to show that the right hand side of the following inequality is zero:
  \begin{align*}
    & \lim_{M \rightarrow \infty} \lim_{m \rightarrow \infty} \E
    \left[ \left| \int_0^t (\LL^m \varphi^M -\LL \varphi)
    (u_s) \dd s \right| \right]\\
    & \lesssim_t \lim_{M \rightarrow \infty} \lim_{m \rightarrow \infty} \|
    (-\LL_0)^{- 1 / 2} (\LL^m \varphi^M -\LL
    \varphi) \|\\
    & \lesssim \lim_{M \rightarrow \infty} \lim_{m \rightarrow \infty} [\|
    (-\LL_0)^{1 / 2} (\varphi^M - \varphi) \| + \|
    (-\LL_0)^{- 1 / 2} (\GG^m \varphi^M -\GG
    \varphi) \| ].
  \end{align*}
  For the first term on the right hand side this follows from the fact that
  $\| (-\LL_0)^{1 / 2} \varphi \| \lesssim \| (-\LL_0)^{1 /
  2} \varphi^{\sharp} \|$ by Lemma~\ref{lem:controlled-fct} and from the
  dominated convergence theorem. For the second term on the right hand side we have by the triangle inequality and Lemma~\ref{lem:G-apriori}
  \begin{align*}
  	\|(-\LL_0)^{- 1 / 2} (\GG^m \varphi^M -\GG\varphi) \| & \leqslant \|(-\LL_0)^{- 1 / 2} \GG^m( \varphi^M - \varphi) \| + \|(-\LL_0)^{- 1 / 2} (\GG^m  -\GG) \varphi) \| \\
  	& \lesssim \|(-\LL_0)^{1 / 2} (1+ \NN) ( \varphi^M - \varphi) \| + \|(-\LL_0)^{- 1 / 2} (\GG^m  -\GG) \varphi) \|.
  \end{align*}
  The first term vanishes as $M \to \infty$, by the same argument as before. The second term vanishes by the uniform estimates of Lemma~\ref{lem:G-apriori} together with the dominated convergence theorem which shows that $(\GG^m  -\GG)$ goes to zero as $m \to \infty$.
\end{proof}

\section{Extensions}\label{sec:extensions}

The uniqueness in law of solutions to the cylinder function martingale problem
is not new, the stationary case was previously treated in \cite{Gubinelli2018Energy}
and a non-stationary case (even slightly more general than the one we study
here) in \cite{Gubinelli2018Probabilistic}. This was extended to Burgers equation with
Dirichlet boundary conditions in \cite{Goncalves2017}. However, these works
are crucially based on the Cole-Hopf transform that linearizes the equation,
and they do not say anything about the generator
$\LL$. In the following we show that our arguments adapt to some variants of Burgers equation, none of which can be linearized via the Cole-Hopf transform. In that
sense our new approach is much more robust than the previous one.

\subsection{Multi-component Burgers equation}

Let us consider the multi-component Burgers equation studied in
\cite{Funaki2017,Kupiainen2017}. This equation reads for $u \in C (\R_+,
(\CS')^d)$ as
\[ \partial_t u^i = \Delta u^i + \sum_{j, j' = 1}^d \Gamma^i_{j j'} \partial_x
   (u^j u^{j'}) + \sqrt{2} \partial_x \xi^i, \qquad i = 1, \ldots, d, \]
where $(\xi^1, \ldots, \xi^d)$ are independent space-time white noises and we assume the so called \emph{trilinear condition} of
\cite{Funaki2017}:
\[ \Gamma^i_{j j'} = \Gamma^i_{j' j} = \Gamma^j_{j' i}, \]
i.e. that $\Gamma$ is symmetric in its three arguments $(i, j, j')$. Under
this condition the product measure $\mu^{\otimes d}$ is invariant for $u$, also at the level of the Galerkin approximation, see Proposition~5.5 of
\cite{Funaki2017}. We can interpret $\mu^{\otimes d}$ as a white noise on $L^2_0
(\{ 1, \ldots, d \} \times \T) \simeq L^2_0 (\T,
\R^d)$, equipped with the inner product
\[ \langle f, g \rangle_{L^2 (\T \times \{ 1, \ldots, d \})} \assign
   \sum_{i = 1}^d \langle f^i, g^i \rangle_{L^2 (\T)} \assign \sum_{i
   = 1}^d \langle f (i, \cdot), g (i, \cdot) \rangle_{L^2 (\T)}
\]
and where we assume that $\hat{f} (i, 0) \assign \widehat{f^i} (0) = 0$ for
all $i$, and similarly for $g$; see also Example~1.1.2 of \cite{Nualart2006}. To
simplify notation we write $\T_d =\T \times \{ 1, \ldots, d
\}$ in what follows, not to be confused with $\T^d$. Cylinder
functions now take the form $\varphi (u) = \Phi (u (f_1), \ldots, u (f_J))$
for $\Phi \in C^2_p (\R^J)$ and $f_j \in C^{\infty} (\T_d)
\simeq C^{\infty} (\T, \R^d)$, where the duality pairing $u
(f)$ is defined as
\[ u (f) = \sum_{i = 1}^d u^i (f^i) = \sum_{i = 1}^d u^i (f (i, \cdot)), \]
and in the following we switch between the notations $f^i (x) = f (i, x)$
depending on what is more convenient. The chaos expansion takes symmetric
kernels $\varphi_n \in L^2_0 (\T_d^n)$ as input, and the Malliavin
derivative acts on the cylinder function $\varphi (u) = \Phi (u (f_1), \ldots,
u (f_J))$ with $f_j \in C^{\infty} (\T_d) \simeq C^{\infty}
(\T, \R^d)$ and $\Phi \in C^2_p (\R^J)$ as
\[ D_{\zeta} \varphi = D_{(i x)} \varphi = \sum_{j = 1}^J \partial_j \Phi (u
   (f_1), \ldots, u (f_J)) f_j^i (x) = \sum_{j = 1}^J \partial_j \Phi (u
   (f_1), \ldots, u (f_J)) f_j (\zeta), \]
where from now on we write $\zeta$ for the elements of $\T_d$. We
also have $D_{\zeta} W_n (\varphi_n) = n W_{n - 1} (\varphi_n (\zeta,
\cdot))$ as for $d = 1$. Let us define formally
\[ B (u) (\zeta) = B (u) (i, x) = \sum_{j, j' = 1}^d \Gamma^i_{j j'}
   \partial_x (u^j u^{j'}) (x) = W_2 \left( \sum_{j, j' = 1}^d \Gamma^i_{j j'}
   \partial_x (\delta_{(j x)} \otimes \delta_{(j' x)}) \right), \]
where $\delta_{(j x)} (i y) =\1_{i = j} \delta (x - y)$. Then the
Burgers part of the generator is formally given by
\[ \GG \varphi (u) = \langle B (u), D \varphi (u) \rangle_{L^2
   (\T_d)} \backassign \int_{\zeta} B (u) (\zeta) D_{\zeta} \varphi
   (u) \dd \zeta . \]
This becomes rigorous if we consider the Galerkin approximation with
cutoff $\Pi_m$, but for simplicity we continue to formally argue in the limit
$m = \infty$. We have the following generalization of Lemma~\ref{lem:Gm}:

\begin{lemma}
  We have $\GG=\GG_+ +\GG_-$, where
  \begin{gather*}
    \GG_+ W_n (\varphi_n) = n W_{n + 1} \left( \int_{(i x)} \sum_{j,
    j' = 1}^d \Gamma^i_{j j'} \partial_x (\delta_{(j x)} \otimes \delta_{(j'
    x)}) (\cdot) \varphi_n ((i x), \cdot) \right),\\
    \GG_- W_n (\varphi_n) = 2 n (n - 1) W_{n - 1} \left( \int_{(i_1
    x_1), (i_2 x_2)} \sum_{j, j' = 1}^d \Gamma^{i_1}_{j j'} \partial_{x_1}
    (\delta_{(j x_1)} (i_2 x_2) \delta_{(j' x_1)} (\cdot)) \varphi_n ((i_1
    x_1), (i_2 x_2), \cdot) \right),
  \end{gather*}
  and moreover we have for all $\varphi_{n + 1} \in L^2 (\T^{n + 1})$
  and $\varphi_n \in L^2 (\T^n)$
  \[ \langle W_{n + 1} (\varphi_{n + 1}), \GG_+ W_n (\varphi_n)
     \rangle = - \langle \GG_- W_{n + 1} (\varphi_{n + 1}), W_n
     (\varphi_n) \rangle . \]
\end{lemma}

\begin{proof}
  This follows similarly as in Lemma~\ref{lem:Gm}, making constant use of the
  trilinear condition for $\Gamma$.
\end{proof}

The Fourier variables now are indexed by $\Z_0 \times \{ 1, \ldots, d
\} \backassign \Z_d$, and we write $(i k)$, $(i, k)$ or $\kappa$ for
the elements of $\Z_d$, and
\[ \hat{f} (\kappa) = \hat{f} (i, k) = \int_{\T} e^{- 2 \pi \iota k
   x} f (i, x) \dd x, \qquad \kappa = (i k) \in \Z_d . \]
We have for $\varphi = \sum_n W_n (\varphi_n)$:
\[ \| \varphi \|^2 = \sum_n n! \sum_{\kappa \in \Z_d^n} |
   \hat{\varphi}_n (\kappa) |^2. \]

\begin{lemma}
  In Fourier variables the operators $\LL_0, \GG_+,
  \GG_-$ are given by
  \begin{gather*}
    \CF (\LL_0 \varphi)_n (\kappa_{1 : n}) = - (| 2 \pi k_1
    |^2 + \cdots + | 2 \pi k_n |^2) \hat{\varphi}_n (\kappa_{1 : n}),\\
    \CF (\GG_+ \varphi)_n (\kappa_{1 : n}) = - (n - 1)
    \sum_{i = 1}^d \Gamma^i_{i_1 i_2} 2 \pi \iota (k_1 + k_2) \hat{\varphi}_n
    ((i, k_1 + k_2), \kappa_{3 : n + 1}),\\
    \CF (\GG_- \varphi)_n (\kappa_{1 : n}) = - 2 \pi \iota
    k_1 n (n + 1) \sum_{j_1, j_2 = 1}^d \Gamma^{i_1}_{j_1 j_2} \sum_{p + q =
    k_1} \hat{\varphi}_n ((j_1 p), (j_2 q), \kappa_{2 : n + 1}),
  \end{gather*}
  respectively.
\end{lemma}

\begin{proof}
	The proof is more or less the same as for $d=1$.
\end{proof}

In other words $\GG_+$ and $\GG_-$ are finite linear
combinations of some mild variations of the operators that we considered in $d
= 1$. In particular they satisfy all the same estimates and we obtain the
existence and uniqueness of solutions for the martingale problem for
$\LL=\LL_0 +\GG_+ +\GG_-$ as before, and
also for the cylinder function martingale problem.

\subsection{Fractional Burgers equation}

In the paper \cite{Gubinelli2013} the authors not only study our stochastic Burgers
equation, but also the fractional generalization
\[ \partial_t u = - A^{\theta} u + \partial_x u^2 + A^{\theta / 2} \xi, \]
for $\theta > 1 / 2$ and $A = - \Delta$. They define and construct stationary
energy solutions for all $\theta > 1 / 2$, and they prove uniqueness in
distribution for $\theta > 5 / 4$. Here we briefly sketch how to adapt our
arguments to deduce the uniqueness for $\theta > 3 / 4$, also for the
non-stationary equation as long as the initial condition is absolutely
continuous with density in $L^2 (\mu)$. Unfortunately we cannot treat the
limiting case $\theta = 3 / 4$ which would be scale-invariant and which plays
an important role in the work \cite{Goncalves2018}.

In Section~4 of \cite{Gubinelli2013} it is shown that $u$ is still invariant
under the distribution $\mu$ of the white noise. By adapting the arguments of
Lemma~3.7 in \cite{Gubinelli2018Energy} we see that the (formal) generator of $u$ is
given by
\[ \LL= \LL_\theta +\GG, \]
where
\[
	\CF (\LL_\theta \varphi)_n(k_{1:n}) = -(|2\pi k_1|^{2\theta} + \dots + |2\pi k_n|^{2\theta}) \hat \varphi_n(k_{1:n}).
\]
Up to multiples of $\NN$ we can estimate $(-\LL_\theta)$ by $(-\LL_0)^\theta$ and vice versa, so we would expect that $(-\LL_\theta)^{-1}$ gains regularity of order $(-\LL_0)^{-\theta}$. We saw in Lemma~\ref{lem:G-apriori} that $\GG$ loses
$(-\LL_0)^{3 / 4}$ regularity, and therefore it is canonical to assume
$\theta > 3 / 4$, so that we can gain back more regularity from the linear
part of the dynamics than the nonlinear part loses. To construct controlled
functions we only need to slightly adapt Lemma~\ref{lem:controlled-fct} and
replace $(-\LL_0)^{- 1}$ by $(-\LL_\theta)^{- 1}$. For simplicity we restrict our attention to $\theta \leqslant 1$ because this allows us to estimate
\begin{equation}\label{eq:theta-comparison}
	(|k_1|^{2\theta} + \dots + | k_n|^{2\theta})^{-1} \leqslant (k_1^{2} + \dots + k_n^{2})^{-\theta}, \qquad \text{i.e. } \|(-\LL_\theta)^{-1} \varphi \| \leqslant \| (-\LL_0)^{-\theta} \varphi \|.
\end{equation}

\begin{lemma}
Let $\theta \in (3/4,1]$, let $w$ be a weight, let $\gamma \in (1/4, 1 / 2]$, and let $L \geqslant 1$. For $N_n = L (1 + n)^{3/(4\theta-3)}$ we have
  \begin{equation}
    \label{eq:controlled-fct-1} \| w (\NN) (-\LL_0)^{\gamma}
    (-\LL_0)^{- 1} \GG^{\succ} \varphi \| \lesssim |w| L^{3/2-
    2\theta} \| w (\NN) (-\LL_0)^{\gamma} \varphi \|,
  \end{equation}
  where the implicit constant on the right hand side is independent of $w$. From here the construction of controlled functions $\varphi
  =\KK \varphi^{\sharp} = (-\LL_\theta)^{- 1} \GG_+^{\succ} \varphi + \varphi^{\sharp}$ for given $\varphi^{\sharp}$ works as in Lemma~\ref{lem:controlled-fct}.
\end{lemma}

\begin{proof}
  We treat $\GG_+^\succ$ and $\GG_-^\succ$ separately. Using~\eqref{eq:theta-comparison} and that $1 - 2\gamma \geqslant 0$, the $\GG_+^\succ$ term can be estimated as in the proof of Lemma~\ref{lem:controlled-fct}:
  \begin{align*}
    & \sum_{k_{1 : n}}  | \CF ((-\LL_0)^{\gamma}
    (-\LL_{\theta})^{- 1} \GG_+^{\succ} \varphi)_n (k_{1 :
    n}) |^2\\
    & \lesssim n \sum_{\ell_{1 : n - 1}, p} (\mathbb{I}_{| p | \geqslant N_n/2}
    +\mathbb{I}_{| \ell_{1 : n - 1} |_{\infty} \geqslant N_n/2}) \frac{(\ell_1^2
    + \cdots + \ell_{n - 1}^2)^{2 \gamma}}{(| p |^2 + | \ell_1 |^2 + \cdots +
    | \ell_{n - 1} |^2)^{2 \theta - 1}} | \hat{\varphi}_{n - 1} (\ell_{1 : n -
    1}) |^2\\
    & \lesssim n \sum_{\ell_{1 : n - 1}} \left( N_n^{3 - 4 \theta} +
    \frac{\mathbb{I}_{| \ell_{1 : n - 1} |_{\infty} \geqslant N_n}}{(\ell_1^2
    + \cdots + \ell_{n - 1}^2)^{2 \theta - 3 / 2}} \right) (\ell_1^2 + \cdots
    + \ell_{n - 1}^2)^{2 \gamma} | \hat{\varphi}_{n - 1} (\ell_{1 : n - 1})
    |^2\\
    & \lesssim n \sum_{\ell_{1 : n - 1}} N_n^{3 - 4 \theta} (\ell_1^2 + \cdots
    + \ell_{n - 1}^2)^{2 \gamma} | \hat{\varphi}_{n - 1} (\ell_{1 : n - 1})
    |^2,
  \end{align*}
  where the third step follows from Lemma~\ref{lem:sum-estimate} (and here we need $\theta < 3/4$).
  
  For the $\GG_-$ term we have by the same arguments as in Lemma~\ref{lem:controlled-fct} and using~\eqref{eq:theta-comparison} and that $\theta > 3/4$
  \begin{align*}
  	\sum_{k_{1 : n}}  | \CF ((-\LL_0)^{\gamma} (-\LL_\theta)^{-1} \GG_-^{\succ} \varphi)_n (k_{1 : n}) |^2 & \lesssim \sum_{k_{1:n}} \frac{ \1_{|k_{1:n}|_\infty \geqslant N_n} n^4 (k_1^2)^{3/2 }}{(k_1^2 + \cdots + k_n^2)^{2\theta}}  \sum_{p + q = k_1} (p^2 + q^2)^{2\gamma } | \hat{\varphi}_{n + 1} (p, q, k_{2 : n}) |^2 \\
  	& \leqslant N_n^{3-4\theta} n^4 \sum_{\ell_{1:n+1}} (\ell_1^2 + \cdots + \ell_{n+1}^2)^{2\gamma } | \hat{\varphi}_{n + 1} (\ell_{1:n+1}) |^2,
  \end{align*}
  and from here on the proof is the same as for Lemma~\ref{lem:controlled-fct}.
\end{proof}

Proposition~\ref{prop:domain} remains essentially unchanged in our setting, because for $\varphi = \KK \varphi^\sharp$ we have $\LL \varphi = \GG^\prec \varphi + \LL_\theta \varphi^\sharp$. The only difference is that, since we still want to measure regularity in terms of $(-\LL_0)$, we have $\| \LL_\theta \varphi^\sharp \| \lesssim \| \NN^{1-\theta} (-\LL_0) \varphi^\sharp\|$ by H\"older's inequality. Also the proof of Lemma~\ref{lem:dom-def} carries over to our setting. And also the analysis of the backward equation is more or less the same as before. The main difference is that now we only have a priori estimates in $(-\LL_0)^{-\theta/2} \Gamma L^2$ and no longer in $(-\LL_0)^{-1/2} \Gamma L^2$ (with weights in $\NN$). But for the controlled analysis it is only important to have an a priori estimate in $(-\LL_0)^{-1/4-\delta} \Gamma L^2$, because that is what we need to control the contribution from $\GG^\prec$. So since $\theta/ 2 > 3/8 > 1/4$ the same arguments work, and then we obtain the existence and uniqueness of solutions to backward equation and martingale problem by the same arguments as for $\theta=1$, and also the cylinder function martingale problem has unique solutions in this case.

\subsection{Burgers equation on the real line}\label{sec:full-space}

Burgers equation on $\R_+ \times \R$ is very similar to the case of periodic boundary conditions. The only difference is that now instead of sums over Fourier modes we have to consider integrals, which might lead to divergences at $k \simeq 0$. But since most of our estimates boil down to an application of Lemma~\ref{lem:sum-estimate}, and this lemma remains true if the sum in $k$ is replaced by an integral, most of our estimates still work on the full space. In fact \emph{all} estimates in Section~\ref{sec:domain} remain true, but some of them are not so useful any more because we no longer have $\| \varphi \| \lesssim \|(-\LL_0)^\gamma \varphi\|$ for $\gamma > 0$ and $\int \varphi \dd \mu =0$. But we can strengthen the results as follows (with the difference to the previous results marked in blue):
\begin{itemize}
	\item In Lemma~\ref{lem:controlled-fct} we can use the cutoff $\1_{|k_{1:n}|_\infty > N_n}$ to estimate
			\begin{align*}
				\|w(\NN) \blue{(1-\LL_0)^\gamma} (-\LL_0)^{-1} \GG^\succ \varphi \| & \lesssim  \| w(\NN) (-\LL_0)^\gamma (-\LL_0)^{-1} \GG^\succ \varphi \|\lesssim |w| L^{-1/2} \| w(\NN) (-\LL_0)^\gamma \varphi \| \\
				& \leqslant |w| L^{-1/2} \| w(\NN) \blue{(1-\LL_0)^\gamma} \varphi \|
			\end{align*}
			and thus
			\[
				\|w(\NN) \blue{(1-\LL_0)^\gamma} \KK \varphi^\sharp \| + L^{1/2}\|w(\NN) \blue{(1-\LL_0)^\gamma} (\KK \varphi^\sharp - \varphi^\sharp) \| \lesssim \|w(\NN) \blue{(1-\LL_0)^\gamma} \varphi^\sharp \|.
			\]
			Similarly we get in Lemma~\ref{lem:controlled-fct-adapted} the better bound
			\[
				\| w (\NN) \blue{(1-\LL_0)^{\gamma}} (-\LL_0)^{- 1} \GG^{\succ} \varphi \| \lesssim |w| \| w (\NN) (1 + \NN)^{3/2} (-\LL_0)^{\gamma - 1 / 4} \varphi \| .
			\]
	\item In the proof Proposition~\ref{prop:domain} we simply bound $(k_1^2 + \dots + k_n^2)^{2\gamma} \1_{|k_{1:n}|_\infty \leqslant N_n} \leqslant n^{2\gamma} N_n^{4\gamma}$, and of course this works also with $(1 + k_1^2 + \dots + k_n^2)^{2\gamma}$, so that we get the slightly stronger result
			\begin{equation*}
				\| w (\NN) \blue{(1-\LL_0)^\gamma} \GG^\prec \varphi \| \lesssim \| w (\NN) (1 + \NN)^{9/2+7\gamma} (-\LL_0)^{1 / 4 + \delta} \varphi^\sharp \|.
			\end{equation*}
	\item The definition of the domain in Lemma~\ref{lem:dom-def} is problematic now, because it does not even guarantee that $\CD(\LL) \subset \Gamma L^2$. So instead we set
  \[ \CD_w (\LL) \assign \{ \KK \varphi^{\sharp} :
     \varphi^{\sharp} \in w (\NN)^{- 1} (-\LL_0)^{- 1} \Gamma
     L^2 \cap w (\NN)^{- 1} (1 + \NN)^{- 9/2}
     \blue{(1 -\LL_0)^{- 1 / 2}} \Gamma L^2 \}, \]
	and then we get from the stronger version of Lemma~\ref{lem:controlled-fct} the better estimate
  \begin{equation*}
  \begin{aligned}
    \| w (\NN) \blue{(1-\LL_0)^{1/2}} (\varphi^M - \psi) \| & \lesssim M^{- 1 / 2} \| w
    (\NN) \blue{(1-\LL_0)^{1/2}} \psi \|,\\ 
    \| w (\NN) \blue{(1-\LL_0)^{1 / 2}} \varphi^M \| & \lesssim \| w
    (\NN) \blue{(1-\LL_0)^{1 / 2}} \psi \|. \\
  \end{aligned}
  \end{equation*}
  \item The analysis in Section~\ref{sec:bw-apriori} does not change, and Lemma~\ref{lem:bw-apriori} together with Corollary~\ref{cor:bw-apriori-2} give as an a priori bound on $\|(1+\NN)^\alpha (1 - \LL_0)^{1/2} \varphi^m \|$ and $\|(1+\NN)^\alpha \partial_t \varphi^m\|$ in terms of $\varphi^m_0$.
  \item In the controlled analysis of Section~\ref{sec:bw-controlled} we can strengthen the bound from Lemma~\ref{lem:phi-sharp-apriori} to control $\|(1+\NN)^\alpha \blue{(1 - \LL_0)^{1/2}} \varphi^{m,\sharp}\|$ in terms of $\varphi^{m,\sharp}_0$, and this is sufficient to control $\blue{(1-\LL_0)^\gamma} \GG^{m,\prec} \varphi^m$. Also for the other terms we now bound $\blue{(1-\LL_0)^\gamma (\cdot)}$ instead of $(-\LL_0)^\gamma (\cdot)$. Here we need the strengthened version of Lemma~\ref{lem:controlled-fct-adapted} mentioned above, and we also use that $\|(1+\NN)^\alpha \blue{(1-\LL_0)^\beta} S_t \psi\| \lesssim \blue{(t^{-\beta} \vee 1)} \|(1+\NN)^\alpha \psi\|$. In the end we get strong solutions to the backward equation for initial conditions in
  		\[
  			\CK_\alpha := \bigcup_{\gamma \in (3/8, 5 / 8)} \KK (1 + \NN)^{- p(\alpha,\gamma)} \blue{(1 -\LL_0)^{- 1 - \gamma}} \Gamma L^2 \subseteq \Gamma L^2.
  		\]
  	\item Existence and uniqueness for the martingale problem are exactly the same as on the torus, the only difference is that we have to use the strengthened version of Proposition~\ref{prop:domain} to approximate cylinder functions by functions in $\CD(\LL)$.
  	\item The cylinder function martingale problem is more complicated: In the proof of Theorem~\ref{thm:cylinder-mp} we used that $\|(-\LL_0)^{-1/2} \GG \varphi \| \lesssim \| (-\LL_0)^{1/2} \varphi\|$, which is no longer true on the full space. But we can decompose $\GG = \GG_- + \GG_+$ and estimate the contribution from $\GG_-$ by directly using Lemma~\ref{lem:G-apriori} for $\gamma = 0$, without applying the It\^o trick (it follows from Young's inequality for products that $\CD(\LL) \subset (1+\NN)^{-1} (-\LL_0)^{-3/4} \Gamma L^2$). And for $\GG_+$ we can use the It\^o trick together with the  bound $\|(-\LL_0)^{-1/2} \GG_+ \varphi \| \lesssim \| (-\LL_0)^{1/4} \varphi\| \lesssim \| \blue{(1-\LL_0)^{1/2}} \varphi\|$, where the right hand side is under control.
\end{itemize}
In that way all results from Section~\ref{sec:controlled}-\ref{sec:mp} apart from Section~\ref{sec:ergodic} carry over to Burgers equation on the full space. Of course the exponential ergodicity of Section~\ref{sec:ergodic} does not hold on the full space, because $\LL_0$ no longer has a spectral gap.

\appendix\section{Auxiliary results}

The following simple estimate is used many times, so we formulate it as a
lemma.

\begin{lemma}
  \label{lem:sum-estimate}Let $C \geqslant 0$, $a > 1 / 2$, and $k \in
  \Z$ be such that $k^2 + C > 0$. Then
  \[ \sum_{p + q = k} \left( \frac{1}{p^2 + q^2 + C} \right)^a = \sum_p \left(
     \frac{1}{p^2 + (k - p)^2 + C} \right)^a \lesssim \left( \frac{1}{k^2 + C}
     \right)^{a - \frac{1}{2}} . \]
\end{lemma}

\begin{proof}
  Since $p^2 + (k-p)^2 \simeq p^2 + k^2$, we have
  \begin{align*}
    \sum_p \left( \frac{1}{p^2 + (k - p)^2 + C} \right)^a & \lesssim \int_0^{\infty} \left( \frac{1}{y^2 + k^2 + C} \right)^a
    \dd y \\
    & = (k^2 + C)^{- a} \int_0^{\infty} \left( \frac{1}{\left(
    \frac{y}{\sqrt{k^2 + C}} \right)^2 + 1} \right)^a \dd y\\
    & = (k^2 + C)^{- a + \frac{1}{2}} \int_0^{\infty} \left( \frac{1}{y^2 +
    1} \right)^a \dd y,
  \end{align*}
  and since $2 a > 1$ the integral on the right hand side is finite and our
  claim follows.
\end{proof}

\begin{lemma}
  \label{lem:controlled-fct-adapted}In the context of
  Lemma~\ref{lem:controlled-fct} let now $\gamma \in (1/4, 3 / 4)$. Then we have
  \[ \| w (\NN) (-\LL_0)^{\gamma} (-\LL_0)^{- 1}
     \GG^{\succ} \varphi \| \lesssim |w| \| w (\NN)
     (1 + \NN)^{3/2} (-\LL_0)^{\gamma - 1 / 4} \varphi \| . \]
\end{lemma}

\begin{proof}
  In the proof of Lemma~\ref{lem:controlled-fct} we derived the estimate
  \begin{align*}
    & \| w (\NN) (-\LL_0)^{\gamma} (-\LL_0)^{- 1}
    \GG_+^{\succ} \varphi \|^2\\
    & \hspace{40pt} \lesssim \sum_{n \geqslant 2} n!w (n)^2 n \sum_{\ell_{1 :
    n - 1}, p} \1_{| \ell_{1 : n - 1} |_{\infty} \vee | p | \geqslant
    N_n} \frac{\ell_1^2 + \cdots + \ell_{n - 1}^2}{(p^2 + \ell_1^2 + \cdots +
    \ell_{n - 1}^2)^{2 - 2 \gamma}} | \hat{\varphi}_{n - 1} (\ell_{1 : n -
    1}) |^2.
  \end{align*}
  For $\gamma < 3 / 4$ (which is equivalent to
  $2 - 2 \gamma > 1 / 2$) it follows from Lemma~\ref{lem:sum-estimate} that
  \begin{align*}
    & \sum_{n \geqslant 2} n!w (n)^2 n \sum_{\ell_{1 : n - 1}, p}
    \1_{| \ell_{1 : n - 1} |_{\infty} \vee | p | \geqslant N_n}
    \frac{\ell_1^2 + \cdots + \ell_{n - 1}^2}{(p^2 + \ell_1^2 + \cdots +
    \ell_{n - 1}^2)^{2 - 2 \gamma}} | \hat{\varphi}_{n - 1} (\ell_{1 : n -
    1}) |^2\\
    & \hspace{40pt} \lesssim \sum_{n \geqslant 2} n!w (n)^2 n \sum_{\ell_{1 :
    n - 1}} \frac{\ell_1^2 + \cdots + \ell_{n - 1}^2}{(\ell_1^2 + \cdots +
    \ell_{n - 1}^2)^{3 / 2 - 2 \gamma}} | \hat{\varphi}_{n - 1} (\ell_{1 :
    n - 1}) |^2\\
    & \hspace{40pt} = \sum_{n \geqslant 1} n! (n + 1) w (n + 1)^2 (n + 1)
    \sum_{\ell_{1 : n}} (\ell_1^2 + \cdots + \ell_n^2)^{2 \gamma - 1 / 2} |
    \hat{\varphi}_n (\ell_{1 : n}) |^2\\
    & \hspace{40pt} \leqslant |w|^2 \| w (\NN+ 1) (1 + \NN)
    (-\LL_0)^{\gamma - 1 / 4} \varphi \|^2 .
  \end{align*}
  For $(-\LL_0)^{- 1} \GG_+^{\succ} \varphi$ we argue similarly as in Lemma~\ref{lem:controlled-fct}: We apply \eqref{eq:Gminus-CS} with $\beta = 1 - \gamma < 1/2$ (here we need $\gamma > 1/2$) to estimate
  \begin{align*}
  	\sum_{k_{1 : n}}  | \CF ((-\LL_0)^{\gamma - 1} \GG_-^{\succ} \varphi)_n (k_{1 : n}) |^2 & \lesssim \sum_{k_{1:n}} \frac{ \1_{|k_{1:n}|_\infty \geqslant N_n} n^4 k_1^2}{(k_1^2 + \cdots + k_n^2)^{2 - 2\gamma}} \bigg| \sum_{p+q = k_1} \hat{\varphi}_{n + 1} (p, q, k_{2 : n}) \bigg|^2 \\
  	& \lesssim \sum_{k_{1:n}} \frac{ \1_{|k_{1:n}|_\infty \geqslant N_n} n^4 k_1^2 (k_1^2)^{1 - 2\gamma}}{(k_1^2 + \cdots + k_n^2)^{2 - 2\gamma}}  \sum_{p + q = k_1} (p^2 + q^2)^{2\gamma -1/2} | \hat{\varphi}_{n + 1} (p, q, k_{2 : n}) |^2 \\
  	& \leqslant n^4 \sum_{\ell_{1:n+1}} (\ell_1^2 + \cdots + \ell_{n+1}^2)^{2\gamma -1/2} | \hat{\varphi}_{n + 1} (\ell_{1:n+1}) |^2,
  \end{align*}
  which leads to $\| w(\NN) (-\LL_0)^\gamma (-\LL_0)^{-1} \GG_-^\succ \varphi \| \lesssim |w| \| w(\NN)(1+\NN)^{3/2} (-\LL_0)^{\gamma-1/4} \varphi \|$.
\end{proof}

\begin{lemma}\label{lem:time-dependent-mp}
  Let $\varphi \in C (\R_+, \CD (\LL)) \cap C^1
  (\R_+, \Gamma L^2)$ and let $u$ be an incompressible to the martingale problem for $\LL$. Then
  \[ \varphi (t, u_t) - \varphi (0, u_0) - \int_0^t (\partial_s +\LL)
     \varphi (s, u_s) \dd s, \qquad t \geqslant 0, \]
  is a martingale.
\end{lemma}

\begin{proof}
  We discretize time: Set $t_k = k t / n$ and
  \begin{align*}
    \varphi (t, u_t) - \varphi (0, u_0) & = \sum_{k = 0}^{n - 1} [\varphi
    (t_{k + 1}, u_{t_{k + 1}}) - \varphi (t_k, u_{t_{k + 1}}) + \varphi (t_k,
    u_{t_{k + 1}}) - \varphi (t_k, u_{t_k})]\\
    & = \sum_{k = 0}^{n - 1} \left[ \int_{t_k}^{t_{k + 1}} \partial_s \varphi
    (s, u_{t_{k + 1}}) \dd s + \int_{t_k}^{t_{k + 1}} \LL \varphi
    (t_k, u_s) \dd s + M^{\varphi (t_k)}_{t_{k + 1}} - M^{\varphi
    (t_k)}_{t_k} \right] .
  \end{align*}
  Now for $[s]^n = \min \{ t_k : t_k \geqslant s \}$ (which depends on $n$
  because the $t_k$ depend on $n$)
  \begin{align*}
    & \E \left[ \left| \sum_{k = 0}^{n - 1} \int_{t_k}^{t_{k + 1}}
    \partial_s \varphi (s, u_{t_{k + 1}}) \dd s - \int_0^t \partial_s
    \varphi (s, u_s) \dd s \right| \right]\\
    & \hspace{30pt} \leqslant \int_0^t \E [| \partial_s \varphi (s,
    u_{[s]^n}) - \partial_s \varphi (s, u_s) |] \dd s,
  \end{align*}
  and
  \[ \E [| \partial_s \varphi (s, u_{[s]^n}) - \partial_s \varphi (s,
     u_s) |] \leqslant \E [| \partial_s \varphi (s, u_{[s]^n}) |]
     +\E [| \partial_s \varphi (s, u_s) |] \lesssim \| \partial_s
     \varphi (s) \| \]
  is bounded in $[0, t]$. Moreover, by approximating $\partial_s \varphi (s)$
  in $\Gamma L^2$ with continuous functions, we get $\lim_{n \rightarrow
  \infty} \E [| \partial_s \varphi (s, u_{[s]^n}) - \partial_s
  \varphi (s, u_s) |] = 0$ for all $s$, and therefore by dominated convergence
  \[ \lim_{n \rightarrow \infty} \E \left[ \left| \sum_{k = 0}^{n -
     1} \int_{t_k}^{t_{k + 1}} \partial_s \varphi (s, u_{t_{k + 1}}) \dd s
     - \int_0^t \partial_s \varphi (s, u_s) \dd s \right| \right] = 0. \]
  Since $\varphi \in C (\R_+, \CD(\LL))$ we know that
  $\LL \varphi \in C (\R_+, \Gamma L^2)$ and thus, using
  once more the incompressibility,
  \[ \lim_{n \rightarrow \infty} \E \left[ \left| \sum_{k = 0}^{n -
     1} \int_{t_k}^{t_{k + 1}} \LL \varphi (t_k, u_s) \dd s -
     \int_0^t \LL \varphi (s, u_s) \dd s \right| \right] = 0. \]
  The convergence of the Lebesgue integrals is in $L^1$, and therefore the
  martingale property is inherited in the limit:
  \begin{align*}
    0 & = \lim_{n \rightarrow \infty} \E \left[ \varphi (t, u_t) -
    \varphi (0, u_0) - \sum_{k = 0}^{n - 1} \left[ \int_{t_k}^{t_{k + 1}}
    \partial_s \varphi (s, u_{t_{k + 1}}) \dd s + \int_{t_k}^{t_{k + 1}}
    \LL \varphi (t_k, u_s) \dd s \right] \right]\\
    & =\E \left[ \varphi (t, u_t) - \varphi (0, u_0) - \int_0^t
    [\partial_s \varphi (s, u_s) +\LL \varphi (s, u_s)] \dd s
    \right],
  \end{align*}
  and similarly for the conditional expectations.
\end{proof}

\bibliography{all}

\begin{thebibliography}{DPFRV16}

\bibitem[Ass02]{Assing2002}
Sigurd Assing.
\newblock A pregenerator for {B}urgers equation forced by conservative noise.
\newblock {\em Comm. Math. Phys.}, 225(3):611--632, 2002.

\bibitem[BCCH17]{Bruned2017Renormalising}
Yvain Bruned, Ajay Chandra, Ilya Chevyrev, and Martin Hairer.
\newblock Renormalising {SPDE}s in regularity structures.
\newblock {\em arXiv preprint arXiv:1711.10239}, 2017.

\bibitem[BCD11]{Bahouri2011}
Hajer Bahouri, Jean-Yves Chemin, and Raphael Danchin.
\newblock {\em {Fourier analysis and nonlinear partial differential
  equations}}.
\newblock Springer, 2011.

\bibitem[BHZ16]{Bruned2016}
Yvain Bruned, Martin Hairer, and Lorenzo Zambotti.
\newblock Algebraic renormalisation of regularity structures.
\newblock {\em arXiv preprint arXiv:1610.08468}, 2016.

\bibitem[CC18]{Cannizzaro2018}
Giuseppe Cannizzaro and Khalil Chouk.
\newblock Multidimensional {SDE}s with singular drift and universal
  construction of the polymer measure with white noise potential.
\newblock {\em Ann. Probab.}, 46(3):1710--1763, 2018.

\bibitem[CH16]{Chandra2016}
Ajay Chandra and Martin Hairer.
\newblock An analytic {BPHZ} theorem for regularity structures.
\newblock {\em arXiv preprint arXiv:1612.08138}, 2016.

\bibitem[Cor12]{Corwin2012}
Ivan Corwin.
\newblock The {K}ardar-{P}arisi-{Z}hang equation and universality class.
\newblock {\em Random Matrices Theory Appl.}, 1(1):1130001, 76, 2012.

\bibitem[DD16]{Delarue2016}
Fran{\c{c}}ois Delarue and Roland Diel.
\newblock Rough paths and 1d {SDE} with a time dependent distributional drift:
  application to polymers.
\newblock {\em Probab. Theory Related Fields}, 165(1-2):1--63, 2016.

\bibitem[DGP17]{Diehl2017}
Joscha Diehl, Massimiliano Gubinelli, and Nicolas Perkowski.
\newblock The {K}ardar-{P}arisi-{Z}hang equation as scaling limit of weakly
  asymmetric interacting {B}rownian motions.
\newblock {\em Comm. Math. Phys.}, 354(2):549--589, 2017.

\bibitem[DPD02]{DaPrato2002}
Giuseppe Da~Prato and Arnaud Debussche.
\newblock Two-dimensional {N}avier-{S}tokes equations driven by a space-time
  white noise.
\newblock {\em J. Funct. Anal.}, 196(1):180--210, 2002.

\bibitem[DPD03]{DaPrato2003}
Giuseppe Da~Prato and Arnaud Debussche.
\newblock Strong solutions to the stochastic quantization equations.
\newblock {\em Ann. Probab.}, 31(4):1900--1916, 2003.

\bibitem[DPFPR13]{DaPrato2013}
G.~Da~Prato, F.~Flandoli, E.~Priola, and M.~R\"{o}ckner.
\newblock Strong uniqueness for stochastic evolution equations in {H}ilbert
  spaces perturbed by a bounded measurable drift.
\newblock {\em Ann. Probab.}, 41(5):3306--3344, 2013.

\bibitem[DPFRV16]{DaPrato2016}
G.~Da~Prato, F.~Flandoli, M.~R\"{o}ckner, and A.~Yu. Veretennikov.
\newblock Strong uniqueness for {SDE}s in {H}ilbert spaces with nonregular
  drift.
\newblock {\em Ann. Probab.}, 44(3):1985--2023, 2016.

\bibitem[DPZ14]{DaPrato2014}
Giuseppe Da~Prato and Jerzy Zabczyk.
\newblock {\em Stochastic equations in infinite dimensions}, volume 152 of {\em
  Encyclopedia of Mathematics and its Applications}.
\newblock Cambridge University Press, Cambridge, second edition, 2014.

\bibitem[Ebe17]{Eberle2017}
Andreas Eberle.
\newblock Markov processes.
\newblock {\em Lecture Notes at University of Bonn}, 2017.

\bibitem[EK86]{Ethier1986}
Stewart~N. Ethier and Thomas~G. Kurtz.
\newblock {\em {Markov processes: Characterization and convergence}}.
\newblock John Wiley \& Sons, 1986.

\bibitem[FGS16]{Franco2016}
Tertuliano Franco, Patr{\'i}cia Gon{\c{c}}alves, and Marielle Simon.
\newblock Crossover to the stochastic {B}urgers equation for the {WASEP} with a
  slow bond.
\newblock {\em Comm. Math. Phys.}, 346(3):801--838, Sep 2016.

\bibitem[FH14]{Friz2014}
Peter~K. Friz and Martin Hairer.
\newblock {\em A course on rough paths}.
\newblock Universitext. Springer, Cham, 2014.
\newblock With an introduction to regularity structures.

\bibitem[FH17]{Funaki2017}
Tadahisa Funaki and Masato Hoshino.
\newblock A coupled {KPZ} equation, its two types of approximations and
  existence of global solutions.
\newblock {\em J. Funct. Anal.}, 273(3):1165--1204, 2017.

\bibitem[FL18a]{Flandoli2018Convergence}
Franco Flandoli and Dejun Luo.
\newblock Convergence of transport noise to {O}rnstein-{U}hlenbeck for {2D}
  {E}uler equations under the enstrophy measure.
\newblock {\em arXiv preprint arXiv:1806.09332}, 2018.

\bibitem[FL18b]{Flandoli2018Kolmogorov}
Franco Flandoli and Dejun Luo.
\newblock Kolmogorov equations associated to the stochastic {2D} {E}uler
  equations.
\newblock {\em arXiv preprint arXiv:1803.05654}, 2018.

\bibitem[FRW03]{Flandoli2003}
Franco Flandoli, Francesco Russo, and Jochen Wolf.
\newblock Some {SDE}s with distributional drift. {I}. {G}eneral calculus.
\newblock {\em Osaka J. Math.}, 40(2):493--542, 2003.

\bibitem[FRW04]{Flandoli2004}
Franco Flandoli, Francesco Russo, and Jochen Wolf.
\newblock Some {SDE}s with distributional drift. {II}. {L}yons-{Z}heng
  structure, {I}t\^{o}'s formula and semimartingale characterization.
\newblock {\em Random Oper. Stochastic Equations}, 12(2):145--184, 2004.

\bibitem[GIP15]{Gubinelli2015Paracontrolled}
Massimiliano Gubinelli, Peter Imkeller, and Nicolas Perkowski.
\newblock Paracontrolled distributions and singular {PDE}s.
\newblock {\em Forum of Mathematics, Pi}, 3(e6), 2015.

\bibitem[GJ13]{Gubinelli2013}
Massimiliano Gubinelli and Milton Jara.
\newblock Regularization by noise and stochastic {B}urgers equations.
\newblock {\em Stochastic Partial Differential Equations: Analysis and
  Computations}, 1(2):325--350, 2013.

\bibitem[GJ14]{Goncalves2014}
Patr{\'i}cia Gon{\c{c}}alves and Milton Jara.
\newblock Nonlinear fluctuations of weakly asymmetric interacting particle
  systems.
\newblock {\em Arch. Ration. Mech. Anal.}, 212(2):597--644, 2014.

\bibitem[GJ18]{Goncalves2018}
Patr{\'i}cia Gon{\c{c}}alves and Milton Jara.
\newblock Density fluctuations for exclusion processes with long jumps.
\newblock {\em Probab. Theory Related Fields}, 170(1-2):311--362, 2018.

\bibitem[GJS15]{Goncalves2015}
Patr{\'\i}cia Gon{\c{c}}alves, Milton Jara, and Sunder Sethuraman.
\newblock A stochastic {B}urgers equation from a class of microscopic
  interactions.
\newblock {\em Ann. Probab.}, 43(1):286--338, 2015.

\bibitem[GP15]{Gubinelli2015EBP}
Massimiliano Gubinelli and Nicolas Perkowski.
\newblock Lectures on singular stochastic {PDE}s.
\newblock {\em Ensaios Mat.}, 29, 2015.

\bibitem[GP16]{Gubinelli2016Hairer}
Massimiliano Gubinelli and Nicolas Perkowski.
\newblock The {H}airer--{Q}uastel universality result at stationarity.
\newblock {\em RIMS K{\^o}ky{\^u}roku Bessatsu}, B59, 2016.

\bibitem[GP17]{Gubinelli2017KPZ}
Massimiliano Gubinelli and Nicolas Perkowski.
\newblock {KPZ} reloaded.
\newblock {\em Comm. Math. Phys.}, 349(1):165--269, 2017.

\bibitem[GP18a]{Gubinelli2018Energy}
Massimiliano Gubinelli and Nicolas Perkowski.
\newblock Energy solutions of {KPZ} are unique.
\newblock {\em J. Amer. Math. Soc.}, 31(2):427--471, 2018.

\bibitem[GP18b]{Gubinelli2018Probabilistic}
Massimiliano Gubinelli and Nicolas Perkowski.
\newblock Probabilistic approach to the stochastic {B}urgers equation.
\newblock In {\em Stochastic Partial Differential Equations and Related Fields.
  In Honor of Michael R{\"o}ckner}, pages 512--527, 2018.

\bibitem[GPS17]{Goncalves2017}
Patr{\'\i}cia Gon{\c{c}}alves, Nicolas Perkowski, and Marielle Simon.
\newblock Derivation of the stochastic {B}urgers equation with {D}irichlet
  boundary conditions from the {WASEP}.
\newblock {\em arXiv preprint arXiv:1710.11011}, 2017.

\bibitem[Gub04]{Gubinelli2004}
Massimiliano Gubinelli.
\newblock {Controlling rough paths}.
\newblock {\em J. Funct. Anal.}, 216(1):86--140, nov 2004.

\bibitem[Gub18]{Gubinelli2018Panorama}
Massimiliano Gubinelli.
\newblock A panorama of singular {SPDE}s.
\newblock In {\em Proc. Int. Cong. of Math}, volume~2, pages 2277--2304, 2018.

\bibitem[GZ03]{Guionnet2003}
A.~Guionnet and B.~Zegarlinski.
\newblock Lectures on logarithmic {S}obolev inequalities.
\newblock In {\em S\'{e}minaire de {P}robabilit\'{e}s, {XXXVI}}, volume 1801 of
  {\em Lecture Notes in Math.}, pages 1--134. Springer, Berlin, 2003.

\bibitem[Hai14]{Hairer2014}
Martin Hairer.
\newblock A theory of regularity structures.
\newblock {\em Invent. Math.}, 198(2):269--504, 2014.

\bibitem[HM18]{Hairer2018Strong}
Martin Hairer and Jonathan Mattingly.
\newblock The strong {F}eller property for singular stochastic {PDE}s.
\newblock {\em Ann. Inst. Henri Poincar\'{e} Probab. Stat.}, 54(3):1314--1340,
  2018.

\bibitem[Jan97]{Janson1997}
Svante Janson.
\newblock {\em {Gaussian {H}ilbert spaces}}, volume 129 of {\em {Cambridge
  Tracts in Mathematics}}.
\newblock Cambridge University Press, Cambridge, 1997.

\bibitem[KLO12]{Komorowski2012}
Tomasz Komorowski, Claudio Landim, and Stefano Olla.
\newblock {\em Fluctuations in {M}arkov processes}, volume 345 of {\em
  Grundlehren der Mathematischen Wissenschaften [Fundamental Principles of
  Mathematical Sciences]}.
\newblock Springer, Heidelberg, 2012.
\newblock Time symmetry and martingale approximation.

\bibitem[KM17]{Kupiainen2017}
Antti Kupiainen and Matteo Marcozzi.
\newblock Renormalization of generalized {KPZ} equation.
\newblock {\em J. Stat. Phys.}, 166(3-4):876--902, 2017.

\bibitem[LCL07]{Lyons2007}
Terry~J. Lyons, Michael Caruana, and Thierry L{\'e}vy.
\newblock {\em {Differential equations driven by rough paths}}, volume 1908 of
  {\em {Lecture Notes in Mathematics}}.
\newblock Springer, Berlin, 2007.

\bibitem[LR15]{Liu2015}
Wei Liu and Michael R\"ockner.
\newblock {\em Stochastic partial differential equations: an introduction}.
\newblock Universitext. Springer, Cham, 2015.

\bibitem[Lyo98]{Lyons1998}
Terry~J. Lyons.
\newblock {Differential equations driven by rough signals}.
\newblock {\em Rev. Mat. Iberoam.}, 14(2):215--310, 1998.

\bibitem[Mit83]{Mitoma1983}
Itaru Mitoma.
\newblock Tightness of probabilities on {$C([0,1];\mathcal{S}')$} and
  {$D([0,1];\mathcal{S}')$}.
\newblock {\em Ann. Probab.}, 11(4):989--999, 1983.

\bibitem[Nua06]{Nualart2006}
David Nualart.
\newblock {\em The {M}alliavin calculus and related topics}.
\newblock Probability and its Applications (New York). Springer-Verlag, Berlin,
  second edition, 2006.

\bibitem[QS15]{Quastel2015}
Jeremy Quastel and Herbert Spohn.
\newblock The one-dimensional {KPZ} equation and its universality class.
\newblock {\em J. Stat. Phys.}, 160(4):965--984, 2015.

\bibitem[Qua12]{Quastel2011}
Jeremy Quastel.
\newblock Introduction to {KPZ}.
\newblock In {\em Current developments in mathematics, 2011}, pages 125--194.
  Int. Press, Somerville, MA, 2012.

\bibitem[RZZ17]{Rockner2017}
Michael R{\"o}ckner, Rongchan Zhu, and Xiangchan Zhu.
\newblock Restricted {M}arkov uniqueness for the stochastic quantization of
  {$P(\Phi)_2$} and its applications.
\newblock {\em J. Funct. Anal.}, 272(10):4263--4303, 2017.

\bibitem[Sta07]{Stannat2007}
Wilhelm Stannat.
\newblock A new a priori estimate for the {K}olmogorov operator of a
  2{D}-stochastic {N}avier-{S}tokes equation.
\newblock {\em Infin. Dimens. Anal. Quantum Probab. Relat. Top.},
  10(4):483--497, 2007.

\bibitem[Wal86]{Walsh1986}
John~B. Walsh.
\newblock An introduction to stochastic partial differential equations.
\newblock In {\em \'Ecole d'\'et\'e de probabilit\'es de {S}aint-{F}lour,
  {XIV}---1984}, volume 1180 of {\em Lecture Notes in Math.}, pages 265--439.
  Springer, Berlin, 1986.

\bibitem[Yan18]{Yang2018}
Kevin Yang.
\newblock The {KPZ} equation, non-equilibrium energy solutions, and weak
  universality for long-range interactions.
\newblock {\em arXiv preprint arXiv:1810.02836}, 2018.

\bibitem[ZZ17]{Zhu2017}
Rongchan Zhu and Xiangchan Zhu.
\newblock Dirichlet form associated with the $\phi^4_3$ model.
\newblock {\em arXiv preprint arXiv:1703.09987}, 2017.

\end{thebibliography}
\bibliographystyle{alpha}

\end{document}